\RequirePackage{fix-cm}
\documentclass[english]{article}
\usepackage[T1]{fontenc}
\usepackage[latin9]{inputenc}
\usepackage{listings}
\usepackage[a4paper]{geometry}
\usepackage{color}
\usepackage{babel}
\usepackage{float}
\usepackage{textcomp}
\usepackage{url}
\usepackage{amsthm}
\usepackage{amsmath}
\usepackage{amssymb}
\usepackage{fixltx2e}
\usepackage{graphicx}
\usepackage{esint}
\usepackage[unicode=true,pdfusetitle,
 bookmarks=true,bookmarksnumbered=true,bookmarksopen=true,bookmarksopenlevel=1,
 breaklinks=true,pdfborder={0 0 1},backref=false,colorlinks=true]
 {hyperref}
\hypersetup{
 linkcolor=black, citecolor=blue, urlcolor=blue, filecolor=blue,pdfpagelayout=OneColumn, pdfnewwindow=true,pdfstartview=XYZ, plainpages=false, pdfpagelabels, hyperindex=true}
\usepackage{breakurl}

\makeatletter

\newcommand{\lyxmathsym}[1]{\ifmmode\begingroup\def\b@ld{bold}
  \text{\ifx\math@version\b@ld\bfseries\fi#1}\endgroup\else#1\fi}

\providecommand{\tabularnewline}{\\}
\floatstyle{ruled}
\newfloat{algorithm}{tbp}{loa}
\providecommand{\algorithmname}{Algorithm}
\floatname{algorithm}{\protect\algorithmname}

\usepackage{enumitem}		
\numberwithin{equation}{section}
\numberwithin{figure}{section}
\numberwithin{table}{section}
  \theoremstyle{plain}
  \newtheorem{assumption}{\protect\assumptionname}
  \theoremstyle{remark}
  \newtheorem{rem}{\protect\remarkname}[section]
  \theoremstyle{definition}
  \newtheorem{example}{\protect\examplename}[section]
  \theoremstyle{plain}
  \newtheorem{thm}{\protect\theoremname}[section]
  \theoremstyle{plain}
  \newtheorem{lem}{\protect\lemmaname}[section]
  \theoremstyle{plain}
  \newtheorem{prop}{\protect\propositionname}[section]
  \theoremstyle{plain}
  \newtheorem{cor}{\protect\corollaryname}[section]

\usepackage{fancyhdr}
\usepackage{lastpage}
\usepackage{amsmath}
\usepackage{amssymb} 
\usepackage{graphicx}
\usepackage{color}
\usepackage{fancybox} 
\usepackage{moreverb} 
\usepackage{hangcaption} 
\usepackage{listings} 
\usepackage{tocbibind}

\usepackage{ifpdf} 
\ifpdf 

 \IfFileExists{lmodern.sty}{\usepackage{lmodern}}{}

\fi 

\let\myTOC\tableofcontents
\renewcommand\tableofcontents{%
  \frontmatter
  \pdfbookmark[1]{\contentsname}{}
  \myTOC
  \mainmatter }

\def\LyX{\texorpdfstring{%
  L\kern-.1667em\lower.25em\hbox{Y}\kern-.125emX\@}
  {LyX}}

\@addtoreset{footnote}{section}

\usepackage{eurosym}

\usepackage{enumitem}
\setlist{leftmargin=*, topsep=0.5em, parsep=0pt, itemsep=1em, labelindent=0pt, align=left}

\@ifundefined{showcaptionsetup}{}{%
 \PassOptionsToPackage{caption=false}{subfig}}
\usepackage{subfig}
\makeatother

  \providecommand{\assumptionname}{Assumption}
  \providecommand{\examplename}{Example}
  \providecommand{\lemmaname}{Lemma}
  \providecommand{\propositionname}{Proposition}
  \providecommand{\remarkname}{Remark}
\providecommand{\corollaryname}{Corollary}
\providecommand{\theoremname}{Theorem}

\begin{document}

\title{A probabilistic numerical method\\
for optimal multiple switching problem\\
and application to investments in electricity generation%
\thanks{The authors would like to thank Thomas Vareschi, Xavier Warin and
the participants of the FiME seminar and the Energy Finance conference
for their helpful remarks.%
}}

\author{René Aïd%
\thanks{EDF R\&D and FiME (Finance for Energy Market Research Centre, www.fime-lab.org)
; rene.aid@edf.fr%
}\quad{}, Luciano Campi%
\thanks{Univ Paris Nord, Sorbonne Paris Cité, LAGA (Laboratoire Analyse, Géométrie
et Applications), FiME and CREST ; campi@math.univ-paris13.fr%
}\quad{}, Nicolas Langrené%
\thanks{Univ Paris Diderot, Sorbonne Paris Cité, LPMA (Laboratoire de Probabilités
et Modèles Aléatoires, CNRS, UMR 7599) ; nicolas.langrene@paris7.jussieu.fr%
}\quad{}, Huyên Pham%
\thanks{Univ Paris Diderot, Sorbonne Paris Cité, LPMA and CREST-ENSAE, pham@math.univ-paris-diderot.fr%
}}
\maketitle
\begin{abstract}
In this paper, we present a probabilistic numerical algorithm combining
dynamic programming, Monte Carlo simulations and local basis regressions
to solve non-stationary optimal multiple switching problems in infinite
horizon. We provide the rate of convergence of the method in terms
of the time step used to discretize the problem, of the size of the
local hypercubes involved in the regressions, and of the truncating
time horizon. To make the method viable for problems in high dimension
and long time horizon, we extend a memory reduction method to the
general Euler scheme, so that, when performing the numerical resolution,
the storage of the Monte Carlo simulation paths is not needed. Then,
we apply this algorithm to a model of optimal investment in power
plants. This model takes into account electricity demand, cointegrated
fuel prices, carbon price and random outages of power plants. It computes
the optimal level of investment in each generation technology, considered
as a whole, w.r.t. the electricity spot price. This electricity price
is itself built according to a new extended structural model. In particular,
it is a function of several factors, among which the installed capacities.
The evolution of the optimal generation mix is illustrated on a realistic
numerical problem in dimension eight, i.e. with two different technologies
and six random factors. 
\end{abstract}

\section{Introduction}

This paper presents a probabilistic numerical method for multiple
switching problem with an application to a new stylized long-term
investment model for electricity generation. Since electricity cannot
be stored and building new plants takes several years, investment
in new capacities must be decided a long time in advance if a country
wishes to be able to satisfy its demand%
\footnote{See for instance the recent massive power system blackouts in Brazil
and Paraguay on November $10^{\mathrm{th}}$ and $11^{\mathrm{th}}$,
2009 (87 millions of people affected), and in India on July $30^{\mathrm{th}}$
and $31^{\mathrm{st}}$, 2012. (670 millions of people affected).%
}. Before the worldwide liberalization of the electricity sector, electric
utilities were monopolies whose objective was to plan the construction
of power plants in order to satisfy demand at the minimum cost under
a given constraint on the loss of load probability or on the level
of energy non-served. This investment process was called \textit{generation
expansion planning} (GEP). Its output was mostly a given set of power
plants to build for the next ten or twenty years (see \cite{IAEA84}
for a comprehensive description of the GEP methodology and related
difficulties). Despite thirty years of liberalization of the electricity
sector, of the recognition that GEP methods were inadequate within
a market context (\cite{Hobbs95,Dyner01}) and of an important set
of alternative methods (see \cite{Foley10} and \cite{Connolly10}
for recent surveys on generation investment models and softwares),
power utilities still heavily rely on GEP methods (see \cite{Aid10}).
However, real option methods, which should have been the natural alternative
valuation method for firms converted to a value maximizing objective,
did not emerge as the method of choice. Despite the important body
of literature that followed \cite{McDonald86}'s seminal paper, \cite{Dixit94}'s
monography and implementations for electricity generation investment
(see for instance \cite{Botterud05}), real options still remain a
marginal way of assessing investment decisions both in the electric
sector and in the industry in general (see for instance, amongst the
recurrent surveys on capital budgeting methods, \cite{Baker12}).
Nevertheless, as shown in \cite{McDonald98}, firms tend to reproduce
with heuristic constraints (such as hurdle rate or profitability index)
the decision criteria given by real option methodology.

The main reason for this situation lies in the considerable mathematical
difficulties involved in the conception of a tractable yet realistic
real option model for electricity generation. This difficulty reflects
in the literature where the main trend consists in designing a small
dimensional (1 or 2) real option model to assess investment behaviour
with respect to some specific variables (see for instance \cite{Aguerrevere03,BarIlan02}
for models in dimension 2 analysing the effects of uncertainty and
time to build). It is still possible to find investment model in dimension
3 based on dynamic programming which are numerically tractable (see
for instance \cite{Mo91,Botterud05} and Section \ref{sec:Application}
for comments). But, in higher dimension, because of the curse of dimensionality,
investment models mainly rely on decision trees to represent random
factors (see \cite{Ahmed03} for a recent typical implementation of
this approach). The resulting tractability is however obtained at
the expense of a crude simplification of the statistical properties
of the factors. 

Our approach in the present paper takes advantage of the considerable
progress made in the last ten years by numerical methods for high-dimensional
American options valuation problems to propose a probabilistic way
to look at future electricity generation mixes. For an up-to-date
state of the art on this subject, the reader is referred to the recent
book \cite{Carmona12-2}. 

In this paper, we first adapt the resolution of American option problems
by Monte-Carlo methods (\cite{Longstaff01,Tsitsiklis01}) to the more
general class of optimal switching problems. The crucial choice of
regression basis is done here in the light of the work of \cite{Bouchard11},
so as to obtain a stable algorithm suited to high-dimensional problems,
aiming at the best possible numerical complexity. The memory complexity,
often acknowledged as the major drawback of such a Monte Carlo approach
(see \cite{Carmona08}), is drastically slashed by generalizing the
memory reduction method from \cite{Chan04,Chan06,Chan11} to any stochastic
differential equation. We provide a rigorous and comprehensive analysis
of the rate of convergence of our algorithm, taking advantage of the
works of, most notably, \cite{Bouchard04}, \cite{Tan11} and \cite{Gassiat11}.
Note that such features as infinite horizon and non-stationarity are
encompassed here. Finally, we build a long-term structural model for
the spot price of electricity, extending the work of \cite{Aid09}
and \cite{Aid12} in several directions (cointegrated fuels and $\mathrm{CO_{2}}$
prices, stochastic availability rate of production capacities, etc.).
This model is itself incorporated into an optimal control problem
corresponding to the search for the optimal investments in electricity
generation. The resolution of this problem using our algorithm is
illustrated on a simple numerical example with two different technologies,
leading to an eight-dimensional problem (demand, $\mathrm{CO_{2}}$
price, and, for each technology, fuel price, random outages and the
controlled installed capacity). The time evolution of the distribution
of power prices and of the generation mix is illustrated on a forty-year
time horizon.

To sum up, the contribution of the paper is threefold. Firstly, it
provides, for a suitably chosen regression basis, a comprehensive
analysis of convergence of a regression-based Monte-Carlo algorithm
for a class of infinite horizon optimal multiple switching problems,
large enough to handle realistic short term profit functions and investment
cost structures with possible seasonality patterns. Secondly, it adapts
and generalizes a memory reduction method in order to slash the amount
of memory required by the Monte Carlo algorithm. Thirdly, a new stylised
investment model for electricity generation is proposed, taking into
account electricity demand, cointegrated fuel prices, carbon price
and random outages of power plants, used as building blocks of a new
structural model for the electricity spot price. A numerical resolution
of this investment problem with our algorithm is illustrated on a
specific example, providing, among many other outputs, an electricity
spot price dynamics consistent with the investment decision process
in power generation.

The outline of the paper is the following. Section \ref{sec:Optimal-switching}
describes the class of optimal switching problems studied here, including
the detailed list of assumptions considered. Section \ref{sub:Convergence}
describes the resolution algorithm and analyzes its rate of convergence,
in terms of the discretization step, of the size of the local hypercubes
from the regression basis, and of the truncating time horizon. Section
\ref{sec:Complexity} details the computational complexity of the
algorithm, as well as its memory complexity, along with the construction
of the memory reduction method. Finally, Section \ref{sec:Application}
introduces the extended structural model of power spot price, the
investment problem, as well as an illustrated numerical resolution.
Section \ref{sec:Conclusion} concludes the paper.

\subsection*{Notation\label{sub:Notation}}

Here are some notation that will be used throughout the paper:
\begin{itemize}
\item The notation $\mathbf{1}\left\{ .\right\} $ stands for the indicator
function.
\item \label{not:SuperConstant}Throughout the paper, $C>0$ denotes a generic
constant whose value may differ from line to line, but which does
not depend on any parameter of our scheme.
\item For any stochastic process $X=\left(X_{s}\right)_{s\geq0}$ taking
values in a given set $\mathcal{X}$, and any $\left(t,x\right)\in\mathbb{R}_{+}\times\mathcal{X}$,
we denote as $X^{t,x}=\left(X_{s}^{t,x}\right)_{s\geq t}$ the stochastic
process with the same dynamics as $X$, but starting from $x$ at
time $t$: $X_{t}^{t,x}=x$. 
\item For any $\left(a,b\right)\in\mathbb{R}\times\mathbb{R}$, $a\wedge b:=\min\left(a,b\right)$
and $a\vee b:=\max\left(a,b\right)$.
\item $\forall p\geq1$, the norms $\left\Vert .\right\Vert _{p}$ and $\left\Vert .\right\Vert _{L_{p}}$
denote respectively the $p-$norm and the $L_{p}$- norm: $\forall x\in\mathbb{R}^{n}$
and any $\mathbb{R}$-valued random variable $X$ such that $\mathbb{E}\left[\left|X\right|^{p}\right]<\infty$:
\[
\begin{array}{rcl}
\left\Vert x\right\Vert _{p}=\left(\sum_{i=1}^{n}\left|x_{i}\right|^{p}\right)^{\frac{1}{p}} & \,,\, & \left\Vert X\right\Vert _{L_{p}}=\mathbb{E}\left[\left|X\right|^{p}\right]^{\frac{1}{p}}\end{array}
\]
We recall that $\forall p\geq1$, $\forall x\in\mathbb{R}^{n}$, $\left\Vert x\right\Vert _{p}\leq\left\Vert x\right\Vert _{1}\leq n^{\frac{p-1}{p}}\left\Vert x\right\Vert _{p}$ 
\end{itemize}

\section{Optimal switching problem\label{sec:Optimal-switching}}

\subsection{Formulation\label{sub:Formulation}}

Fix a filtered probability space $\left(\Omega,\mathcal{F},\mathbb{F}=\left(\mathcal{F}_{t}\right)_{t\geq0},\mathbb{P}\right)$,
where $\mathbb{F}$ satisfies the usual conditions of right-continuity
and $\mathbb{P}$-completeness. We consider the following general
class of (non-stationary) optimal switching problems:
\begin{equation}
v\left(t,x,i\right)=\sup_{\alpha\in\mathcal{A}_{t,i}}\mathbb{E}\left[\int_{t}^{\infty}f\left(s,X_{s}^{t,x},I_{s}^{\alpha}\right)ds-\sum_{\tau_{n}\geq t}k\left(\tau_{n},\zeta_{n}\right)\right]\label{eq:v}
\end{equation}
where:
\begin{itemize}
\item $X^{t,x}=\left(X_{s}^{t,x}\right)_{s\geq t}$ is an $\mathbb{R}^{d}$-valued,
$\mathbb{F}$-adapted markovian diffusion starting from $X_{t}=x\in\mathbb{R}^{d}$,
with generator $\mathcal{L}$.
\item $I^{\alpha}=\left(I_{s}^{\alpha}\right)_{s\geq0}$ is a càd-làg, $\mathbb{R}^{d'}$-valued,
$\mathbb{F}$-adapted piecewise constant process. It is controlled
by a strategy $\alpha$, described below. We suppose it can only take
values into a fixed finite set $\mathbb{I}_{q}=\left\{ i_{1},i_{2},\ldots,i_{q}\right\} $,
$q\in\mathbb{N}^{*}$ with $i_{0}=0$ $\left(\in\mathbb{R}^{d'}\right)$,
which means that equation \eqref{eq:v} corresponds to an optimal
switching problem. 
\item An impulse control strategy $\alpha$ corresponds to a sequence $\left(\tau_{n},\iota_{n}\right)_{n\in\mathbb{N}}$
of increasing stopping times $\tau_{n}\geq0$, and $\mathcal{F}_{\tau_{n}}$-measurable
random variables $\iota_{n}$ valued in $\mathbb{I}_{q}$. Using this
sequence, $I^{\alpha}=\left(I_{s}^{\alpha}\right)_{s\geq0}$ is defined
as follows:
\[
I_{s}^{\alpha}=\iota_{0}\mathbf{1}\left\{ 0\leq s<\tau_{0}\right\} +\sum_{n\in\mathbb{N}}\iota_{n}\mathbf{1}\left\{ \tau_{n}\leq s<\tau_{n+1}\right\} \in\mathbb{I}_{q}
\]
Alternatively, $\alpha$ can be described by the sequence $\left(\tau_{n},\zeta_{n}\right)_{n\in\mathbb{N}}$,
where $\zeta_{n}:=\iota_{n}-\iota_{n-1}$ (and $\zeta_{0}:=0$). Using
this alternative sequence, $I^{\alpha}$ can be written as follows:
\[
I_{s}^{\alpha}=\iota_{0}+\sum_{\tau_{n}\leq s}\zeta_{n}\in\mathbb{I}_{q}
\]

\item $\mathcal{A}$ is the set of admissible strategies: a strategy $\alpha$
belongs to $\mathcal{A}$ if $\tau_{n}\rightarrow+\infty$ a.s. as
$n\rightarrow\infty$.
\item For any $\left(t,i\right)\in\mathbb{R}_{+}\times\mathbb{I}_{q}$,
the set $\mathcal{A}_{t,i}\subset\mathcal{A}$ is defined as the subset
of admissible strategies $\alpha$ such that $I_{t}^{\alpha}=i$.
\item $f$ and $k$ are $\mathbb{R}$-valued measurable functions.
\end{itemize}

\subsection{Assumptions\label{sub:Assumptions}}

We complete the above formulation with the following relevant assumptions.
\begin{assumption}
\label{asm:diffusion}{[}Diffusion{]} The $\mathbb{R}^{d}$-valued
uncontrolled process $X$ is a diffusion process, governed by the
dynamics
\begin{eqnarray}
dX_{s} & = & b\left(s,X_{s}\right)ds+\sigma\left(s,X_{s}\right)dW_{s}\label{eq:SDE-X}\\
X_{0} & = & x_{0}\in\mathbb{R}^{d}\nonumber 
\end{eqnarray}
where $W$ is a $d$-dimensional Brownian motion, and $b$ and $\sigma$
are respectively $\mathbb{R}^{d}$-valued and $\mathbb{R}^{d\times d}$-valued
functions.
\end{assumption}
\smallskip{}

\begin{assumption}
\label{asm:SDE-ExistenceUnicity}{[}Lipschitz{]} The functions $b:\mathbb{R}_{+}\times\mathbb{R}^{d}\rightarrow\mathbb{R}^{d}$
and $\sigma:\mathbb{R}_{+}\times\mathbb{R}^{d}\rightarrow\mathbb{R}^{d\times d}$
are Lipschitz-continuous (uniformly in $t$) with linear growth: $\exists C_{b},C_{\sigma}>0$
s.t. $\forall t\in\mathbb{R}_{+}$, $\forall\left(x,x'\right)\in\left(\mathbb{R}^{d}\right)^{2}$:
\begin{eqnarray*}
\left|b\left(t,x\right)-b\left(t,x'\right)\right| & \leq & C_{b}\left|x-x'\right|\\
\left|b\left(t,x\right)\right| & \leq & C_{b}\left(1+\left|x\right|\right)\\
\left|\sigma\left(t,x\right)-\sigma\left(t,x'\right)\right| & \leq & C_{\sigma}\left|x-x'\right|\\
\left|\sigma\left(t,x\right)\right| & \leq & C_{\sigma}\left(1+\left|x\right|\right)
\end{eqnarray*}

\end{assumption}

\begin{rem}
Assumption \ref{asm:SDE-ExistenceUnicity} is sufficient to prove
the existence and uniqueness of a strong solution to the SDE \eqref{eq:SDE-X}
(see for instance Theorem 4.5.3 in \cite{Kloeden99}).
\end{rem}
\smallskip{}

\begin{rem}
Under Assumption \ref{asm:SDE-ExistenceUnicity}, there exist, for
every $p\geq1$, positive constants $C_{p}$ and $\rho_{p}$ such
that $\forall s\geq t\geq0$ and $\forall x\in\mathbb{R}^{d}$:
\begin{equation}
\mathbb{E}\left[\left|X_{s}^{t,x}\right|^{p}\right]\leq C_{p}\left(1+\left|x\right|^{p}\right)\exp\left(\rho_{p}\left(s-t\right)\right)\label{eq:momentBound}
\end{equation}
(use Burkholder-Davis-Gundy inequality and Gronwall's Lemma, see for
instance \cite{Kloeden99} Theorem 4.5.4 for the even power case).

\smallskip{}

\end{rem}

\begin{assumption}
\label{asm:Lipschitz-f-g}{[}Lipschitz\textup{\&}Discount{]} The functions
$f$ and $k$ decrease exponentially in time: $\exists\rho>0$ s.t.
$\forall\left(t,x,i,j\right)\in\mathbb{R}_{+}\times\mathbb{R}^{d}\times\left(\mathbb{I}_{q}\right)^{2}$:
\begin{eqnarray*}
f\left(t,x,i\right) & = & e^{-\rho t}\tilde{f}\left(t,x,i\right)\\
k\left(t,j-i\right) & = & e^{-\rho t}\tilde{k}\left(t,j-i\right)
\end{eqnarray*}

where the functions $\tilde{f}$ and $\tilde{k}$ are Lipschitz continuous
with linear growth:

$\exists C_{f},C_{k}>0$ s.t. $\forall\left\{ \left(t,x,i,j\right),\left(t',x',i',j'\right)\right\} \in\left\{ \mathbb{R}_{+}\times\mathbb{R}^{d}\times\left(\mathbb{I}_{q}\right)^{2}\right\} ^{2}$:
\begin{eqnarray*}
\left|\tilde{f}\left(t,x,i\right)-\tilde{f}\left(t',x',i'\right)\right| & \leq & C_{f}\left(\left|t-t'\right|+\left|x-x'\right|+\left|i-i'\right|\right)\\
\left|\tilde{f}\left(t,x,i\right)\right| & \leq & C_{f}\left(1+\left|x\right|\right)\\
\left|\tilde{k}\left(t,j-i\right)-\tilde{k}\left(t',j'-i'\right)\right| & \leq & C_{k}\left(\left|t-t'\right|+\left|\left(j-i\right)-\left(j'-i'\right)\right|\right)
\end{eqnarray*}

Moreover, we assume in the following that $\rho>\rho_{1}$ where $\rho_{1}$
is defined in equation \eqref{eq:momentBound}.
\end{assumption}
\smallskip{}

\begin{assumption}
\label{asm:costs}{[}Fixed costs{]} The cost function $k:\mathbb{R}_{+}\times\mathbb{R}^{d'}\rightarrow\mathbb{R}_{+}$
is such that:
\begin{itemize}
\item $\forall t\in\mathbb{R}_{+}$, $k\left(t,0\right)=0$.
\item \textup{$\exists\kappa>0$ s.t. $\forall t\in\mathbb{R}_{+}$, $\forall\left(i,j\right)\in\left(\mathbb{I}_{q}\right)^{2}$},\textup{
}$\left\{ i\neq j\right\} \Rightarrow\left\{ \tilde{k}\!\left(t,j-i\right)\geq\kappa\right\} $.
\item (triangular inequality) $\forall t\in\mathbb{R}_{+}$, $\forall\left(i,j,k\right)\in\left(\mathbb{I}_{q}\right)^{3}$
with $i\neq j$ and $j\neq k$:\textup{\vspace{-0.3em}
}
\[
k\!\left(t,k-i\right)<k\!\left(t,j-i\right)+k\!\left(t,k-j\right)\,.
\]

\end{itemize}
\end{assumption}
\smallskip{}

\begin{rem}
The economic interpretations of Assumption \ref{asm:costs} are the
following:
\begin{enumerate}
\item There is no cost for not switching, but any switch incurs at least
a positive fixed cost.
\item At any given date, it is always cheaper to switch directly from $i$
to $k$ than to switch first from $i$ to $j$ and then from $j$
to $k$.
\end{enumerate}
\end{rem}
\smallskip{}

\begin{rem}
Under those standard assumptions, the value function $v$ is well-defined
and finite. Indeed, using equation \eqref{eq:momentBound}, $\forall\left(t_{0},t,x,i\right)\in\mathbb{R}_{+}\times\mathbb{R}_{+}\times\mathbb{R}^{d}\times\mathbb{R}^{d'}$
with $t_{0}\leq t$ and $\forall\alpha\in\mathcal{A}_{t_{0},i}$:
\begin{eqnarray}
\mathbb{E}\left[\int_{t}^{\infty}\left|f\left(s,X_{s}^{t_{0},x},I_{s}^{\alpha}\right)\right|ds\right] & \leq & C_{f}\int_{t}^{\infty}e^{-\rho s}\left(1+\mathbb{E}\left[\left|X_{s}^{t_{0},x}\right|\right]\right)ds\nonumber \\
 & \leq & C_{f}\left(e^{-\rho t}+\left(1+\left|x\right|\right)\int_{t}^{\infty}e^{-\rho s}e^{\rho_{1}\left(s-t_{0}\right)}ds\right)\nonumber \\
 & \leq & C_{f}\left(1+\left|x\right|\right)e^{-\bar{\rho}t-\rho_{1}t_{0}}\label{eq:ftbound}
\end{eqnarray}
where $\bar{\rho}:=\rho-\rho_{1}>0$ (Assumption \ref{asm:Lipschitz-f-g}).
In particular, the costs being positive (Assumption \ref{asm:costs}),
and recalling \eqref{eq:v}, it holds that:
\begin{equation}
v\left(t,x,i\right)\leq C_{f}\left(1+\left|x\right|\right)e^{-\rho t}\label{eq:vBound}
\end{equation}

\end{rem}

\subsection{Outline of the solution\label{sec:Outline}}

From a theoretical point of view, the value functions $v_{i}:=v\left(.,.,i\right)$,
$i\in\mathbb{I}_{q}$ from equation \eqref{eq:v} are known to satisfy
(under suitable conditions on $f_{i}\left(.,.\right):=f\left(.,.,i\right)$
and $k$, see for instance \cite{Seydel09-2} in a much more general
setting) the following Hamilton-Jacobi-Bellman Quasi-Variational Inequalities
(HJBQVI): $\forall\left(t,x,i\right)\in\mathbb{R}_{+}\times\mathbb{R}^{d}\times\mathbb{I}_{q}$
\begin{equation}
\min\left\{ -\frac{\partial v_{i}}{\partial t}\left(t,x\right)-\mathcal{L}v_{i}\left(t,x\right)-f_{i}\left(t,x\right)\,,\, v_{i}\left(t,x\right)-\max_{j\in\mathbb{I}_{P},\, j\neq i}\left(v_{j}\left(t,x\right)-k\left(t,j-i\right)\right)\right\} =0\label{eq:HJBQVI}
\end{equation}
together with suitable limit condition.

Alternatively, the process $v\left(t,X_{t},i\right)$, $t\geq0$ can
be characterized as the solution of a particular Reflected Backward
Stochastic Differential Equation (\cite{Hamadene99,ElAsri10}).

Moreover, the value function \eqref{eq:v} satisfies the well-known
dynamic programming principle, i.e., for any stopping time $\tau\geq t$:
\begin{equation}
v\left(t,x,i\right)=\sup_{\alpha\in\mathcal{A}_{t,i}}\mathbb{E}\left[\int_{t}^{\tau}f\left(s,X_{s}^{t,x},I_{s}^{\alpha}\right)ds-\sum_{t\leq\tau_{n}\leq\tau}k\left(\tau_{n},\zeta_{n}\right)+v\left(\tau,X_{\tau}^{t,x},I_{\tau}^{\alpha}\right)\right]\,.\label{eq:DPP}
\end{equation}
From a practical point of view, apart from a few simple examples in
low-dimension, finding directly the solution of the HJBQVI \eqref{eq:HJBQVI}
is usually infeasible, and the numerical PDE tools become cumbersome
and inefficient in the multi-dimensional setting. Instead, probabilistic
methods based on \eqref{eq:DPP}, in the spirit of \cite{Carmona08},
are usually more practical and versatile.

Indeed, as the diffusion $X$ is not controlled, this optimal switching
problem can be seen as an extended American option problem. This suggests
that, up to some adjustments, the probabilistic numerical tools developed
in this context (see \cite{Bouchard11} for instance) may be adapted
to solve \eqref{eq:v}.

To be more specific, consider a variant $\hat{v}$ of \eqref{eq:v}
such that the switching decisions can only take place on a finite
time grid $\Pi=\left\{ t_{0}=0<t_{1}<\ldots<t_{m}=T\right\} $ for
a fixed $T>0$. Then $\forall i\in\mathbb{I}_{q}$ , $\forall x\in\mathbb{R}^{d}$,
and $\forall t_{k}\in\Pi$, the dynamic programming principle \eqref{eq:DPP}
becomes:
\begin{align}
\hat{v}_{i}\left(t_{k},x\right) & =\max\left\{ E_{i}\left(t_{k},x\right),\max_{j\in\mathbb{I}_{q},\, j\neq i}\left\{ \hat{v}_{j}\left(t_{k},x\right)-k\left(t_{k},j-i\right)\right\} \right\} \label{eq:DPP-approx-1}
\end{align}
where:
\begin{align}
E_{i}\left(T,x\right) & :=\mathbb{E}\left[\int_{T}^{\infty}f_{i}\left(s,X_{s}^{T,x}\right)dt\right]\label{eq:DPP-Final}\\
E_{i}\left(t_{k},x\right) & :=\mathbb{E}\left[\int_{t_{k}}^{t_{k+1}}f_{i}\left(s,X_{s}^{t_{k},x}\right)dt+\hat{v}_{i}\left(t_{k+1},X_{t_{k+1}}^{t_{k},x}\right)\right]\,,\,\, k=m-1,\ldots,0\label{eq:DPP-CondExp}
\end{align}
and where the notation $X^{t,x}:=\left(X_{s}^{t,x}\right)_{s\geq t}$
refers to the process $X$ conditioned on the initial value $X_{t}=x$.

If, moreover, the cost function $k$ is such that at most one switch
can occur on a given date $t_{k}$ (triangular condition), then equation
\eqref{eq:DPP-approx-1} can be simplified into:
\begin{align}
\hat{v}_{i}\left(t_{k},x\right) & =\max_{j\in\mathbb{I}_{q}}\left\{ E_{j}\left(t_{k},x\right)-k\left(t_{k},j-i\right)\mathbf{1}_{\left\{ j\neq i\right\} }\right\} \label{eq:DPP-approx-2}
\end{align}
which is explicit in the sense that $\hat{v}_{.}\left(t_{k},.\right)$
directly depends on $\hat{v}_{.}\left(t_{k+1},.\right)$.

In practice, apart from the potential approximation of the stochastic
process $X$ and of the final values \eqref{eq:DPP-Final}, the difficulty
lies in the efficient computation of the conditional expectations
\eqref{eq:DPP-CondExp}.

In the American option literature, various approaches have been developed
to solve \eqref{eq:DPP-approx-2} efficiently. Notable examples are
the least-squares' approach (\cite{Longstaff01,Tsitsiklis01}), the
quantization approach and the Malliavin calculus based formulation
(see \cite{Bouchard11} for a thorough comparison and improvements
of these techniques). In the spirit of \cite{Carriere96}, one may
also consider non-parametric regression (see \cite{Kohler10} and
\cite{Todorovic07}) combined with speeding up techniques like Kd-trees
(\cite{Gray03,Lang05}) or the Fast Gauss Transform (\cite{Yang03,Morariu09,Raykar10,Spivak10,Sampath10})
in the case of kernel regression.

Here, we intend to solve \eqref{eq:v} on numerical applications which
bears the particularity of handling stochastic processes in high dimension
($\dim\left(X\right)=d\gg3$, with however $\dim\left(I\right)=d'\approx3$,
see Section \ref{sec:Application}). For such problems, the most adequate
technique so far seems to be the local regression method developed
in \cite{Bouchard11}. We are thus going to make use of this specific
method to solve \eqref{eq:DPP-approx-2} in practice.

In the following, we provide a detailed analysis of the above suggested
computational method.

\section{Numerical approximation and convergence analysis\label{sub:Convergence}}

This section is devoted to the precise description of the resolution
of \eqref{eq:v}, along the lines of the discussions from Subsection
\ref{sec:Outline}. Moreover, the convergence rate of the proposed
algorithm will be precisely assessed.

\subsection{Approximations}

Recall equation \eqref{eq:v} defining the value function $v\left(t,x,i\right)$
:
\begin{equation}
v\left(t,x,i\right)=\sup_{\alpha\in\mathcal{A}_{t,i}}\mathbb{E}\left[\int_{t}^{\infty}f\left(s,X_{s}^{t,x},I_{s}^{\alpha}\right)ds-\sum_{\tau_{n}\geq t}k\left(\tau_{n},\zeta_{n}\right)\right]\label{eq:v-1}
\end{equation}
We are going to consider the following sequence of approximations:
\begin{itemize}
\item \textit{{[}Finite time horizon{]}} The time horizon will be truncated
to a finite horizon $T$.
\item \textit{{[}Time discretization{]}} The continuous state process $X$
and investment process $I$ will be discretized with a time step $h$.
\item \textit{{[}Space localization{]}} The $\mathbb{R}^{d}$- valued process
$X$ will be projected into a bounded domain $\mathcal{D}_{\varepsilon}$,
parameterized by $\varepsilon$.
\item \textit{{[}Conditional expectation approximation{]}} The conditional
expectation involved in the dynamic programming equation will be replaced
by an empirical least-squares regression, computed on a bundle of
$M$ Monte Carlo trajectories, on a finite basis of local hypercubes
with edges of size $\delta$.
\end{itemize}
The rate of convergence of the algorithm will then be provided, as
a function of these five numerical parameters: $T$, $h$, $\varepsilon$,
$M$ and $\delta$.

\subsubsection{Finite time horizon}

The first step is to reduce the set of strategies to a finite horizon:
\begin{eqnarray}
v_{T}\left(t,x,i\right) & = & \sup_{\alpha\in\mathcal{A}_{t,i}^{T}}\mathbb{E}\left[\int_{t}^{T}f\left(s,X_{s}^{t,x},I_{s}^{\alpha}\right)ds-\sum_{t\leq\tau_{n}\leq T}k\left(\tau_{n},\zeta_{n}\right)+g_{f}\left(T,X_{T}^{t,x},I_{T}^{\alpha}\right)\right]\label{eq:v-2}\\
g_{f}\left(T,x,i\right) & := & \mathbb{E}\left[\int_{T}^{\infty}f\left(s,X_{s}^{T,x},i\right)ds\right]
\end{eqnarray}
where $0\leq t\leq T<+\infty$, and $\mathcal{A}_{t,i}^{T}\subset\mathcal{A}_{t,i}$
is the subset of strategies without switches strictly after time $T$.
Hence the final value $g_{f}$ corresponds to the remaining gain after
$T$.

Alternatively, one may choose, for convenience, another final value
$g$ instead of $g_{f}$, as long as it is Lipschitz-continuous and
satisfies a suitable condition (cf. equation \eqref{eq:gfInequality}).
The set of such functions will be denoted as $\Theta_{g_{f}}$. The
difference between the two value functions is quantified in Proposition
\ref{pro:ErrorControl-v1-v2}.

This freedom on the final values will be used in practice to avoid
a computation on an infinite interval $\left[T,\infty\right[$ as
in the definition of $g_{f}$.

From now on, we choose and fix one such $g\in\Theta_{g_{f}}$.

\subsubsection{Time discretization}

Then, we discretize the time segment $\left[0,T\right]$. Introduce
a time grid $\Pi=\left\{ t_{0}=0<t_{1}<\ldots<t_{N}=T\right\} $ with
constant mesh $h$. Consider the following approximation:
\begin{equation}
v_{\Pi}\left(t,x,i\right)=\sup_{\alpha\in\mathcal{A}_{t,i}^{\Pi}}\mathbb{E}\left[\int_{t}^{T}f\left(s,X_{s}^{t,x},I_{s}^{\alpha}\right)ds-\sum_{t\leq\tau_{n}\leq T}k\left(\tau_{n},\zeta_{n}\right)+g\left(T,X_{T}^{t,x},I_{T}^{\alpha}\right)\right]\label{eq:v-3}
\end{equation}
where $\mathcal{A}_{t,i}^{\Pi}\subset\mathcal{A}_{t,i}^{T}$ is the
subset of strategies such that switches can only occur at dates $\tau_{n}\in\Pi\cap\left[t,T\right]$.

Now, with a slight abuse of notation, we can safely switch from the
notation $\alpha=\left(\tau_{n},\zeta_{n}\right)_{n\geq0}$ to the
notation $\alpha=\left(\tau_{n},\iota_{n}\right)_{n\geq0}$ (remember
Subsection \ref{sub:Formulation}), replacing the quantity $\sum_{t\leq\tau_{n}\leq T}k\left(\tau_{n},\zeta_{n}\right)$
by $\sum_{t\leq\tau_{n}\leq T}k\left(\tau_{n},I_{\tau_{n}-h}^{\alpha},I_{\tau_{n}}^{\alpha}\right)$
or by $\sum_{t\leq\tau_{n}\leq T}k\left(\tau_{n},\iota_{n-1},\iota_{n}\right)$,
where $k\left(t,i,j\right)=k(t,j-i)$. The error between $v_{T}$
and $v_{\Pi}$ is quantified in Proposition \ref{pro:Time-Convergence}.

Next we also approximate the stochastic process $X$ by its Euler
scheme $\bar{X}=\left(\bar{X}_{s}\right)_{0\leq s\leq T}$, with dynamics:
\begin{eqnarray}
d\bar{X}_{s} & = & b\left(\pi\left(s\right),\bar{X}_{\pi\left(s\right)}\right)ds+\sigma\left(\pi\left(s\right),\bar{X}_{\pi\left(s\right)}\right)dW_{s}\,,\,\,0\leq s\leq T\label{eq:SDE-Euler}\\
\bar{X}_{0} & = & x_{0}\in\mathbb{R}^{d}\nonumber 
\end{eqnarray}
where $\forall s\in\left[0,T\right]$, $\pi\left(s\right):=\max\left\{ t\in\Pi;t\leq s\right\} $.
The new value function reads:
\begin{equation}
\bar{v}_{\Pi}\left(t,x,i\right)=\sup_{\alpha\in\mathcal{A}_{t,i}^{\Pi}}\mathbb{E}\left[\int_{t}^{T}f\left(\pi\left(s\right),\bar{X}_{s}^{t,x},I_{s}^{\alpha}\right)ds-\sum_{t\leq\tau_{n}\leq T}k\left(\tau_{n},\iota_{n-1},\iota_{n}\right)+g\left(T,\bar{X}_{T}^{t,x},I_{T}^{\alpha}\right)\right]\label{eq:v-4}
\end{equation}
The error between $v_{\Pi}$ and $\bar{v}_{\Pi}$ is computed in Proposition
\ref{pro:Euler-Convergence}.

\subsubsection{Space localization \label{sub:SpaceLocIntro}}

Fix $\varepsilon>0$. $\forall t\in\left[0,T\right]$, let $\mathcal{D}_{t}^{\varepsilon}$
be a bounded convex domain of $\mathbb{R}^{d}$. In particular there
exists $C\left(t,\varepsilon\right)>0$ such that $\forall x\in\mathcal{D}_{t}^{\varepsilon}$,
$\left|x\right|\leq C\left(t,\varepsilon\right)$. Let $\mathcal{P}_{t}^{\varepsilon}:\mathbb{R}^{d}\rightarrow\mathbb{R}^{d}$
denote the projection on $\mathcal{D}_{t}^{\varepsilon}$. This domain
is chosen such that $\forall s\in\left[0,t\right]$,
\begin{equation}
\mathbb{E}\left[\left|\bar{X}_{s}-\mathcal{P}_{t}^{\varepsilon}\left(\bar{X}_{s}\right)\right|\right]\leq\varepsilon\,.\label{eq:EpsilonDef}
\end{equation}
Denote this projection as $\bar{X}_{t}^{\varepsilon}$ :
\[
\bar{X}_{t}^{\varepsilon}:=\mathcal{P}_{t}^{\varepsilon}\left(\bar{X}_{t}\right)
\]
In other words, $\bar{X}_{t}^{\varepsilon}$ is equal to $\bar{X}_{t}$
most of the time (i.e. when $\bar{X}_{t}\in\mathcal{D}_{t}^{\varepsilon}$),
except when $\bar{X}_{t}$ is outside $\mathcal{D}_{t}^{\varepsilon}$,
in which case $\bar{X}_{t}^{\varepsilon}$ corresponds to the projection
of $\bar{X}_{t}$ onto $\mathcal{D}_{t}^{\varepsilon}$.

Define $\bar{v}_{\Pi}^{\varepsilon}$ as the value function $\bar{v}_{\Pi}$
from equation \eqref{eq:v-4} with $\bar{X}$ replaced by $\bar{X}^{\varepsilon}$.
The error between those two value functions is computed in Proposition
\ref{pro:SpaceLocalization}. 
\begin{example}
\label{ex:Brownian-Loc}To clarify this construction of space localization,
we explicit it on the very simple example of a $d$-dimensional standard
brownian motion $\left(W_{t}\right)_{t\in\left[0,T\right]}$. In this
case, $\bar{X}_{t}=X_{t}=W_{t}$. Choose $\mathcal{D}_{t}^{\varepsilon}$
to be a centered, symmetric hypercube: $\mathcal{D}_{t}^{\varepsilon}=\left[-C\left(t,\varepsilon\right),C\left(t,\varepsilon\right)\right]^{d}$
for some constant $C\left(t,\varepsilon\right)$. Hence, $\forall x\in\mathbb{R}^{d}$,
$\mathcal{P}_{t}^{\varepsilon}\left(x\right):=-C\left(t,\varepsilon\right)\wedge x\vee C\left(t,\varepsilon\right)$
component-wise. With this expressions, one can find a $C\left(t,\varepsilon\right)$
such that \eqref{eq:EpsilonDef} holds. Indeed, $\forall s\in\left[0,T\right]$:
\begin{equation}
\mathbb{E}\left[\left|W_{s}-\mathcal{P}_{t}^{\varepsilon}\left(W_{s}\right)\right|\right]\leq\mathbb{E}\left[\left|W_{t}-\mathcal{P}_{t}^{\varepsilon}\left(W_{t}\right)\right|\right]=\mathbb{E}\left[\left|W_{t}-C\left(t,\varepsilon\right)\right|\mathbf{1}_{\left\{ \left|W_{t}\right|>C\left(t,\varepsilon\right)\right\} }\right]=2^{d}\mathbb{E}\left[\left(W_{t}^{1}-C\left(t,\varepsilon\right)\right)^{+}\right]^{d}\label{eq:Bachelier}
\end{equation}
where $W^{1}$ is a one-dimensional Brownian motion. Hence, finding
a value for $C\left(t,\varepsilon\right)$ such that \eqref{eq:EpsilonDef}
holds boils down to inverting Bachelier's option pricing formula in
order to get the strike as a function of the price of the call option.
This is done in \cite{Bachelier00}, see \cite{Schachermayer08},
but under the form of a series expansion for small moneyness, which
is unsuitable for our purpose (because $C\left(t,\varepsilon\right)\rightarrow\infty$
when $\varepsilon\rightarrow0$). Thus, we are here only going to
look for a simply invertible upper bound for \eqref{eq:Bachelier}.
Denoting as $\mathcal{N}$ the cumulative distribution function of
a standard Gaussian random variable, and using the standard inequality
$1-\mathcal{N}\left(x\right)\geq\frac{1}{\sqrt{2\pi}}\frac{x}{x^{2}+1}e^{-\frac{x^{2}}{2}}$:
\begin{eqnarray*}
\mathbb{E}\left[\left(W_{t}^{1}-K\right)^{+}\right] & = & \int_{\frac{K}{\sqrt{t}}}^{+\infty}\left(x\sqrt{t}-K\right)\mathcal{N}'\left(x\right)dx=\frac{\sqrt{t}}{\sqrt{2\pi}}e^{-\frac{K^{2}}{2t}}-K\left(1-\mathcal{N}\left(\frac{K}{\sqrt{t}}\right)\right)\\
 & \leq & \frac{\sqrt{t}}{\sqrt{2\pi}}\left(1+\frac{K^{2}}{K^{2}+t}\right)e^{-\frac{K^{2}}{2t}}\leq\frac{2\sqrt{t}}{\sqrt{2\pi}}e^{-\frac{K^{2}}{2t}}
\end{eqnarray*}
Inverting this last upper bound, the inequality \eqref{eq:EpsilonDef}
is satisfied with $C\left(t,\varepsilon\right)=\sqrt{t\ln\left(\frac{8t}{\pi\varepsilon^{\frac{2}{d}}}\right)}$.
\end{example}

\subsubsection{Conditional expectation approximation\label{sub:CondExpIntro}}

From now on, in order to prevent the notation from becoming too cumbersome
and clumsy, we are going to drop the $\varepsilon$ index in the following,
i.e. $\bar{X}_{t}$ will stand for $\bar{X}_{t}^{\varepsilon}$, and
$\bar{v}_{\Pi}$ for $\bar{v}_{\Pi}^{\varepsilon}$.

For the fully discretized problem \eqref{eq:v-4}, the dynamic programming
principle \eqref{eq:DPP-approx-2} becomes:
\begin{align}
\bar{v}_{\Pi}\left(T,x,i\right) & =g\left(T,x,i\right)\nonumber \\
\bar{v}_{\Pi}\left(t_{n},x,i\right) & =\max_{j\in\mathbb{I}_{q}}\left\{ hf\left(t_{n},x,j\right)-k\left(t_{n},i,j\right)+\mathbb{E}\left[\bar{v}_{\Pi}\left(t_{n+1},\bar{X}_{t_{n+1}}^{t_{n},x},j\right)\right]\right\} \,,n=N-1,\ldots,0\label{eq:DDP-approx-3}
\end{align}

The last step is to approximate the conditional expectation appearing
in equation \eqref{eq:DDP-approx-3}. As discussed in Subsection \ref{sec:Outline},
we choose to approximate it using the following regression procedure.
Consider basis functions $\left(e_{k}\left(x\right)\right)_{1\leq k\leq K}$,
$K\in\mathbb{N}\cup\left\{ +\infty\right\} $, $x\in\mathbb{R}^{d}$.
For suitable functions $\varphi:\Pi\times\mathbb{R}^{d}\times\mathbb{I}_{q}\rightarrow\mathbb{R}$,
define:
\begin{equation}
\tilde{\lambda}=\tilde{\lambda}_{i}^{t_{n}}\left(\varphi\right):=\arg\min_{\lambda\in\mathbb{R}^{K}}\mathbb{E}\left[\left(\varphi\left(t_{n+1},\bar{X}_{t_{n+1}},i\right)-\sum_{k=1}^{K}\lambda_{k}e_{k}\left(\bar{X}_{t_{n}}\right)\right)^{2}\right]\label{eq:regr0}
\end{equation}
As truncating the approximated conditional expectations is a necessity
in theory as well as in practice (see \cite{Bouchard04,Gobet05,Tan11}),
suppose that there exist known bounds $\underline{\Gamma}^{t_{n},x}\left(\text{\ensuremath{\varphi}}\right)$
and $\overline{\Gamma}^{t_{n},x}\left(\varphi\right)$ on $\mathbb{E}\left[\varphi\left(t_{n+1},\bar{X}_{t_{n+1}}^{t_{n},x},i\right)\right]$
:
\begin{equation}
\underline{\Gamma}^{t_{n},x}\left(\text{\ensuremath{\varphi}}\right)\leq\mathbb{E}\left[\varphi\left(t_{n+1},\bar{X}_{t_{n+1}}^{t_{n},x},i\right)\right]\leq\overline{\Gamma}^{t_{n},x}\left(\varphi\right)\label{eq:condExpBounds}
\end{equation}
Then, $\forall i\in\mathbb{I}_{q}$ the quantity $\mathbb{E}\left[\varphi\left(t_{n+1},\bar{X}_{t_{n+1}}^{t_{n},x},i\right)\right]$
is approximated by:
\begin{equation}
\tilde{\mathbb{E}}\left[\varphi\left(t_{n+1},\bar{X}_{t_{n+1}}^{t_{n},x},i\right)\right]:=\underline{\Gamma}^{t_{n},x}\left(\text{\ensuremath{\varphi}}\right)\vee\sum_{k=1}^{K}\tilde{\lambda}_{k}e_{k}\left(x\right)\wedge\overline{\Gamma}^{t_{n},x}\left(\varphi\right)\label{eq:condExp0}
\end{equation}
which is used to define the next approximation $\tilde{v}_{\Pi}$
of the value function:
\begin{eqnarray}
\tilde{v}_{\Pi}\left(T,x,i\right) & = & g\left(T,x,i\right)\nonumber \\
\tilde{v}_{\Pi}\left(t_{n},x,i\right) & = & \max_{j\in\mathbb{I}_{q}}\left\{ hf\left(t_{n},x,j\right)-k\left(t_{n},i,j\right)+\tilde{\mathbb{E}}\left[\tilde{v}_{\Pi}\left(t_{n+1},\bar{X}_{t_{n+1}}^{t_{n},x},j\right)\right]\right\} \,,\,\, n=N-1,\ldots,0\label{eq:DDP-approx-4}
\end{eqnarray}

Interesting discussions on the choice of function basis can be found
in \cite{Bouchard11}. In particular they advocate bases of local
polynomials, which is numerically efficient and well-suited to tackle
large-dimensional problems (see Subsection \ref{sub:Complexity}).
However, for the sake of simplicity, we will restrict our study in
this section to a basis of indicator functions on local hypercubes
(as in \cite{Tan11} and \cite{Gobet05}) (which is the simplest example
of local polynomials) as defined below:

For every $t_{n}\in\Pi$, consider a partition of the domain $\mathcal{D}_{t_{n}}^{\varepsilon}$
into hypercubes $\left(B_{t_{n}}^{k}\right)_{k=1,\ldots,K_{\varepsilon}}$,
i.e., $\cup_{k=1,\ldots,K_{\varepsilon}}B_{t_{n}}^{k}=\mathcal{D}_{t_{n}}^{\varepsilon}$
and $B_{t_{n}}^{i}\cap B_{t_{n}}^{j}=\emptyset$ $\forall i\neq j$.
It may be deterministic, or $\mathcal{F}_{t}$-measurable. We only
assume that there exists $\left(\underline{\delta},\delta\right)\in\mathbb{R}_{+}^{2}$
with $\underline{\delta}\leq\delta$ such that the lengths of the
edges of the hypercubes, in each dimension, belong to $\left[\underline{\delta},\delta\right]$
(in particular, the volume of each hypercube $B_{t_{n}}^{k}$ belongs
to $\left[\underline{\delta}^{d},\delta^{d}\right]$). This liberty
over the definition of the partition enables to encompass the kind
of adaptative partition described in \cite{Bouchard11}. Then, the
basis functions considered here are defined by $e_{t_{n}}^{k}\left(x\right):=\mathbf{1}\left\{ x\in B_{t_{n}}^{k}\right\} $,
$x\in\mathbb{R}^{d}$, $1\leq k\leq K_{\varepsilon}$.

With this choice of function basis, the error between $\bar{v}_{\Pi}$
and $\tilde{v}_{\Pi}$ is computed in Proposition \ref{pro:RegConv}.

Finally, let $\left(\bar{X}_{t_{n}}^{m}\right)_{1\leq n\leq N}^{1\leq m\leq M}$
be a finite sample of size $M$ of paths of the process $\bar{X}$.
The final step is to replace the regression \eqref{eq:regr0} by a
regression on this sample:
\begin{equation}
\hat{\lambda}=\hat{\lambda}_{i}^{t_{n}}\left(\varphi\right):=\arg\min_{\lambda\in\mathbb{R}^{K}}\frac{1}{M}\sum_{m=1}^{M}\left[\left(\varphi\left(t_{n+1},\bar{X}_{t_{n+1}}^{m},i\right)-\sum_{k=1}^{K}\lambda_{k}e_{k}\left(\bar{X}_{t_{n}}^{m}\right)\right)^{2}\right]\,.\label{eq:regrM}
\end{equation}
Then $\forall i\in\mathbb{I}_{q}$ the quantity $\mathbb{E}\left[\varphi\left(t_{n+1},\bar{X}_{t_{n+1}}^{t_{n},x},i\right)\right]$
is approximated by:
\begin{equation}
\hat{\mathbb{E}}\left[\varphi\left(t_{n+1},\bar{X}_{t_{n+1}}^{t_{n},x},i\right)\right]:=\underline{\Gamma}^{t_{n},x}\left(\text{\ensuremath{\varphi}}\right)\vee\sum_{k=1}^{K}\hat{\lambda}_{k}e_{k}\left(x\right)\wedge\overline{\Gamma}^{t_{n},x}\left(\varphi\right)\label{eq:condExpM}
\end{equation}
leading to the final, computable approximation $\hat{v}_{\Pi}$ of
the value function:
\begin{eqnarray}
\hat{v}_{\Pi}\left(T,x,i\right) & = & g\left(T,x,i\right)\nonumber \\
\hat{v}_{\Pi}\left(t_{n},x,i\right) & = & \max_{j\in\mathbb{I}_{q}}\left\{ hf\left(t_{n},x,j\right)-k\left(t_{n},i,j\right)+\hat{\mathbb{E}}\left[\hat{v}_{\Pi}\left(t_{n+1},\bar{X}_{t_{n+1}}^{t_{n},x},j\right)\right]\right\} \,,\,\, n=N-1,\ldots,0\label{eq:DDP-approx-5}
\end{eqnarray}

The error between $\tilde{v}_{\Pi}$ and $\hat{v}_{\Pi}$ with the
same choice of function basis is given in Proposition \ref{pro:RegMConv}.
This proposition will make use of the following quantity:
\begin{eqnarray}
p\left(t_{n},\delta,\varepsilon\right) & := & \min_{t\in\Pi\cap\left[0,t_{n}\right]}\min_{B_{t}^{k}\subset\mathcal{D}_{t}^{\varepsilon}}\mathbb{P}\left(\bar{X}_{t}\in B_{t}^{k}\right)\label{eq:probaMin}
\end{eqnarray}
which is strictly positive, as the domains $\mathcal{D}_{t}^{\varepsilon}$,
$t\in\left[0,T\right]$ are bounded. More precisely, only lower bounds
of these quantities will be required.
\begin{example}
\label{ex:Brownian-Prob}Carrying on with Example \ref{ex:Brownian-Loc}
of a $d$-dimensional Brownian motion, we explicit a lower bound for
$p\left(t_{n},\delta,\varepsilon\right)$ in this simple case. First,
$\mathbb{P}\left(W_{t}\in B_{t_{n}}^{k}\right)=\int_{B_{t_{n}}^{k}}f_{W_{t}}\left(x\right)dx$
where $f_{W_{t}}$ is the density of $W_{t}$. As $\forall k$, $B_{t_{n}}^{k}\subset\mathcal{D}_{t_{n}}^{\varepsilon}$,
where in this example $\mathcal{D}_{t}^{\varepsilon}=\left[-C\left(t,\varepsilon\right),C\left(t,\varepsilon\right)\right]^{d}$
with $C\left(t,\varepsilon\right)=\sqrt{t\ln\left(\frac{8t}{\pi\varepsilon^{\frac{2}{d}}}\right)}$,
it holds that $\forall x\in\mathcal{D}_{t}^{\varepsilon}$, $f_{W_{t}}\left(x\right)\geq\left(f_{W_{t}^{1}}\left(C\left(t,\varepsilon\right)\right)\right)^{d}=\frac{\varepsilon}{\left(4t\right)^{d}}$.
Hence $\mathbb{P}\left(W_{t}\in B_{t}^{k}\right)\geq\frac{\varepsilon}{\left(4t\right)^{d}}\mathrm{Vol}\left(B_{t}^{k}\right)\geq\frac{\varepsilon}{\left(4t\right)^{d}}\underbar{\ensuremath{\delta}}^{d}$.
As a conclusion, $p\left(t_{n},\delta,\varepsilon\right)\geq\frac{\varepsilon}{\left(4t_{n}\right)^{d}}\underbar{\ensuremath{\delta}}^{d}$
. Remark however that this lower bound is very crude, and that it
can be very far below $p\left(t_{n},\delta,\varepsilon\right)$ for
large $\delta$.
\end{example}
Combining all these results, we obtain a rate of convergence of $\hat{v}_{\Pi}$
towards $v$:\smallskip{}

\begin{thm}
\label{Th:CONVERGENCERATE}$\forall p\geq1$ , $\exists C_{p}>0$
such that:
\begin{eqnarray*}
 &  & \left\Vert \max_{i\in\mathbb{I}_{q}}\left|v\left(t_{0},x_{0},i\right)-\hat{v}_{\Pi}\left(t_{0},x_{0},i\right)\right|\right\Vert _{L_{p}}\\
 & \leq & C_{p}\left\{ \left(1+\left|x_{0}\right|\right)e^{-\bar{\rho}T}+\left(1+\left|x_{0}\right|\right)^{\frac{3}{2}}\sqrt{h}+\varepsilon+\frac{\delta}{h}+\frac{1+C\left(T,\varepsilon\right)}{h\sqrt{M}p\left(T,\delta,\varepsilon\right)^{1-\frac{1}{p\vee2}}}+\frac{1+C\left(T,\varepsilon\right)}{hMp\left(T,\delta,\varepsilon\right)}\right\} 
\end{eqnarray*}
In particular, $\hat{v}_{\Pi}\left(0,x_{0},i\right)\rightarrow_{L_{p}}v\left(0,x_{0},i\right)$
uniformly in $i\in\mathbb{I}_{q}$ when $T\rightarrow\infty$, $h\rightarrow0$,
$\varepsilon\rightarrow0$, $\delta\rightarrow0$ and $M\rightarrow\infty$
with $\frac{\delta}{h}\rightarrow0$, $\frac{1+C\left(T,\varepsilon\right)}{h\sqrt{M}p\left(T,\delta,\varepsilon\right)^{1-\frac{1}{p\vee2}}}\rightarrow0$
and $\frac{1+C\left(T,\varepsilon\right)}{hMp\left(T,\delta,\varepsilon\right)}\rightarrow0$.
\end{thm}
\smallskip{}

\begin{rem}
If the cost function $k$ (recall Assumption \ref{asm:Lipschitz-f-g})
were to depend on $x$, then, under a usual Lipschitz condition on
$k$ (similar to that of $f$), Theorem \ref{Th:CONVERGENCERATE}
would still hold, replacing only the term $\left(1+\left|x_{0}\right|\right)^{\frac{3}{2}}\sqrt{h}$
by $\left(1+\left|x_{0}\right|^{\frac{5}{2}}\right)\sqrt{h\log\left(\frac{2T}{h}\right)}$
(recalling Remark \ref{rk:discretization}).
\end{rem}

\smallskip{}

\begin{rem}
The adaptative local basis can be such that each hypercube contains
approximately the same number of Monte Carlo trajectories (see \cite{Bouchard11}).
This means that $\frac{1}{p\left(T,\delta,\varepsilon\right)}\sim b$
where $b$ is the number of functions in the regression basis. With
this remark in mind, the leading error term in Theorem \ref{Th:CONVERGENCERATE}
behaves like $\frac{\sqrt{b}}{h\sqrt{M}}$ for $p=2$. This is close
to the corresponding statistical error term in \cite{Lemor06} ($\sqrt{\frac{b\log\left(M\right)}{hM}}$)
in the context of BSDEs. The advantage of their approach is that they
can handle any (orthonormal) regression basis, while our approach
(in the context of optimal switching) provides a bound on the $L_{p}$
error for every $p\geq1$.
\end{rem}
\smallskip{}

\begin{example}
\label{ex:Brownian-Rate}In the case of a $d$-dimensional Brownian
motion, the rate of convergence of Theorem \ref{Th:CONVERGENCERATE}
can be explicited further, using the upper bound on $C\left(T,\varepsilon\right)$
from Example \ref{ex:Brownian-Loc} and the lower bound on $p\left(T,\delta,\varepsilon\right)$
from Example \ref{ex:Brownian-Prob}. Moreover, one can express the
rate of convergence as a function of only one parameter, choosing
the five numerical parameters $T$, $h$, $\varepsilon$, $\delta$
and $M$ accordingly. For instance, assuming $\underline{\delta}=\delta$,
and minimizing over $\delta$, $h$, $\varepsilon$ and $T$, one
can get a convergence rate upper bounded by $C_{p}\left(1+\left|x\right|\right)^{\frac{3}{2}}z$
by choosing $M\sim z^{-\frac{1}{2}\left[6\left(d+1\right)\right]^{2}}$.
This is admittedly highly demanding in terms of sample size $M$,
but remember that this expression suffers from the crude lower bound
on $p\left(T,\delta,\varepsilon\right)$ we chose previously.
\end{example}

\subsection{Convergence analysis}

From now on, we suppose that all the assumptions from Subsection \ref{sub:Assumptions}
are in force.

\subsubsection{Finite time horizon}
\begin{lem}
\label{lem:ErrorControl-v1-v2}There exists $C>0$ such that $\forall\left(t,x,i\right)\in\mathbb{R}_{+}\times\mathbb{R}^{d}\times\mathbb{R}^{d'}$:
\[
0\leq v\left(t,x,i\right)-v_{T}\left(t,x,i\right)\leq C\left(1+\left|x\right|\right)e^{-\bar{\rho}t\vee T-\rho_{1}t}\,.
\]
\end{lem}
\begin{proof}
First, we introduce the following notations:
\begin{eqnarray}
H\left(t,T,x,\alpha\right) & := & \int_{t}^{T}f\left(s,X_{s}^{t,x},I_{s}^{\alpha}\right)ds-\sum_{t<\tau_{n}\leq T}k\left(\tau_{n},\zeta_{n}\right)\label{eq:H}\\
J\left(t,T,x,\alpha\right) & := & \mathbb{E}\left[H\left(t,T,x,\alpha\right)\right]\label{eq:J}
\end{eqnarray}
for any admissible strategy $\alpha\in\mathcal{A}_{t,i}$. In particular:
\begin{eqnarray}
v\left(t,x,i\right)=\sup_{\alpha\in\mathcal{A}_{t,i}}J\left(t,+\infty,x,\alpha\right) & \,,\,\, & v_{T}\left(t,x,i\right)=\sup_{\alpha\in\mathcal{A}_{t,i}^{T}}J\left(t,+\infty,x,\alpha\right)\,.\label{eq:vT}
\end{eqnarray}
Fix $\left(t,x,i\right)\in\mathbb{R}_{+}\times\mathbb{R}^{d}\times\mathbb{R}^{d'}$.
Using equation \eqref{eq:vT}:
\[
v_{T}\left(t,x,i\right)=\sup_{\alpha\in\mathcal{A}_{t,i}^{T}}J\left(t,\infty,x,\alpha\right)\leq\sup_{\alpha\in\mathcal{A}_{t,i}}J\left(t,\infty,x,\alpha\right)=v\left(t,x,i\right)
\]
which provides the first inequality. Consider now the second inequality.
Choose $\varepsilon>0$. From the definition of $v$ (equation \eqref{eq:v-1})
there exists a strategy $\alpha^{\varepsilon}\in\mathcal{A}_{t,i}$
such that:
\[
v\left(t,x,i\right)-\varepsilon\leq J\left(t,\infty,x,\alpha^{\varepsilon}\right)\leq v\left(t,x,i\right)
\]
Define the truncated strategy $\alpha_{T}^{\varepsilon}\in\mathcal{A}_{t,i}^{T}$
such that $\forall s\in\left[t,T\right]$, $I_{s}^{\alpha_{T}^{\varepsilon}}=I_{s}^{\alpha^{\varepsilon}}$
and $\forall s>T$, $I_{s}^{\alpha_{T}^{\varepsilon}}=I_{T}^{\alpha^{\varepsilon}}$.
In order not to mix up the variables $\tau_{n}$ and $\zeta_{n}$
from different strategies, we add the name of the strategy in index
when needed. Then:
\begin{eqnarray*}
 &  & H\left(t,\infty,x,\alpha^{\varepsilon}\right)-H\left(t,\infty,x,\alpha_{T}^{\varepsilon}\right)\\
 & = & \left\{ \int_{t}^{\infty}f\left(s,X_{s}^{t,x},I_{s}^{\alpha^{\varepsilon}}\right)ds-\sum_{\tau_{n}^{\alpha^{\epsilon}}\geq t}k\left(\tau_{n}^{\alpha^{\epsilon}},\zeta_{n}^{\alpha^{\epsilon}}\right)\right\} \\
 &  & -\left\{ \int_{t}^{\infty}f\left(s,X_{s}^{t,x},I_{s}^{\alpha_{T}^{\varepsilon}}\right)ds-\sum_{\tau_{n}^{\alpha_{T}^{\epsilon}}\geq t}k\left(\tau_{n}^{\alpha_{T}^{\epsilon}},\zeta_{n}^{\alpha_{T}^{\epsilon}}\right)\right\} \\
 & = & \left\{ \int_{t}^{\infty}f\left(s,X_{s}^{t,x},I_{s}^{\alpha^{\varepsilon}}\right)ds-\sum_{\tau_{n}^{\alpha^{\epsilon}}\geq t}k\left(\tau_{n}^{\alpha^{\epsilon}},\zeta_{n}^{\alpha^{\epsilon}}\right)\right\} \\
 &  & -\left\{ \int_{t}^{t\vee T}f\left(s,X_{s}^{t,x},I_{s}^{\alpha^{\varepsilon}}\right)ds+\int_{t\vee T}^{\infty}f\left(s,X_{s}^{t,x},I_{t\vee T}^{\alpha^{\varepsilon}}\right)ds-\sum_{t\vee T\geq\tau_{n}^{\alpha^{\epsilon}}\geq t}k\left(\tau_{n}^{\alpha^{\epsilon}},\zeta_{n}^{\alpha^{\epsilon}}\right)\right\} \\
 & = & \int_{t\vee T}^{\infty}f\left(s,X_{s}^{t,x},I_{s}^{\alpha^{\varepsilon}}\right)ds-\int_{t\vee T}^{\infty}f\left(s,X_{s}^{t,x},I_{t\vee T}^{\alpha^{\varepsilon}}\right)ds-\sum_{\tau_{n}^{\alpha^{\epsilon}}\geq t\vee T}k\left(\tau_{n}^{\alpha^{\epsilon}},\zeta_{n}^{\alpha^{\epsilon}}\right)\\
 & \leq & \int_{t\vee T}^{\infty}f\left(s,X_{s}^{t,x},I_{s}^{\alpha^{\varepsilon}}\right)ds-\int_{t\vee T}^{\infty}f\left(s,X_{s}^{t,x},I_{t\vee T}^{\alpha^{\varepsilon}}\right)ds
\end{eqnarray*}
as $k\left(s,0\right)=0$ and $k\geq0$ (Assumption \ref{asm:costs}).
Hence, using Jensen's inequality and equation \eqref{eq:ftbound},
$\exists C>0$ such that
\begin{eqnarray*}
\left|J\left(t,\infty,x,\alpha^{\varepsilon}\right)-J\left(t,\infty,x,\alpha_{T}^{\varepsilon}\right)\right| & \leq & \mathbb{E}\left[\left|H\left(t,\infty,x,\alpha^{\varepsilon}\right)-H\left(t,\infty,x,\alpha_{T}^{\varepsilon}\right)\right|\right]\\
 & \leq & \mathbb{E}\left[\int_{t\vee T}^{\infty}\left|f\left(s,X_{s}^{t,x},I_{s}^{\alpha^{\varepsilon}}\right)\right|ds\right]+\mathbb{E}\left[\int_{t\vee T}^{\infty}\left|f\left(s,X_{s}^{t,x},I_{t\vee T}^{\alpha^{\varepsilon}}\right)\right|ds\right]\\
 & \leq & C\left(1+\left|x\right|\right)e^{-\bar{\rho}t\vee T-\rho_{1}t}
\end{eqnarray*}
Finally, given that $\, v\left(t,x,i\right)\leq\varepsilon+J\left(t,\infty,x,\alpha^{\varepsilon}\right)\,$
and $\, v_{T}\left(t,x,i\right)\geq J\left(t,\infty,x,\alpha_{T}^{\varepsilon}\right)\,$,
the following holds:
\begin{eqnarray*}
v\left(t,x,i\right)-v_{T}\left(t,x,i\right) & \leq & \varepsilon+J\left(t,\infty,x,\alpha^{\varepsilon}\right)-J\left(t,\infty,x,\alpha_{T}^{\varepsilon}\right)\\
 & \leq & \varepsilon+C\left(1+\left|x\right|\right)e^{-\bar{\rho}t\vee T-\rho_{1}t}\,.
\end{eqnarray*}
Since this is true for any $\varepsilon>0$, and that $C$, $\rho$
and $\rho_{1}$ do not depend on $\varepsilon$, the proposition is
proved.
\end{proof}

\medskip{}

Now, we focus on the final boundary $g_{f}$. For the time being,
denote the value function \eqref{eq:v-2} as $v_{T}^{g_{f}}$ to emphasize
the dependence of $v$ on the terminal condition. As a consequence
of equation \eqref{eq:ftbound}, $\forall\left(x,i\right)\in\mathbb{R}^{d}\times\mathbb{I}_{q}$:
\begin{equation}
\left|g_{f}(T,x,i)\right|\leq C\left(1+\left|x\right|\right)e^{-\rho T}\label{eq:gfInequality}
\end{equation}
Hence, define the class $\Theta_{g_{f}}$ of Lipschitz functions from
$\mathbb{R}_{+}\times\mathbb{R}^{d}\times\mathbb{I}_{q}$ into $\mathbb{R}$
such that $\forall g\in\Theta_{g_{f}}$, $\forall\left(T,x,x',i\right)\in\mathbb{R}_{+}\times\mathbb{R}^{d}\times\mathbb{R}^{d}\times\mathbb{I}_{q}$:
\begin{eqnarray}
\left|g(T,x,i)-g(T,x',i)\right| & \leq & C_{g}e^{-\rho T}\left|x-x'\right|\label{eq:gLipschitz}\\
\left|g(T,x,i)\right| & \leq & C_{g}e^{-\rho T}\left(1+\left|x\right|\right)\label{eq:gInequality}
\end{eqnarray}
for some $C_{g}>0$. Obviously $g_{f}\in\Theta_{g_{f}}$. Now, for
any $g\in\Theta_{g_{f}}$, denote as $v_{T}^{g}$ the value function
defined as in equation \eqref{eq:v-2} with $g$ instead of $g_{f}$.
We are going to show that the precise approximation error due to the
choice of final value $g$ does not matter much as long as $g$ is
chosen in this class $\Theta_{g_{f}}$.

\medskip{}

\begin{lem}
\label{lem:ErrorControl-vgt-vg}There exists $C>0$ such that $\forall\left(t,x,i\right)\in\mathbb{R}_{+}\times\mathbb{R}^{d}\times\mathbb{I}_{q}$:
\[
\left|v_{T}^{g_{f}}\left(t,x,i\right)-v_{T}^{g}\left(t,x,i\right)\right|\leq C\left(1+\left|x\right|\right)e^{-\bar{\rho}t\vee T-\rho_{1}t}
\]
\end{lem}
\begin{proof}
Fix $\left(t,x,i\right)\in\mathbb{R}_{+}\times\mathbb{R}^{d}\times\mathbb{I}_{q}$.
To shorten the proof, we assume that $v_{T}^{g_{f}}$ (resp. $v_{T}^{g}$)
admits an optimal strategy $\alpha_{f}^{*}\in\mathcal{A}_{t,i}^{T}$
(resp. $\alpha^{*}\in\mathcal{A}_{t,i}^{T}$) (this assumption can
then be relaxed using $\varepsilon$-optimal strategies as in the
proof of Proposition \ref{lem:ErrorControl-v1-v2})%
\footnote{Note that under the assumptions from Subsection \ref{sub:Assumptions},
one may use Theorem 3.1 from \cite{Hu10} to get the existence of
a unique optimal strategy $\alpha^{*}$ for the value function \eqref{eq:v-2},
satisfying $\mathbb{E}\left[\left|\sum_{0\leq\tau_{n}^{\alpha^{*}}\leq T}k\left(\tau_{n}^{\alpha^{*}},\zeta_{n}^{\alpha^{*}}\right)\right|^{2}\right]<\infty$%
}. Therefore, recalling the notations $H$ (equation \eqref{eq:H})
and $J$ (equation \eqref{eq:J}) introduced in the proof of Lemma
\ref{lem:ErrorControl-v1-v2}:
\begin{eqnarray*}
v_{T}^{g_{f}}\left(t,x,i\right)-v_{T}^{g}\left(t,x,i\right) & = & J\left(t,T,x,\alpha_{f}^{*}\right)+\mathbb{E}\left[g_{f}\left(T,X_{T}^{t,x},I_{T}^{\alpha_{f}^{*}}\right)\right]-J\left(t,T,x,\alpha^{*}\right)-\mathbb{E}\left[g\left(T,X_{T}^{t,x},I_{T}^{\alpha^{*}}\right)\right]\\
 & = & J\left(t,T,x,\alpha_{f}^{*}\right)+\mathbb{E}\left[g\left(T,X_{T}^{t,x},I_{T}^{\alpha_{f}^{*}}\right)\right]-J\left(t,T,x,\alpha^{*}\right)-\mathbb{E}\left[g\left(T,X_{T}^{t,x},I_{T}^{\alpha^{*}}\right)\right]\\
 &  & +\mathbb{E}\left[g_{f}\left(T,X_{T}^{t,x},I_{T}^{\alpha_{f}^{*}}\right)-g\left(T,X_{T}^{t,x},I_{T}^{\alpha^{*}}\right)\right]\\
 & \leq & \mathbb{E}\left[g_{f}\left(T,X_{T}^{t,x},I_{T}^{\alpha_{f}^{*}}\right)-g\left(T,X_{T}^{t,x},I_{T}^{\alpha^{*}}\right)\right]\\
 & \leq & C\left(1+\mathbb{E}\left[\left|X_{T}^{t,x}\right|\right]\right)e^{-\rho T}\,\,\leq\,\,\, C\left(1+\left|x\right|\right)e^{-\bar{\rho}t\vee T-\rho_{1}t}
\end{eqnarray*}
Symmetrically, the same inequality holds for $v_{T}^{g}\left(t,x,i\right)-v_{T}^{g_{f}}\left(t,x,i\right)$,
ending the proof.
\end{proof}

\medskip{}

\begin{prop}
\label{pro:ErrorControl-v1-v2}There exists $C>0$ independent of
$T$ such that $\forall\left(t,x,i\right)\in\mathbb{R}_{+}\times\mathbb{R}^{d}\times\mathbb{I}_{q}$
and $\forall g\in\Theta_{g_{f}}$:
\[
\left|v\left(t,x,i\right)-v_{T}^{g}\left(t,x,i\right)\right|\leq C\left(1+\left|x\right|\right)e^{-\bar{\rho}t\vee T-\rho_{1}t}
\]
\end{prop}
\begin{proof}
Combine Lemmas \ref{lem:ErrorControl-v1-v2} and \ref{lem:ErrorControl-vgt-vg}.
\end{proof}

\medskip{}

From now on, we choose and keep one final value function $g\in\Theta_{g_{f}}$,
and remove the index $g$ from the notation of $v$ and its subsequent
approximations.

\subsubsection{Time Discretization}
\begin{prop}
\label{pro:Time-Convergence}There exists a positive constant $C$
such that for any $\left(t,x,i\right)\in\Pi\times\mathbb{R}^{d}\times\mathbb{I}_{q}$
:
\begin{equation}
\left|v_{T}\left(t,x,i\right)-v_{\Pi}\left(t,x,i\right)\right|\leq Ce^{-\rho t}\left(1+\left|x\right|^{\frac{3}{2}}\right)h^{\frac{1}{2}}\label{eq:DiscretizationCvg}
\end{equation}
\end{prop}
\begin{proof}
Under the assumptions from Subsection \ref{sub:Assumptions}, one
can apply Theorem 3.1 in \cite{Gassiat11} to prove \eqref{eq:DiscretizationCvg},
noticing that the cost function $k$ does not depend on the state
variable $x$. 

Use the discounting factor in the definition of $f$ to factor the
$e^{-\rho t}$ term and to get that $C$ does not depend on $T$.\end{proof}
\begin{rem}
Another alternative to get this rate of $h^{\frac{1}{2}}$ is to work
with the reflected BSDE representation of $v_{\Pi}$, as in \cite{Carmona08}
(adapting \cite{Bouchard04}) or \cite{Chassagneux11}.
\end{rem}

\begin{rem}
\label{rk:discretization}Were the cost function $k$ to depend on
the state variable, the upper bound in Proposition \ref{pro:Time-Convergence}
would only be $Ce^{-\rho t}\left(1+\left|x\right|^{\frac{5}{2}}\right)\left(h\log\left(\frac{2T}{h}\right)\right)^{\frac{1}{2}}$,
as stated in \cite{Gassiat11} (making use of results from \cite{Fischer09}). 
\end{rem}

\medskip{}

\begin{prop}
\label{pro:Euler-Convergence}There exists $C>0$ such that for any
$\left(t,x,i\right)\in\Pi\times\mathbb{R}^{d}\times\mathbb{I}_{q}$
:
\[
\left|v_{\Pi}\left(t,x,i\right)-\bar{v}_{\Pi}\left(t,x,i\right)\right|\leq Ce^{-\rho t}h^{\frac{1}{2}}
\]
\end{prop}
\begin{proof}
$T$ and $g$ being fixed, we can define, in the spirit of equations
\eqref{eq:H} and \eqref{eq:J}, the following quantities:
\begin{eqnarray}
H\left(t,x,\alpha\right) & := & \int_{t}^{T}f\left(s,X_{s}^{t,x},I_{s}^{\alpha}\right)ds-\sum_{t\leq\tau_{n}\leq T}k\left(\tau_{n},\iota_{n-1},\iota_{n}\right)+g\left(t\vee T,X_{t\vee T}^{t,x},I_{t\vee T}^{\alpha}\right)\label{eq:newH}\\
J\left(t,x,\alpha\right) & := & \mathbb{E}\left[H\left(t,x,\alpha\right)\right]\label{eq:newJ}\\
\bar{H}\left(t,x,\alpha\right) & := & \int_{t}^{T}f\left(\pi\left(s\right),\bar{X}_{s}^{t,x},I_{s}^{\alpha}\right)ds-\sum_{t\leq\tau_{n}\leq T}k\left(\tau_{n},\iota_{n-1},\iota_{n}\right)+g\left(t\vee T,\bar{X}_{t\vee T}^{t,x},I_{t\vee T}^{\alpha}\right)\label{eq:Hbar}\\
\bar{J}\left(t,x,\alpha\right) & := & \mathbb{E}\left[\bar{H}\left(t,x,\alpha\right)\right]\label{eq:Jbar}
\end{eqnarray}
for any admissible strategy $\alpha\in\mathcal{A}_{t,i}^{\Pi}$. For
these discretized problems, the existence of optimal controls $\alpha^{*}$
and $\bar{\alpha}^{*}$ is granted. Hence:
\begin{eqnarray*}
v_{\Pi}\left(t,x,i\right)-\bar{v}_{\Pi}\left(t,x,i\right) & = & J\left(t,x,\alpha^{*}\right)-\bar{J}\left(t,x,\bar{\alpha}^{*}\right)\\
 & = & J\left(t,x,\alpha^{*}\right)-\bar{J}\left(t,x,\alpha^{*}\right)+\left\{ \bar{J}\left(t,x,\alpha^{*}\right)-\bar{J}\left(t,x,\bar{\alpha}^{*}\right)\right\} \\
 & \leq & J\left(t,x,\alpha^{*}\right)-\bar{J}\left(t,x,\alpha^{*}\right)\\
 & = & \int_{t}^{T}e^{-\rho s}\mathbb{E}\left[\tilde{f}\left(s,X_{s}^{t,x},I_{s}^{\alpha^{*}}\right)-\tilde{f}\left(\pi\left(s\right),\bar{X}_{s}^{t,x},I_{s}^{\alpha^{*}}\right)\right]ds\\
 &  & +\mathbb{E}\left[g\left(T,X_{T}^{t,x},I_{T}^{\alpha^{*}}\right)-g\left(T,\bar{X}_{T}^{t,x},I_{T}^{\alpha^{*}}\right)\right]\\
 & \leq & C_{f}\int_{t}^{T}e^{-\rho s}\mathbb{E}\left[\left|X_{s}^{t,x}-\bar{X}_{s}^{t,x}\right|\right]ds+C_{g}e^{-\rho T}\mathbb{E}\left[\left|X_{T}^{t,x}-\bar{X}_{T}^{t,x}\right|\right]\\
 & \leq & Ce^{-\rho t}\mathbb{E}\left[\sup_{t\leq s\leq T}\left|X_{s}^{t,x}-\bar{X}_{s}^{t,x}\right|\right]\leq Ce^{-\rho t}h^{\frac{1}{2}}
\end{eqnarray*}
using the strong convergence speed of the Euler scheme on $\left[t,T\right]$.
Symmetrically, the same inequality holds for $\bar{v}_{\Pi}\left(t,x,i\right)-v_{\Pi}\left(t,x,i\right)$,
ending the proof.
\end{proof}

\subsubsection{Space localization}

Recall from Subsection \ref{sub:SpaceLocIntro} the definition of
the bounded domain $\mathcal{D}_{t}^{\varepsilon}$, $t\in\left[0,T\right]$.
\begin{prop}
\label{pro:SpaceLocalization}$\forall\varepsilon>0$, there exists
$C>0$ such that for any $\left(x,i\right)\in\mathbb{R}^{d}\times\mathbb{I}_{q}$
:
\[
\left|\bar{v}_{\Pi}\left(0,x,i\right)-\bar{v}_{\Pi}^{\varepsilon}\left(0,x,i\right)\right|\leq C\varepsilon
\]
\end{prop}
\begin{proof}
Recall the definitions of $\bar{H}\left(t,x,\alpha\right)$ (equation
\eqref{eq:Hbar}) and $\bar{J}\left(t,x,\alpha\right)$ (equation
\eqref{eq:Jbar}), and define the quantities $\bar{H}^{\varepsilon}\left(t,x,\alpha\right)$
and $\bar{J}^{\varepsilon}\left(t,x,\alpha\right)$ like $\bar{H}\left(t,x,\alpha\right)$
and $\bar{J}\left(t,x,\alpha\right)$ but with $\bar{X}$ replaced
by $\bar{X}^{\varepsilon}$. Then, for every $\left(t,x,i\right)\in\Pi\times\mathbb{R}^{d}\times\mathbb{I}_{q}$
and $\alpha\in\mathcal{A}_{t,i}^{\Pi}$:
\begin{eqnarray*}
\bar{J}\left(t,x,\alpha\right) & = & \bar{J}^{\varepsilon}\left(t,x,\alpha\right)+\int_{t}^{T}\mathbb{E}\left[f\left(\pi\left(s\right),\bar{X}_{s}^{t,x},I_{s}^{\alpha}\right)-f\left(\pi\left(s\right),\bar{X}_{s}^{\varepsilon,t,x},I_{s}^{\alpha}\right)\right]ds\\
 &  & +\,\mathbb{E}\left[g\left(T,\bar{X}_{T}^{t,x},I_{T}^{\alpha}\right)-g\left(T,\bar{X}_{T}^{\varepsilon,t,x},I_{T}^{\alpha}\right)\right]
\end{eqnarray*}
But:
\begin{eqnarray*}
 &  & \left|\int_{t}^{T}\mathbb{E}\left[f\left(\pi\left(s\right),\bar{X}_{s}^{t,x},I_{s}^{\alpha}\right)-f\left(\pi\left(s\right),\bar{X}_{s}^{\varepsilon,t,x},I_{s}^{\alpha}\right)\right]ds+\mathbb{E}\left[g\left(T,\bar{X}_{T}^{t,x},I_{T}^{\alpha}\right)-g\left(T,\bar{X}_{T}^{\varepsilon,t,x},I_{T}^{\alpha}\right)\right]\right|\\
 &  & \leq C_{f}\int_{t}^{T}e^{-\rho s}\mathbb{E}\left[\left|\bar{X}_{s}^{t,x}-\bar{X}_{s}^{\varepsilon,t,x}\right|\right]ds+C_{g}e^{-\rho T}\mathbb{E}\left[\left|\bar{X}_{T}^{t,x}-\bar{X}_{T}^{\varepsilon,t,x}\right|\right]
\end{eqnarray*}
It follows that:
\[
\left|\bar{v}_{\Pi}\left(t,x,i\right)-\bar{v}_{\Pi}^{\varepsilon}\left(t,x,i\right)\right|\leq C_{f}\int_{t}^{T}e^{-\rho s}\mathbb{E}\left[\left|\bar{X}_{s}^{t,x}-\bar{X}_{s}^{\varepsilon,t,x}\right|\right]ds+C_{g}e^{-\rho T}\mathbb{E}\left[\left|\bar{X}_{T}^{t,x}-\bar{X}_{T}^{\varepsilon,t,x}\right|\right]
\]
In particular, at $t=0$, using equation \eqref{eq:EpsilonDef}, $\exists C>0$
such that:
\[
\left|\bar{v}_{\Pi}\left(0,x,i\right)-\bar{v}_{\Pi}^{\varepsilon}\left(0,x,i\right)\right|\leq C\varepsilon
\]

\end{proof}

\subsubsection{Conditional expectation approximation}

From now on the domains $\mathcal{D}_{t}^{\varepsilon}$, $t\in\left[0,T\right]$
are fixed once and for all, and, with a slight abuse of notation,
we will drop $\varepsilon$ from the subsequent notations.

We start with preliminary remarks. Recalling Subsection \ref{sub:CondExpIntro},
with this choice of basis, $\tilde{\lambda}_{i}^{t_{n}}\left(\varphi\right)$
(equation \eqref{eq:regr0}) and $\hat{\lambda}_{i}^{t_{n}}\left(\varphi\right)$
(equation \eqref{eq:regrM}) become:
\begin{align*}
\tilde{\lambda}_{i}^{t_{n}}\left(\varphi\right) & =\frac{\mathbb{E}\left[\varphi\left(t_{n+1},\bar{X}_{t_{n+1}},i\right)\mathbf{1}\left\{ \bar{X}_{t_{n}}\in B_{t_{n}}^{k}\right\} \right]}{\mathbb{P}\left(\bar{X}_{t_{n}}\in B_{t_{n}}^{k}\right)}=\mathbb{E}\left[\left.\varphi\left(t_{n+1},\bar{X}_{t_{n+1}},i\right)\right|\bar{X}_{t_{n}}\in B_{t_{n}}^{k}\right]\,,\,\,1\leq k\leq K_{\varepsilon}\\
\hat{\lambda}_{i}^{t_{n}}\left(\varphi\right) & =\frac{\frac{1}{M}\sum_{m=1}^{M}\varphi\left(t_{n+1},\bar{X}_{t_{n+1}}^{m},i\right)\mathbf{1}\left\{ \bar{X}_{t_{n}}^{m}\in B_{t_{n}}^{k}\right\} }{\frac{1}{M}\sum_{m=1}^{M}\mathbf{1}\left\{ \bar{X}_{t_{n}}^{m}\in B_{t_{n}}^{k}\right\} }\,,\,\,1\leq k\leq K_{\varepsilon}
\end{align*}

Extending these equations, define
\begin{align}
\tilde{\lambda}_{i}^{t_{n},x}\left(\varphi\right) & :=\frac{\mathbb{E}\left[\varphi\left(t_{n+1},\bar{X}_{t_{n+1}},i\right)\mathbf{1}\left\{ \bar{X}_{t_{n}}\in B_{t_{n}}\left(x\right)\right\} \right]}{\mathbb{P}\left(\bar{X}_{t_{n}}\in B_{t_{n}}\left(x\right)\right)}=\mathbb{E}\left[\left.\varphi\!\left(t_{n+1},\bar{X}_{t_{n+1}},i\right)\right|\bar{X}_{t_{n}}\in B_{t_{n}}\left(x\right)\right]\label{eq:lambda(x)tilde}\\
\hat{\lambda}_{i}^{t_{n},x}\left(\varphi\right) & :=\frac{\frac{1}{M}\sum_{m=1}^{M}\varphi\left(t_{n+1},\bar{X}_{t_{n+1}}^{m},i\right)\mathbf{1}\left\{ \bar{X}_{t_{n}}^{m}\in B_{t_{n}}\left(x\right)\right\} }{\frac{1}{M}\sum_{m=1}^{M}\mathbf{1}\left\{ \bar{X}_{t_{n}}^{m}\in B_{t_{n}}\left(x\right)\right\} }=\frac{1}{M_{t_{n}}^{x}}\sum_{m\in\mathcal{M}_{t_{n}}^{x}}\!\!\varphi\!\left(t_{n+1},\bar{X}_{t_{n+1}}^{m},i\right)\label{eq:lambda(x)hat}
\end{align}

for every $\left(t_{n},x,i\right)\in\Pi\tilde{\times}\mathcal{D}_{\Pi}^{\varepsilon}\times\mathbb{I}_{q}$,
where $\forall x\in\mathcal{D}_{t_{n}}^{\varepsilon}$, $B_{t_{n}}\left(x\right)$
is the unique hypercube in the partition which contains $x$ at time
$t_{n}$, $\mathcal{M}_{t_{n}}^{x}:=\left\{ m\in\left[1,M\right],\,\bar{X}_{t_{n}}^{m}\in B_{t_{n}}\left(x\right)\right\} $
and $M_{t_{n}}^{x}:=\#\mathcal{M}_{t_{n}}^{x}$.

Finally, recalling the approximated conditional expectations \eqref{eq:condExp0}
and \eqref{eq:condExpM},

define for any $\left(t_{n},x,j\right)\in\Pi\tilde{\times}\mathcal{D}_{\Pi}^{\varepsilon}\times\mathbb{I}_{q}$
and any measurable function $\varphi:\Pi\times\mathbb{R}^{d}\times\mathbb{I}_{q}\rightarrow\mathbb{R}$,
the following quantities:
\begin{eqnarray}
\Phi_{j}^{t_{n},x}\left(\varphi\right) & := & \mathbb{E}\left[\varphi\left(t_{n+1},\bar{X}_{t_{n+1}}^{t_{n},x},j\right)\right]\label{eq:phi(x)}\\
\tilde{\Phi}_{j}^{t_{n},x}\left(\varphi\right) & := & \tilde{\mathbb{E}}\left[\varphi\left(t_{n+1},\bar{X}_{t_{n+1}}^{t_{n},x},j\right)\right]=\underline{\Gamma}^{t_{n},x}\left(\text{\ensuremath{\varphi}}\right)\vee\tilde{\lambda}_{j}^{t_{n},x}\left(\varphi\right)\wedge\overline{\Gamma}^{t_{n},x}\left(\varphi\right)\label{eq:phi(x)tilde}\\
\hat{\Phi}_{j}^{t_{n},x}\left(\varphi\right) & := & \hat{\mathbb{E}}\left[\varphi\left(t_{n+1},\bar{X}_{t_{n+1}}^{t_{n},x},j\right)\right]=\underline{\Gamma}^{t_{n},x}\left(\text{\ensuremath{\varphi}}\right)\vee\hat{\lambda}_{j}^{t_{n},x}\left(\varphi\right)\wedge\overline{\Gamma}^{t_{n},x}\left(\varphi\right)\label{eq:phi(x)hat}
\end{eqnarray}
where (recalling equation \ref{eq:condExpBounds}) $\underline{\Gamma}^{t_{n},x}\left(\text{\ensuremath{\varphi}}\right)$
and $\overline{\Gamma}^{t_{n},x}\left(\varphi\right)$ are lower and
upper bounds on $\Phi_{j}^{t_{n},x}\left(\varphi\right)$:
\[
\underline{\Gamma}^{t_{n},x}\left(\text{\ensuremath{\varphi}}\right)\leq\Phi_{j}^{t_{n},x}\left(\varphi\right)\leq\overline{\Gamma}^{t_{n},x}\left(\varphi\right)
\]

\begin{rem}
\label{rk:Recurrences}These definitions are useful to express the
dynamic programming equations \eqref{eq:DDP-approx-3}, \eqref{eq:DDP-approx-4}
and \eqref{eq:DDP-approx-5}. Indeed, these equations become:
\begin{eqnarray*}
\bar{v}_{\Pi}\left(T,x,i\right) & = & g\left(T,x,i\right)\\
\bar{v}_{\Pi}\left(t_{n},x,i\right) & = & \max_{j\in\mathbb{I}_{q}}\left\{ hf\left(t_{n},x,j\right)-k\left(t_{n},i,j\right)+\Phi_{j}^{t_{n},x}\left(\bar{v}_{\Pi}\right)\right\} \,,\,\, n=N-1,\ldots,0\\
\\
\tilde{v}_{\Pi}\left(T,x,i\right) & = & g\left(T,x,i\right)\\
\tilde{v}_{\Pi}\left(t_{n},x,i\right) & = & \max_{j\in\mathbb{I}_{q}}\left\{ hf\left(t_{n},x,j\right)-k\left(t_{n},i,j\right)+\tilde{\Phi}_{j}^{t_{n},x}\left(\tilde{v}_{\Pi}\right)\right\} \,,\,\, n=N-1,\ldots,0\\
\\
\hat{v}_{\Pi}\left(T,x,i\right) & = & g\left(T,x,i\right)\\
\hat{v}_{\Pi}\left(t_{n},x,i\right) & = & \max_{j\in\mathbb{I}_{q}}\left\{ hf\left(t_{n},x,j\right)-k\left(t_{n},i,j\right)+\hat{\Phi}_{j}^{t_{n},x}\left(\hat{v}_{\Pi}\right)\right\} \,,\,\, n=N-1,\ldots,0
\end{eqnarray*}

\end{rem}

\begin{rem}
\label{rk:RegrBound}For $\varphi=\bar{v}_{\Pi}$, we can easily explicit
bounding functions $\underline{\Gamma}^{t_{n},x}\left(\bar{v}_{\Pi}\right)$
and $\overline{\Gamma}^{t_{n},x}\left(\bar{v}_{\Pi}\right)$ of $\Phi_{j}^{t_{n},x}\left(\bar{v}_{\Pi}\right)$.
Indeed, using the growth conditions on $f$ and $g$, the nonnegativity
of $k$ and the definition of $C\left(T,\varepsilon\right)$ (see
Paragraph \ref{sub:SpaceLocIntro}), there exists $C>0$ such that
$\forall\left(t_{n},x,j\right)\in\Pi\tilde{\times}\mathcal{D}_{\Pi}^{\varepsilon}\times\mathbb{I}_{q}$:
\begin{eqnarray}
\left|\bar{v}_{\Pi}\left(t_{n},x,j\right)\right| & \leq & Ce^{-\rho t_{n}}\left(1+C\left(T,\varepsilon\right)\right)\label{eq:vBarPiBound}\\
\left|\Phi_{j}^{t_{n},x}\left(\bar{v}_{\Pi}\right)\right| & \leq & \Gamma^{t_{n}}\left(\bar{v}_{\Pi}\right):=Ce^{-\rho t_{n}}\left(1+C\left(T,\varepsilon\right)\right)\label{eq:regrBound}
\end{eqnarray}
Moreover, the same is true for $\varphi=\tilde{v}_{\Pi}$: there exists
$C>0$ such that $\forall\left(t_{n},x,j\right)\in\Pi\tilde{\times}\mathcal{D}_{\Pi}^{\varepsilon}\times\mathbb{I}_{q}$:
\begin{eqnarray}
\left|\tilde{v}_{\Pi}\left(t_{n},x,j\right)\right| & \leq & Ce^{-\rho t_{n}}\left(1+C\left(T,\varepsilon\right)\right)\label{eq:vTildePiBound}\\
\left|\tilde{\Phi}_{j}^{t_{n},x}\left(\tilde{v}_{\Pi}\right)\right| & \leq & \Gamma^{t_{n}}\left(\tilde{v}_{\Pi}\right):=Ce^{-\rho t_{n}}\left(1+C\left(T,\varepsilon\right)\right)\label{eq:regrTildeBound}
\end{eqnarray}
Finally, we impose the same bound for the definition of $\hat{v}_{\Pi}$,
i.e. $\Gamma^{t_{n}}\left(\hat{v}_{\Pi}\right):=\Gamma^{t_{n}}\left(\bar{v}_{\Pi}\right)$.
\end{rem}

Now we can start the assessment of the regression error.
\begin{lem}
\label{lem:LipschitzEuler}Consider a measurable function $\varphi:\Pi\times\mathbb{R}^{d}\times\mathbb{I}_{q}\rightarrow\mathbb{R}$.
Suppose that, for a fixed $t_{n+1}\in\Pi$, it is Lipschitz with constant
$C_{n+1}$, uniformly in $j$: $\forall\left(x_{1},x_{2},j\right)\in\mathbb{R}^{d}\times\mathbb{R}^{d}\times\mathbb{I}_{q}$

\begin{equation}
\left|\varphi\left(t_{n+1},x_{1},j\right)-\varphi\left(t_{n+1},x_{2},j\right)\right|\leq C_{n+1}\left|x_{1}-x_{2}\right|\label{eq:VarphiLipschitz}
\end{equation}

Then $\Phi_{j}^{t_{n},x}\left(\varphi\right)$ is Lipschitz with constant
$C_{n+1}\left(1+Lh\right)$, uniformly in $j$, where $L:=C_{b}+\frac{C_{\sigma}^{2}}{2}>0$.\end{lem}
\begin{proof}
Choose $\left(t_{n},j,x_{1},x_{2}\right)\in\Pi\times\mathbb{I}_{q}\times\mathbb{R}^{d}\times\mathbb{R}^{d}$.
Then:
\begin{eqnarray*}
\left|\Phi_{j}^{t_{n},x_{1}}\left(\varphi\right)-\Phi_{j}^{t_{n},x_{2}}\left(\varphi\right)\right| & = & \left|\mathbb{E}\left[\varphi\left(t_{n+1},\bar{X}_{t_{n+1}}^{t_{n},x_{1}},j\right)-\varphi\left(t_{n+1},\bar{X}_{t_{n+1}}^{t_{n},x_{2}},j\right)\right]\right|\\
 & \leq & \left\Vert \varphi\left(t_{n+1},\bar{X}_{t_{n+1}}^{t_{n},x_{1}},j\right)-\varphi\left(t_{n+1},\bar{X}_{t_{n+1}}^{t_{n},x_{2}},j\right)\right\Vert _{L_{1}}\\
 & \leq & \left\Vert \varphi\left(t_{n+1},\bar{X}_{t_{n+1}}^{t_{n},x_{1}},j\right)-\varphi\left(t_{n+1},\bar{X}_{t_{n+1}}^{t_{n},x_{2}},j\right)\right\Vert _{L_{2}}
\end{eqnarray*}
Now, using equations \eqref{eq:VarphiLipschitz} and \eqref{eq:SDE-Euler},
and $G$ denoting a $d$-dimensional standard Gaussian random variable,
we have
\begin{eqnarray*}
 &  & \mathbb{E}\left[\left(\varphi\left(t_{n+1},\bar{X}_{t_{n+1}}^{t_{n},x_{1}},j\right)-\varphi\left(t_{n+1},\bar{X}_{t_{n+1}}^{t_{n},x_{2}},j\right)\right)^{2}\right]\\
 & \leq & C_{n+1}^{2}\mathbb{E}\left[\left(\bar{X}_{t_{n+1}}^{t_{n},x_{1}}-\bar{X}_{t_{n+1}}^{t_{n},x_{2}}\right)^{2}\right]\\
 & \leq & C_{n+1}^{2}\mathbb{E}\left[\left(x_{1}-x_{2}+h\left(b\left(t_{n},x_{1}\right)-b\left(t_{n},x_{2}\right)\right)+\sqrt{h}\left(\sigma\left(t_{n},x_{1}\right)-\sigma\left(t_{n},x_{2}\right)\right)G\right)^{2}\right]\\
 & = & C_{n+1}^{2}\left\{ \left(x_{1}-x_{2}+h\left(b\left(t_{n},x_{1}\right)-b\left(t_{n},x_{2}\right)\right)\right)^{2}+h\mathbb{E}\left[\left(\left(\sigma\left(t_{n},x_{1}\right)-\sigma\left(t_{n},x_{2}\right)\right)G\right)^{2}\right]\right\} \\
 & \leq & C_{n+1}^{2}\left(x_{1}-x_{2}\right)^{2}\left\{ 1+\left(2C_{b}+C_{\sigma}^{2}\right)h+C_{b}^{2}h^{2}\right\} \\
 & \leq & C_{n+1}^{2}\left(x_{1}-x_{2}\right)^{2}\left(C_{b}+\frac{C_{\sigma}^{2}}{2}\right)^{2}\left(\frac{1}{C_{b}+\frac{C_{\sigma}^{2}}{2}}+h\right)^{2}\,.
\end{eqnarray*}
Thus:
\[
\left|\Phi_{j}^{t_{n},x_{1}}\left(\varphi\right)-\Phi_{j}^{t_{n},x_{2}}\left(\varphi\right)\right|\leq C_{n+1}\left(1+\left(C_{b}+\frac{C_{\sigma}^{2}}{2}\right)h\right)\left|x_{1}-x_{2}\right|
\]
\end{proof}
\begin{lem}
\label{lem:RegApprox}Consider again a function $\varphi:\Pi\times\mathbb{R}^{d}\times\mathbb{I}_{q}\rightarrow\mathbb{R}$
such that \eqref{eq:VarphiLipschitz} holds for a given $t_{n+1}\in\Pi$.
Then, $\forall\left(x,j\right)\in\mathcal{D}_{t_{n}}^{\varepsilon}\times\mathbb{I}_{q}$:
\[
\left|\tilde{\lambda}_{j}^{t_{n},x}\left(\varphi\right)-\Phi_{j}^{t_{n},x}\left(\varphi\right)\right|\leq C_{n+1}\delta\left(1+Lh\right)\,.
\]
In particular:
\begin{equation}
\left|\tilde{\Phi}_{j}^{t_{n},x}\left(\varphi\right)-\Phi_{j}^{t_{n},x}\left(\varphi\right)\right|\leq C_{n+1}\delta\left(1+Lh\right)\label{eq:RegApprox}
\end{equation}
\end{lem}
\begin{proof}
Recalling the definitions of $B_{t_{n}}\left(x\right)$, of $\tilde{\lambda}_{j}^{t_{n},x}\left(\varphi\right)$
(equation \eqref{eq:lambda(x)tilde}) and of $\Phi_{j}^{t_{n},x}\left(\varphi\right)$
(equation \eqref{eq:phi(x)}), simply remark that:
\begin{eqnarray*}
\min_{\tilde{x}\in B_{t_{n}}\left(x\right)}\Phi_{j}^{t_{n},\tilde{x}}\left(\varphi\right) & \leq\Phi_{j}^{t_{n},x}\left(\varphi\right) & \leq\max_{\tilde{x}\in B_{t_{n}}\left(x\right)}\Phi_{j}^{t_{n},\tilde{x}}\left(\varphi\right)\\
\min_{\tilde{x}\in B_{t_{n}}\left(x\right)}\Phi_{j}^{t_{n},\tilde{x}}\left(\varphi\right) & \leq\tilde{\lambda}_{j}^{t_{n},x}\left(\varphi\right) & \leq\max_{\tilde{x}\in B_{t_{n}}\left(x\right)}\Phi_{j}^{t_{n},\tilde{x}}\left(\varphi\right)\,.
\end{eqnarray*}
Now, using Lemma \ref{lem:LipschitzEuler}:
\begin{eqnarray*}
\left|\tilde{\lambda}_{j}^{t_{n},x}\left(\varphi\right)-\Phi_{j}^{t_{n},x}\left(\varphi\right)\right| & \leq & \max_{\tilde{x}\in B_{t_{n}}\left(x\right)}\Phi_{j}^{t_{n},\tilde{x}}\left(\varphi\right)-\min_{\tilde{x}\in B_{t_{n}}\left(x\right)}\Phi_{j}^{t_{n},\tilde{x}}\left(\varphi\right)\\
 & \leq & C_{n+1}\left(1+Lh\right)\max_{\left(x_{1},x_{2}\right)\in B_{t_{n}}\left(x\right)^{2}}\left|x_{1}-x_{2}\right|\\
 & \leq & C_{n+1}\left(1+Lh\right)\delta
\end{eqnarray*}
\end{proof}
\begin{lem}
\label{lem:vLipschitz}$\forall\left(t_{n},x_{1},x_{2},i\right)\in\Pi\times\left(\mathbb{R}^{d}\right)^{2}\times\mathbb{I}_{q}$:
\begin{equation}
\left|\bar{v}_{\Pi}\left(t_{n},x_{1},i\right)-\bar{v}_{\Pi}\left(t_{n},x_{2},i\right)\right|\leq C_{n}\left|x_{1}-x_{2}\right|\label{eq:vbarLip}
\end{equation}
where:
\begin{eqnarray}
C_{N} & = & e^{-\rho t_{N}}C_{g}\nonumber \\
C_{n} & = & hC_{f}e^{-\rho t_{n}}+C_{n+1}\left(1+Lh\right)\,,\,\, n=N-1,\ldots,0\label{eq:LipCoefRec}
\end{eqnarray}
In particular, $\exists C>0$ such that $\forall n=0,1,\ldots,N$:
\begin{equation}
C_{n}\leq Ce^{-\rho t_{n}}e^{L\left(T-t_{n}\right)}\label{eq:LipUBound}
\end{equation}
\end{lem}
\begin{proof}
Recall Remark \ref{rk:Recurrences}. We prove the lemma by induction.
First, remark that, using hypothesis \eqref{eq:gLipschitz}, it holds
for $n=N$. Now, suppose that it holds for some $(n+1)\in\left[1,\ldots,N\right]$.
Then, using Lemma \ref{lem:LipschitzEuler}:
\begin{align*}
 & \bar{v}_{\Pi}\left(t_{n},x_{1},i\right)\\
 & =\max_{j\in\mathbb{I}_{q}}\left\{ hf\!\left(t_{n},x_{1},j\right)-k\!\left(t_{n},i,j\right)+\Phi_{j}^{t_{n},x_{1}}\!\left(\bar{v}_{\Pi}\right)\right\} \\
 & =\max_{j\in\mathbb{I}_{q}}\left\{ hf\!\left(t_{n},x_{2},j\right)-k\!\left(t_{n},i,j\right)+\Phi_{j}^{t_{n},x_{2}}\!\left(\bar{v}_{\Pi}\right)+h\!\left(f\!\left(t_{n},x_{1},j\right)\!-\! f\!\left(t_{n},x_{2},j\right)\right)\!+\!\left(\Phi_{j}^{t_{n},x_{1}}\!\left(\bar{v}_{\Pi}\right)-\Phi_{j}^{t_{n},x_{2}}\!\left(\bar{v}_{\Pi}\right)\right)\right\} \\
 & \leq\max_{j\in\mathbb{I}_{q}}\left\{ hf\!\left(t_{n},x_{2},j\right)-k\!\left(t_{n},i,j\right)+\Phi_{j}^{t_{n},x_{2}}\!\left(\bar{v}_{\Pi}\right)+he^{-\rho t_{n}}C_{f}\left|x_{1}-x_{2}\right|+C_{n+1}\left(1+Lh\right)\left|x_{1}-x_{2}\right|\right\} \\
 & =\bar{v}_{\Pi}\left(t_{n},x_{2},i\right)+\left(he^{-\rho t_{n}}C_{f}+C_{n+1}\left(1+Lh\right)\right)\left|x_{1}-x_{2}\right|
\end{align*}
Symmetrically, the same inequality holds for $\bar{v}_{\Pi}\left(t_{n},x_{2},i\right)-\bar{v}_{\Pi}\left(t_{n},x_{1},i\right)$,
yielding equations \eqref{eq:vbarLip} and \eqref{eq:LipCoefRec}.
Finally, use the discrete version of Gronwall's inequality to obtain
equation \eqref{eq:LipUBound}
\end{proof}

\medskip{}

\begin{prop}
\label{pro:RegConv} $\exists C>0$ s.t. $\forall\left(t,x,i\right)\in\Pi\times\mathbb{R}^{d}\times\mathbb{I}_{q}$
:
\[
\left|\bar{v}_{\Pi}\left(t,x,i\right)-\tilde{v}_{\Pi}\left(t,x,i\right)\right|\leq C\frac{\delta}{h}e^{-\rho t}\,.
\]
\end{prop}
\begin{proof}
For each $t_{n}\in\Pi$, we look for an upper bound $E_{n}$, independent
of $x$ and $i$, of the quantity $\left|\bar{v}_{\Pi}\left(t_{n},x,i\right)-\tilde{v}_{\Pi}\left(t_{n},x,i\right)\right|$.
First:
\[
\left|\bar{v}_{\Pi}\left(T,x,i\right)-\tilde{v}_{\Pi}\left(T,x,i\right)\right|=\left|g\left(T,x,i\right)-g\left(T,x,i\right)\right|=0
\]
Hence $E_{N}=0$. Fix now $n\in\left[0,N-1\right]$. Using Remark
\ref{rk:Recurrences}:
\begin{eqnarray*}
\tilde{v}_{\Pi}\left(t_{n},x,i\right) & = & \max_{j\in\mathbb{I}_{q}}\left\{ hf\left(t_{n},x,j\right)-k\left(t_{n},i,j\right)+\tilde{\Phi}_{j}^{t_{n},x}\left(\tilde{v}_{\Pi}\right)\right\} \\
 & = & \max_{j\in\mathbb{I}_{q}}\left\{ hf\left(t_{n},x,j\right)-k\left(t_{n},i,j\right)+\Phi_{j}^{t_{n},x}\left(\bar{v}_{\Pi}\right)\right.\\
 &  & \qquad\;\left.+\tilde{\Phi}_{j}^{t_{n},x}\left(\bar{v}_{\Pi}\right)-\Phi_{j}^{t_{n},x}\left(\bar{v}_{\Pi}\right)\right.\\
 &  & \qquad\;\left.+\tilde{\Phi}_{j}^{t_{n},x}\left(\tilde{v}_{\Pi}\right)-\tilde{\Phi}_{j}^{t_{n},x}\left(\bar{v}_{\Pi}\right)\right\} 
\end{eqnarray*}
Using Lemmas \ref{lem:RegApprox} and \ref{lem:vLipschitz}, $\tilde{\Phi}_{j}^{t_{n},x}\left(\bar{v}_{\Pi}\right)-\Phi_{j}^{t_{n},x}\left(\bar{v}_{\Pi}\right)\leq C_{n+1}\delta\left(1+Lh\right)$
where $C_{n+1}$ is the Lipschitz constant of $\bar{v}_{\Pi}$ at
time $t_{n+1}$ (see Lemma \ref{lem:vLipschitz}). Moreover,
\begin{eqnarray*}
\tilde{\Phi}_{j}^{t_{n},x}\left(\tilde{v}_{\Pi}\right)-\tilde{\Phi}_{j}^{t_{n},x}\left(\bar{v}_{\Pi}\right) & \leq & \mathbb{E}\left[\left.\tilde{v}_{\Pi}\left(t_{n+1},\bar{X}_{t_{n+1}},j\right)-\bar{v}_{\Pi}\left(t_{n+1},\bar{X}_{t_{n+1}},j\right)\right|X_{t_{n}}\in B_{t_{n}}\left(x\right)\right]\\
 & \leq & E_{n+1}\,.
\end{eqnarray*}
Hence:
\[
\tilde{v}_{\Pi}\left(t_{n},x,i\right)\leq\bar{v}_{\Pi}\left(t_{n},x,i\right)+C_{n+1}\delta\left(1+Lh\right)+E_{n+1}
\]
Symmetrically, the same inequality holds for $\bar{v}_{\Pi}\left(T,x,i\right)-\tilde{v}_{\Pi}\left(t_{n},x,i\right)$,
leading to:
\[
\left|\bar{v}_{\Pi}\left(t_{n},x,i\right)-\tilde{v}_{\Pi}\left(t_{n},x,i\right)\right|\leq E_{n}
\]
where:\vspace{-5mm}
\begin{eqnarray*}
E_{N} & = & 0\\
E_{n} & = & C_{n+1}\delta\left(1+Lh\right)+E_{n+1}\,.
\end{eqnarray*}
Consequently, using equation \eqref{eq:LipUBound}:\vspace{-2mm}
\[
E_{n}=\delta\left(1+Lh\right)\sum_{k=n+1}^{N}C_{k}\leq C\frac{\delta}{h}e^{-\rho t_{n}}
\]
where $C>0$ does not depend on $t_{n}$ nor $T$.
\end{proof}

The following lemma measures the regression error. It is an extension
of Lemma 3.8 in \cite{Tan11} (itself inspired by Theorem 5.1 in \cite{Bouchard04}). 
\begin{lem}
\label{lem:RegressionM}Consider a measurable function $\varphi:\Pi\times\mathbb{R}^{d}\times\mathbb{I}_{q}\rightarrow\mathbb{R}$.
For any $p\geq1$, there exists $C_{p}\geq0$ such that $\forall\left(t_{n},l,j\right)\in\Pi\times\left[1,M\right]\times\mathbb{I}_{q}$:
\begin{equation}
\left\Vert \hat{\Phi}_{j}^{t_{n},\bar{X}_{t_{n}}^{l}}\!\left(\varphi\right)-\tilde{\Phi}_{j}^{t_{n},\bar{X}_{t_{n}}^{l}}\!\left(\varphi\right)\right\Vert _{L_{p}}\leq\frac{C_{p}}{\sqrt{M}}\frac{\Gamma^{t_{n}}\left(\varphi\right)+\bar{\varphi}^{t_{n}}}{\mathbb{P}\left(\bar{X}_{t_{n}}\!\in B_{t_{n}}\!\left(\bar{X}_{t_{n}}^{l}\right)\right)^{1-\frac{1}{p\vee2}}}+\frac{C_{p}}{M}\frac{\bar{\varphi}^{t_{n}}}{\mathbb{P}\left(\bar{X}_{t_{n}}\!\in B_{t_{n}}\!\left(\bar{X}_{t_{n}}^{l}\right)\right)}\label{eq:RegressionM}
\end{equation}

where $\bar{\varphi}^{t_{n}}\in\mathbb{R}_{+}$ is such that $\left|\varphi\left(t_{n+1},\bar{X}_{t_{n+1}},j\right)\right|\leq\bar{\varphi}^{t_{n}}$
a.s. .\end{lem}
\begin{proof}
Define the following centered random variables:
\begin{align*}
\varepsilon_{j}^{t_{n},\bar{X}_{t_{n}}^{l}}\left(\varphi\right) & :=\frac{1}{M}\sum_{m=1}^{M}\varphi\!\left(t_{n+1},\bar{X}_{t_{n+1}}^{m},j\right)\!\mathbf{1}\!\left\{ \bar{X}_{t_{n}}^{m}\in B_{t_{n}}\!\!\left(\bar{X}_{t_{n}}^{l}\right)\right\} -\mathbb{E}\!\left[\varphi\!\left(t_{n+1},\bar{X}_{t_{n+1}}^{m},j\right)\!\mathbf{1}\!\left\{ \bar{X}_{t_{n}}^{m}\in B_{t_{n}}\!\!\left(\bar{X}_{t_{n}}^{l}\right)\right\} \right]\\
\varepsilon^{t_{n},\bar{X}_{t_{n}}^{l}}\left(1\right) & :=\frac{1}{M}\sum_{m=1}^{M}\mathbf{1}\left\{ \bar{X}_{t_{n}}^{m}\in B_{t_{n}}\left(\bar{X}_{t_{n}}^{l}\right)\right\} -\mathbb{P}\left(\bar{X}_{t_{n}}^{m}\in B_{t_{n}}\left(\bar{X}_{t_{n}}^{l}\right)\right)
\end{align*}
Then:
\begin{align*}
 & \left|\hat{\Phi}_{j}^{t_{n},\bar{X}_{t_{n}}^{l}}\left(\varphi\right)-\tilde{\Phi}_{j}^{t_{n},\bar{X}_{t_{n}}^{l}}\left(\varphi\right)\right|=\left|\hat{\Phi}_{j}^{t_{n},\bar{X}_{t_{n}}^{l}}\left(\varphi\right)-\tilde{\Phi}_{j}^{t_{n},\bar{X}_{t_{n}}^{l}}\left(\varphi\right)\right|\wedge2\Gamma^{t_{n}}\left(\varphi\right)\\
 & \leq\left|\hat{\Phi}_{j}^{t_{n},\bar{X}_{t_{n}}^{l}}\left(\varphi\right)-\tilde{\Phi}_{j}^{t_{n},\bar{X}_{t_{n}}^{l}}\left(\varphi\right)\right|\mathbf{1}\left\{ \frac{\left|\varepsilon^{t_{n},\bar{X}_{t_{n}}^{l}}\left(1\right)\right|}{\mathbb{P}\left(\bar{X}_{t_{n}}\in B_{t_{n}}\left(\bar{X}_{t_{n}}^{l}\right)\right)}\leq\frac{1}{2}\right\} +2\Gamma^{t_{n}}\left(\varphi\right)\mathbf{1}\left\{ \frac{\left|\varepsilon^{t_{n},\bar{X}_{t_{n}}^{l}}\left(1\right)\right|}{\mathbb{P}\left(\bar{X}_{t_{n}}\in B_{t_{n}}\left(\bar{X}_{t_{n}}^{l}\right)\right)}\!>\!\frac{1}{2}\right\} 
\end{align*}
and:
\begin{align*}
 & \left|\hat{\Phi}_{j}^{t_{n},\bar{X}_{t_{n}}^{l}}\left(\varphi\right)-\tilde{\Phi}_{j}^{t_{n},\bar{X}_{t_{n}}^{l}}\left(\varphi\right)\right|\mathbf{1}\left\{ \frac{\left|\varepsilon^{t_{n},\bar{X}_{t_{n}}^{l}}\left(1\right)\right|}{\mathbb{P}\left(\bar{X}_{t_{n}}\in B_{t_{n}}\left(\bar{X}_{t_{n}}^{l}\right)\right)}\leq\frac{1}{2}\right\} \hspace{80mm}
\end{align*}
\begin{align*}
 & =\left|\hat{\Phi}_{j}^{t_{n},\bar{X}_{t_{n}}^{l}}\left(\varphi\right)-\tilde{\Phi}_{j}^{t_{n},\bar{X}_{t_{n}}^{l}}\left(\varphi\right)\frac{\mathbb{P}\left(\bar{X}_{t_{n}}\in B_{t_{n}}\left(\bar{X}_{t_{n}}^{l}\right)\right)}{\frac{1}{M}\sum_{m=1}^{M}\mathbf{1}\left\{ \bar{X}_{t_{n}}^{m}\in B_{t_{n}}\left(\bar{X}_{t_{n}}^{l}\right)\right\} }-\right.\\
 & \left.\tilde{\Phi}_{j}^{t_{n},\bar{X}_{t_{n}}^{l}}\left(\varphi\right)\frac{\varepsilon^{t_{n},\bar{X}_{t_{n}}^{l}}\left(1\right)}{\frac{1}{M}\sum_{m=1}^{M}\mathbf{1}\left\{ \bar{X}_{t_{n}}^{m}\in B_{t_{n}}\left(\bar{X}_{t_{n}}^{l}\right)\right\} }\right|\mathbf{1}\left\{ \frac{\left|\varepsilon^{t_{n},\bar{X}_{t_{n}}^{l}}\left(1\right)\right|}{\mathbb{P}\left(\bar{X}_{t_{n}}\in B_{t_{n}}\left(\bar{X}_{t_{n}}^{l}\right)\right)}\leq\frac{1}{2}\right\} \\
 & \leq\left\{ \frac{\left|\varepsilon_{j}^{t_{n},\bar{X}_{t_{n}}^{l}}\left(\varphi\right)\right|}{\frac{1}{M}\sum_{m=1}^{M}\mathbf{1}\left\{ \bar{X}_{t_{n}}^{m}\in B_{t_{n}}\left(\bar{X}_{t_{n}}^{l}\right)\right\} }\wedge3\Gamma^{t_{n}}\left(\varphi\right)+\right.\\
 & \left.\left|\tilde{\Phi}_{j}^{t_{n},\bar{X}_{t_{n}}^{l}}\left(\varphi\right)\right|\frac{\left|\varepsilon^{t_{n},\bar{X}_{t_{n}}^{l}}\left(1\right)\right|}{\frac{1}{M}\sum_{m=1}^{M}\mathbf{1}\left\{ \bar{X}_{t_{n}}^{m}\in B_{t_{n}}\left(\bar{X}_{t_{n}}^{l}\right)\right\} }\right\} \mathbf{1}\left\{ \frac{\left|\varepsilon^{t_{n},\bar{X}_{t_{n}}^{l}}\left(1\right)\right|}{\mathbb{P}\left(\bar{X}_{t_{n}}\in B_{t_{n}}\left(\bar{X}_{t_{n}}^{l}\right)\right)}\leq\frac{1}{2}\right\} \\
 & \leq\frac{2}{\mathbb{P}\left(\bar{X}_{t_{n}}\in B_{t_{n}}\left(\bar{X}_{t_{n}}^{l}\right)\right)}\left\{ \left|\varepsilon_{j}^{t_{n},\bar{X}_{t_{n}}^{l}}\left(\varphi\right)\right|\wedge5\Gamma^{t_{n}}\left(\varphi\right)+\left|\varepsilon^{t_{n},\bar{X}_{t_{n}}^{l}}\left(1\right)\right|\Gamma^{t_{n}}\left(\varphi\right)\right\} \mathbf{1}\left\{ \frac{\left|\varepsilon^{t_{n},\bar{X}_{t_{n}}^{l}}\left(1\right)\right|}{\mathbb{P}\left(\bar{X}_{t_{n}}\in B_{t_{n}}\left(\bar{X}_{t_{n}}^{l}\right)\right)}\leq\frac{1}{2}\right\} 
\end{align*}
Now, for any $p\geq1$:
\begin{align*}
 & \left|\hat{\Phi}_{j}^{t_{n},\bar{X}_{t_{n}}^{l}}\left(\varphi\right)-\tilde{\Phi}_{j}^{t_{n},\bar{X}_{t_{n}}^{l}}\left(\varphi\right)\right|^{p}\\
 & \leq\frac{2^{3p-2}}{\mathbb{P}\left(\bar{X}_{t_{n}}\in B_{t_{n}}\left(\bar{X}_{t_{n}}^{l}\right)\right)^{p}}\left\{ \left\{ \left|\varepsilon_{j}^{t_{n},\bar{X}_{t_{n}}^{l}}\left(\varphi\right)\right|\wedge5\Gamma^{t_{n}}\left(\varphi\right)\right\} ^{p}+\left\{ \left|\varepsilon^{t_{n},\bar{X}_{t_{n}}^{l}}\left(1\right)\right|\Gamma^{t_{n}}\left(\varphi\right)\right\} ^{p}\right\} \times\\
 & \mathbf{1}\left\{ \frac{\left|\varepsilon^{t_{n},\bar{X}_{t_{n}}^{l}}\left(1\right)\right|}{\mathbb{P}\left(\bar{X}_{t_{n}}\in B_{t_{n}}\left(\bar{X}_{t_{n}}^{l}\right)\right)}\leq\frac{1}{2}\right\} +2^{2p-1}\left(\Gamma^{t_{n}}\left(\varphi\right)\right)^{p}\mathbf{1}\left\{ \frac{\left|\varepsilon^{t_{n},\bar{X}_{t_{n}}^{l}}\left(1\right)\right|}{\mathbb{P}\left(\bar{X}_{t_{n}}\in B_{t_{n}}\left(\bar{X}_{t_{n}}^{l}\right)\right)}>\frac{1}{2}\right\} 
\end{align*}
and:
\begin{align}
 & \mathbb{E}\left[\left|\hat{\Phi}_{j}^{t_{n},\bar{X}_{t_{n}}^{l}}\left(\varphi\right)-\tilde{\Phi}_{j}^{t_{n},\bar{X}_{t_{n}}^{l}}\left(\varphi\right)\right|^{p}\right]\nonumber \\
 & \leq\frac{2^{3p-2}}{\mathbb{P}\left(\bar{X}_{t_{n}}\in B_{t_{n}}\left(\bar{X}_{t_{n}}^{l}\right)\right)^{p}}\left\{ \mathbb{E}\left[\left\{ \left|\varepsilon_{j}^{t_{n},\bar{X}_{t_{n}}^{l}}\left(\varphi\right)\right|\wedge5\Gamma^{t_{n}}\left(\varphi\right)\right\} ^{p}\right]+\left(\Gamma^{t_{n}}\left(\varphi\right)\right)^{p}\mathbb{E}\left[\left|\varepsilon^{t_{n},\bar{X}_{t_{n}}^{l}}\left(1\right)\right|^{p}\right]\right\} \nonumber \\
 & +2^{2p-1}\left(\Gamma^{t_{n}}\left(\varphi\right)\right)^{p}\mathbb{P}\left(\left|\varepsilon^{t_{n},\bar{X}_{t_{n}}^{l}}\left(1\right)\right|^{p}>\frac{\mathbb{P}\left(\bar{X}_{t_{n}}\in B_{t_{n}}\left(\bar{X}_{t_{n}}^{l}\right)\right)^{p}}{2^{p}}\right)\nonumber \\
 & \leq\frac{8^{p}}{\mathbb{P}\left(\bar{X}_{t_{n}}\!\in B_{t_{n}}\!\left(\bar{X}_{t_{n}}^{l}\right)\right)^{p}}\left\{ \mathbb{E}\left[\left\{ \left|\varepsilon_{j}^{t_{n},\bar{X}_{t_{n}}^{l}}\!\left(\varphi\right)\right|\wedge5\Gamma^{t_{n}}\left(\varphi\right)\right\} ^{p}\right]+\left\{ \Gamma^{t_{n}}\left(\varphi\right)\right\} ^{p}\mathbb{E}\left[\left|\varepsilon_{j}^{t_{n},\bar{X}_{t_{n}}^{l}}\!\left(1\right)\right|^{p}\right]\right\} \label{eq:regProofIneq1}
\end{align}
using Markov's inequality. Now, the following lemma will provide upper
bounds for $\mathbb{E}\left[\left|\varepsilon^{t_{n},\bar{X}_{t_{n}}^{l}}\left(1\right)\right|^{p}\right]$
and $\mathbb{E}\left[\left|\varepsilon_{j}^{t_{n},\bar{X}_{t_{n}}^{l}}\left(\varphi\right)\right|^{p}\right]$.
\begin{lem}
\label{lem:MZ}For every $p\geq1$, there exists $C_{p}>0$ such that
for any i.i.d. sample $X_{1},\ldots,X_{M}$ of $\mathbb{R}$-valued
random variables such that $\mathbb{E}\left[X_{1}\right]=0$ and $\mathbb{E}\left[\left|X_{1}\right|^{p\vee2}\right]<\infty$,
the following holds:
\begin{equation}
\left\Vert \frac{1}{M}\sum_{m=1}^{M}X_{m}\right\Vert _{L_{p}}\leq\frac{C_{p}}{\sqrt{M}}\left\Vert X_{1}\right\Vert _{L_{p\vee2}}\label{eq:MZ}
\end{equation}
\end{lem}
\begin{proof}
Using Marcinkiewicz-Zygmund's inequality, there exists $C_{p}>0$
such that:
\[
\mathbb{E}\left[\left|\sum_{m=1}^{M}X_{m}\right|^{p}\right]\leq C_{p}\mathbb{E}\left[\left(\sum_{m=1}^{M}\left|X_{m}\right|^{2}\right)^{\frac{p}{2}}\right]
\]
Multiplying both sides by $\frac{1}{M^{p}}$:
\begin{equation}
\mathbb{E}\left[\left|\frac{1}{M}\sum_{m=1}^{M}X_{m}\right|^{p}\right]\leq\frac{C_{p}}{M^{\frac{p}{2}}}\mathbb{E}\left[\left(\frac{1}{M}\sum_{m=1}^{M}\left|X_{m}\right|^{2}\right)^{\frac{p}{2}}\right]\label{eq:MZ1}
\end{equation}
If $p\geq2$, then $\frac{p}{2}\geq1$ and, using Jensen's inequality:
\[
\left(\frac{1}{M}\sum_{m=1}^{M}\left|X_{m}\right|^{2}\right)^{\frac{p}{2}}\leq\frac{1}{M}\sum_{m=1}^{M}\left(\left|X_{m}\right|^{2}\right)^{\frac{p}{2}}=\frac{1}{M}\sum_{m=1}^{M}\left|X_{m}\right|^{p}
\]
Taking expectations on both sides:\vspace{-2mm}
\begin{equation}
\mathbb{E}\left[\left(\frac{1}{M}\sum_{m=1}^{M}\left|X_{m}\right|^{2}\right)^{\frac{p}{2}}\right]\leq\mathbb{E}\left[\left|X_{1}\right|^{p}\right]\label{eq:MZ2}
\end{equation}
Now, if $p<2$, then $\frac{p}{2}<1$ and, using Jensen's inequality:
\begin{equation}
\mathbb{E}\left[\left(\frac{1}{M}\sum_{m=1}^{M}\left|X_{m}\right|^{2}\right)^{\frac{p}{2}}\right]\leq\mathbb{E}\left[\left(\frac{1}{M}\sum_{m=1}^{M}\left|X_{m}\right|^{2}\right)\right]^{\frac{p}{2}}=\mathbb{E}\left[\left|X_{1}\right|^{2}\right]^{\frac{p}{2}}\label{eq:MZ3}
\end{equation}
Then combine inequalities \eqref{eq:MZ1}, \eqref{eq:MZ2} and \eqref{eq:MZ3}
and take the power $\frac{1}{p}$ to obtain inequality \eqref{eq:MZ}.
\end{proof}
Now, suppose that $\exists\bar{\varphi}^{t_{n}}\in\mathbb{R}_{+}$
s.t. $\left|\varphi\left(t_{n+1},\bar{X}_{t_{n+1}},j\right)\right|\leq\bar{\varphi}^{t_{n}}$
a.s. . Then, using Lemma \ref{lem:MZ}, $\exists C_{p}>0$ such that:
\begin{eqnarray}
\mathbb{E}\left[\left|\varepsilon^{t_{n},\bar{X}_{t_{n}}^{l}}\left(1\right)\right|^{p}\right] & \leq & \frac{C_{p}}{M^{\frac{p}{2}}}\mathbb{E}\left[\left|\mathbf{1}\left\{ \bar{X}_{t_{n}}\!\!\in B_{t_{n}}\!\!\left(\bar{X}_{t_{n}}^{l}\right)\right\} -\mathbb{P}\left(\bar{X}_{t_{n}}\!\!\in B_{t_{n}}\!\!\left(\bar{X}_{t_{n}}^{l}\right)\right)\right|^{p\vee2}\right]^{\frac{p}{p\vee2}}\label{eq:regProofIneq2}\\
\mathbb{E}\left[\left|\varepsilon_{j}^{t_{n},\bar{X}_{t_{n}}^{l}}\left(\varphi\right)\right|^{p}\right] & \leq & C_{p}\left\{ \frac{\left(\bar{\varphi}^{t_{n}}\right)^{p}}{M^{p}}+\frac{1}{M^{\frac{p}{2}}}\mathbb{E}\left[\left|\varphi\!\left(t_{n\!+\!1},\bar{X}_{t_{n\!+\!1}},j\right)\mathbf{1}\!\left\{ \bar{X}_{t_{n}}\!\!\in\! B_{t_{n}}\!\!\left(\bar{X}_{t_{n}}^{l}\right)\right\} \right.\right.\right.\nonumber \\
 &  & \left.\left.\left.-\mathbb{E}\left[\varphi\!\left(t_{n\!+\!1},\bar{X}_{t_{n\!+\!1}},j\right)\mathbf{1}\!\left\{ \bar{X}_{t_{n}}\!\!\in\! B_{t_{n}}\!\!\left(\bar{X}_{t_{n}}^{l}\right)\right\} \right]\right|^{p\vee2}\right]^{\frac{p}{p\vee2}}\right\} \label{eq:regProofIneq3}
\end{eqnarray}
where, for the second inequality, the term $m=l$ in the sum was treated
separately. Then:
\begin{align}
 & \mathbb{E}\left[\left|\varphi\left(t_{n+1},\bar{X}_{t_{n+1}},j\right)\mathbf{1}\left\{ \bar{X}_{t_{n}}\!\!\in B_{t_{n}}\!\!\left(\bar{X}_{t_{n}}^{l}\right)\right\} -\mathbb{E}\left[\varphi\left(t_{n+1},\bar{X}_{t_{n+1}},j\right)\mathbf{1}\left\{ \bar{X}_{t_{n}}\!\!\in B_{t_{n}}\!\!\left(\bar{X}_{t_{n}}^{l}\right)\right\} \right]\right|^{p\vee2}\right]^{\frac{p}{p\vee2}}\nonumber \\
 & \leq\left(2^{p\vee2-1}\mathbb{E}\left[\left(\bar{\varphi}^{t_{n}}\right)^{p\vee2}\mathbf{1}\left\{ \bar{X}_{t_{n}}\!\!\in B_{t_{n}}\!\!\left(\bar{X}_{t_{n}}^{l}\right)\right\} +\mathbb{E}\left[\left(\bar{\varphi}^{t_{n}}\right)^{p\vee2}\mathbf{1}\left\{ \bar{X}_{t_{n}}\!\!\in B_{t_{n}}\!\!\left(\bar{X}_{t_{n}}^{l}\right)\right\} \right]\right]\right)^{\frac{p}{p\vee2}}\nonumber \\
 & \leq2^{p}\left(\bar{\varphi}^{t_{n}}\right)^{p}\mathbb{P}\left(\bar{X}_{t_{n}}\in B_{t_{n}}\left(\bar{X}_{t_{n}}^{l}\right)\right)^{\frac{p}{p\vee2}}\label{eq:regProofIneq4}
\end{align}
In a similar manner:
\begin{equation}
\mathbb{E}\left[\left|\mathbf{1}\left\{ \bar{X}_{t_{n}}\!\in B_{t_{n}}\!\left(\bar{X}_{t_{n}}^{l}\right)\right\} -\mathbb{P}\left(\bar{X}_{t_{n}}\!\in B_{t_{n}}\!\left(\bar{X}_{t_{n}}^{l}\right)\right)\right|^{p\vee2}\right]^{\frac{p}{p\vee2}}\leq2^{p}\mathbb{P}\left(\bar{X}_{t_{n}}\!\in B_{t_{n}}\!\left(\bar{X}_{t_{n}}^{l}\right)\right)^{\frac{p}{p\vee2}}\label{eq:regProofIneq5}
\end{equation}
Finally, the combination of inequalities \eqref{eq:regProofIneq1},
\eqref{eq:regProofIneq2}, \eqref{eq:regProofIneq3}, \eqref{eq:regProofIneq4}
and \eqref{eq:regProofIneq5} proves equation \eqref{eq:RegressionM}.
\end{proof}
We now apply Lemma \ref{lem:RegressionM} to $\bar{v}_{\Pi}$ in the
following Corollary:
\begin{cor}
\label{cor:RegressionMv}For every $p\geq1$, there exists $C_{p}\geq0$
s.t. $\forall\left(t_{n},l,j\right)\in\Pi\times\left[1,M\right]\times\mathbb{I}_{q}$:\textup{
\[
\left\Vert \hat{\Phi}_{j}^{t_{n},\bar{X}_{t_{n}}^{l}}\left(\tilde{v}_{\Pi}\right)-\tilde{\Phi}_{j}^{t_{n},\bar{X}_{t_{n}}^{l}}\left(\tilde{v}_{\Pi}\right)\right\Vert _{L_{p}}\leq C_{p}e^{-\rho t_{n}}\frac{1+C\left(T,\varepsilon\right)}{\sqrt{M}p\left(T,\delta,\varepsilon\right)^{1-\frac{1}{p\vee2}}}\left(1+\frac{1}{\sqrt{M}p\left(T,\delta,\varepsilon\right)^{\frac{1}{p\vee2}}}\right)
\]
}\end{cor}
\begin{proof}
First, recall from equation \eqref{eq:vTildePiBound} and \eqref{eq:regrTildeBound}
that there exists $C>0$ such that for every $\left(t_{n},j\right)\in\Pi\times\mathbb{I}_{q}$:
\begin{eqnarray*}
\Gamma_{j}^{t_{n}}\left(\tilde{v}_{\Pi}\right) & = & Ce^{-\rho t_{n}}\left(1+C\left(T,\varepsilon\right)\right)\\
\left|\tilde{v}_{\Pi}\left(t_{n+1},\bar{X}_{t_{n+1}},j\right)\right| & \leq & Ce^{-\rho t_{n}}\left(1+C\left(T,\varepsilon\right)\right)
\end{eqnarray*}
Hence one can apply Lemma \ref{lem:RegressionM} to $\tilde{v}_{\Pi}$
with these upper bounds. The final step is to recall that the minimum
probability $p\left(T,\delta,\varepsilon\right)$ defined in equation
\eqref{eq:probaMin} is a lower bound on $\mathbb{P}\left(\bar{X}_{t_{n}}\!\in B_{t_{n}}\!\left(\bar{X}_{t_{n}}^{l}\right)\right)$
for any $\left(t_{n},l\right)\in\Pi\times\left[1,M\right]$.
\end{proof}
Using this result, we can now assess the error between $\tilde{v}_{\Pi}$
and $\hat{v}_{\Pi}$.
\begin{prop}
\label{pro:RegMConv}$\forall p\geq1$, $\exists C_{p}>0$ s.t. $\forall\left(t_{n},l\right)\in\Pi\times\left[1,M\right]$
:
\[
\left\Vert \sup_{i\in\mathbb{I}_{q}^{t_{n}}}\left|\tilde{v}_{\Pi}\left(t,\bar{X}_{t_{n}}^{l},i\right)-\hat{v}_{\Pi}\left(t,\bar{X}_{t_{n}}^{l},i\right)\right|\right\Vert _{L_{p}}\leq C_{p}e^{-\rho t_{n}}\frac{1+C\left(T,\varepsilon\right)}{h\sqrt{M}p\left(T,\delta,\varepsilon\right)^{1-\frac{1}{p\vee2}}}\left(1+\frac{1}{\sqrt{M}p\left(T,\delta,\varepsilon\right)^{\frac{1}{p\vee2}}}\right)
\]

where $\mathbb{I}_{q}^{t_{n}}$ is the set of $\mathcal{F}_{t_{n}}$-measurable
random variables taking values in $\mathbb{I}_{q}$.\end{prop}
\begin{proof}
For each $t_{n}\in\Pi$, we look for an upper bound $E_{n}$, independent
of $l$, such that: 
\[
\left\Vert \sup_{i\in\mathbb{I}_{q}^{t_{n}}}\left|\tilde{v}_{\Pi}\left(t,\bar{X}_{t_{n}}^{l},i\right)-\hat{v}_{\Pi}\left(t,\bar{X}_{t_{n}}^{l},i\right)\right|\right\Vert _{L_{p}}\leq E_{n}\,\,.
\]
 First:
\[
\left\Vert \sup_{i\in\mathbb{I}_{q}^{T}}\left|\tilde{v}_{\Pi}\left(T,\bar{X}_{T}^{l},i\right)-\hat{v}_{\Pi}\left(T,\bar{X}_{T}^{l},i\right)\right|\right\Vert _{L_{p}}=\left\Vert \sup_{i\in\mathbb{I}_{q}^{T}}\left|g\left(T,\bar{X}_{T}^{l},i\right)-g\left(T,\bar{X}_{T}^{l},i\right)\right|\right\Vert _{L_{p}}=0
\]
Hence $E_{N}=0$. Fix now $n\in\left[0,N-1\right]$. Recall the dynamic
programming equations from Remark \ref{rk:Recurrences}, and, for
every $\left(i,l\right)\in\mathbb{I}_{q}^{t_{n}}\times\left[1,M\right]$,
introduce $\tilde{j}^{*}$ (resp. $\hat{j}^{*}$) the $\arg\max$
for $\tilde{v}_{\Pi}$ (resp. $\hat{v}_{\Pi}$) at point $\bar{X}_{t_{n}}^{l}$
, i.e.:
\begin{eqnarray*}
\tilde{v}_{\Pi}\left(t_{n},\bar{X}_{t_{n}}^{l},i\right) & = & hf\left(t_{n},\bar{X}_{t_{n}}^{l},\tilde{j}^{*}\right)-k\left(t_{n},i,\tilde{j}^{*}\right)+\tilde{\Phi}_{\tilde{j}^{*}}^{t_{n},\bar{X}_{t_{n}}^{l}}\left(\tilde{v}_{\Pi}\right)\\
\hat{v}_{\Pi}\left(t_{n},\bar{X}_{t_{n}}^{l},i\right) & = & hf\left(t_{n},\bar{X}_{t_{n}}^{l},\hat{j}^{*}\right)-k\left(t_{n},i,\hat{j}^{*}\right)+\hat{\Phi}_{\hat{j}^{*}}^{t_{n},\bar{X}_{t_{n}}^{l}}\left(\hat{v}_{\Pi}\right)
\end{eqnarray*}
Now:
\begin{eqnarray*}
\hat{v}_{\Pi}\left(t_{n},\bar{X}_{t_{n}}^{l},i\right) & = & hf\left(t_{n},\bar{X}_{t_{n}}^{l},\hat{j}^{*}\right)-k\left(t_{n},i,\hat{j}^{*}\right)+\hat{\Phi}_{\hat{j}^{*}}^{t_{n},\bar{X}_{t_{n}}^{l}}\left(\hat{v}_{\Pi}\right)\\
 & = & \left\{ hf\left(t_{n},\bar{X}_{t_{n}}^{l},\hat{j}^{*}\right)-k\left(t_{n},i,\hat{j}^{*}\right)+\tilde{\Phi}_{\hat{j}^{*}}^{t_{n},\bar{X}_{t_{n}}^{l}}\left(\tilde{v}_{\Pi}\right)\right\} \\
 &  & +\left\{ \hat{\Phi}_{\hat{j}^{*}}^{t_{n},\bar{X}_{t_{n}}^{l}}\left(\tilde{v}_{\Pi}\right)-\tilde{\Phi}_{\hat{j}^{*}}^{t_{n},\bar{X}_{t_{n}}^{l}}\left(\tilde{v}_{\Pi}\right)\right\} +\left\{ \hat{\Phi}_{\hat{j}^{*}}^{t_{n},\bar{X}_{t_{n}}^{l}}\left(\hat{v}_{\Pi}\right)-\hat{\Phi}_{\hat{j}^{*}}^{t_{n},\bar{X}_{t_{n}}^{l}}\left(\tilde{v}_{\Pi}\right)\right\} \\
 & \leq & \tilde{v}_{\Pi}\left(t_{n},\bar{X}_{t_{n}}^{l},i\right)+\sum_{j\in\mathbb{I}_{q}}\left|\hat{\Phi}_{j}^{t_{n},\bar{X}_{t_{n}}^{l}}\left(\tilde{v}_{\Pi}\right)-\tilde{\Phi}_{j}^{t_{n},\bar{X}_{t_{n}}^{l}}\left(\tilde{v}_{\Pi}\right)\right|\\
 &  & +\sup_{j\in\mathbb{I}_{q}^{t_{n}}}\left|\hat{\Phi}_{j}^{t_{n},\bar{X}_{t_{n}}^{l}}\left(\hat{v}_{\Pi}\right)-\hat{\Phi}_{j}^{t_{n},\bar{X}_{t_{n}}^{l}}\left(\tilde{v}_{\Pi}\right)\right|
\end{eqnarray*}
Symmetrically:
\begin{eqnarray*}
\tilde{v}_{\Pi}\left(t_{n},\bar{X}_{t_{n}}^{l},i\right) & \leq & \hat{v}_{\Pi}\left(t_{n},\bar{X}_{t_{n}}^{l},i\right)+\sum_{j\in\mathbb{I}_{q}}\left|\tilde{\Phi}_{j}^{t_{n},\bar{X}_{t_{n}}^{l}}\left(\tilde{v}_{\Pi}\right)-\hat{\Phi}_{j}^{t_{n},\bar{X}_{t_{n}}^{l}}\left(\tilde{v}_{\Pi}\right)\right|\\
 &  & +\sup_{j\in\mathbb{I}_{q}^{t_{n}}}\left|\hat{\Phi}_{j}^{t_{n},\bar{X}_{t_{n}}^{l}}\left(\tilde{v}_{\Pi}\right)-\hat{\Phi}_{j}^{t_{n},\bar{X}_{t_{n}}^{l}}\left(\hat{v}_{\Pi}\right)\right|
\end{eqnarray*}
Combining these two inequalities:
\begin{eqnarray*}
\sup_{i\in\mathbb{I}_{q}^{t_{n}}}\left|\tilde{v}_{\Pi}\left(t_{n},\bar{X}_{t_{n}}^{l},i\right)-\hat{v}_{\Pi}\left(t_{n},\bar{X}_{t_{n}}^{l},i\right)\right| & \leq & \sum_{j\in\mathbb{I}_{q}}\left|\hat{\Phi}_{j}^{t_{n},\bar{X}_{t_{n}}^{l}}\left(\tilde{v}_{\Pi}\right)-\tilde{\Phi}_{j}^{t_{n},\bar{X}_{t_{n}}^{l}}\left(\tilde{v}_{\Pi}\right)\right|\\
 &  & +\sup_{j\in\mathbb{I}_{q}^{t_{n}}}\left|\hat{\Phi}_{j}^{t_{n},\bar{X}_{t_{n}}^{l}}\left(\hat{v}_{\Pi}\right)-\hat{\Phi}_{j}^{t_{n},\bar{X}_{t_{n}}^{l}}\left(\tilde{v}_{\Pi}\right)\right|
\end{eqnarray*}
Hence, using the triangular inequality, Corollary \ref{cor:RegressionMv},
equation \eqref{eq:lambda(x)hat}, and the induction hypothesis:
\begin{eqnarray*}
\left\Vert \sup_{i\in\mathbb{I}_{q}^{t_{n}}}\left|\tilde{v}_{\Pi}\left(t_{n},\bar{X}_{t_{n}}^{l},i\right)-\hat{v}_{\Pi}\left(t_{n},\bar{X}_{t_{n}}^{l},i\right)\right|\right\Vert _{L_{p}} & \leq E_{n}:= & C_{p}e^{-\rho t_{n}}\frac{1+C\left(T,\varepsilon\right)}{\sqrt{M}p\left(T,\delta,\varepsilon\right)^{1-\frac{1}{p\vee2}}}\\
 &  & +C_{p}e^{-\rho t_{n}}\frac{1+C\left(T,\varepsilon\right)}{Mp\left(T,\delta,\varepsilon\right)}+E_{n+1}
\end{eqnarray*}
for some constant $C_{p}>0$ which depends only on $p$. Consequently:\vspace{-1mm}
\begin{eqnarray*}
E_{n} & \leq & C_{p}e^{-\rho t_{n}}\frac{1+C\left(T,\varepsilon\right)}{h\sqrt{M}p\left(T,\delta,\varepsilon\right)^{1-\frac{1}{p\vee2}}}\left(1+\frac{1}{\sqrt{M}p\left(T,\delta,\varepsilon\right)^{\frac{1}{p\vee2}}}\right)
\end{eqnarray*}
where $C_{p}>0$ depends only on $p$.
\end{proof}
Finally, the combination of Propositions \ref{pro:ErrorControl-v1-v2}
\ref{pro:Time-Convergence}, \ref{pro:Euler-Convergence}, \ref{pro:RegConv}
and \ref{pro:RegMConv} at time $t=t_{0}$ proves Theorem \ref{Th:CONVERGENCERATE}.

\section{Complexity analysis and memory reduction\label{sec:Complexity}}

\subsection{Complexity\label{sub:Complexity}}

\subsubsection{Computational complexity}

The number of operations required by the algorithm described below
is in $\mathcal{O}\!\left(q^{2}.N.M\right)$, where we recall that
$q$ is the number of possible switches, $N$ is the number of time
steps and $M$ is the number of Monte Carlo trajectories.
\begin{itemize}
\item The $q^{2}$ term stems from the fact that for every $i\in\mathbb{I}_{q}$,
one has to compute a maximum on $j\in\mathbb{I}_{q}$ (see equation
\eqref{eq:DDP-approx-5}). However, this $q^{2}$ can be reduced to
$q$ as soon as the two following conditions are satisfied:\end{itemize}
\begin{enumerate}
\item (Irreversibility) The controlled variable can only be increased (or,
symmetrically, can only be decreased)
\item (Cost Separability) There exists two functions $k_{1}$ and $k_{2}$
such that $\forall\left(t,i,j\right)\in\mathbb{R}_{+}\times\mathbb{I}_{q}\times\mathbb{I}_{q}$,
$k\left(t,i,j\right)=k_{1}\left(t,i\right)+k_{2}\left(t,j\right)$.
For instance, this is true of affine costs.
\end{enumerate}
Indeed, under those two conditions, equation \eqref{eq:DDP-approx-5}
becomes:
\[
\hat{v}_{\Pi}\left(t_{n},x,i\right)+k_{1}\left(t_{n},i\right)=\max_{j\in\mathbb{I}_{q},\, j\geq i}\left\{ hf\left(t_{n},x,j\right)-k_{2}\left(t_{n},j\right)+\hat{\mathbb{E}}\left[\hat{v}_{\Pi}\left(t_{n+1},\bar{X}_{t_{n+1}}^{t_{n},x},j\right)\right]\right\} \,,\,\, n=N-1,\ldots,0
\]

These maxima can be computed in $\mathcal{O}\!\left(q\right)$ instead
of $\mathcal{O}\!\left(q^{2}\right)$ by starting from the biggest
element $i=i_{q}$ down to the smallest element $i=i_{1}$ (in lexicographical
order) and keeping track of the partial maxima.

Note that these two conditions hold for the numerical application
from Section \ref{sec:Application}, providing the improved complexity
$\mathcal{O}\!\left(q.N.M\right)$.
\begin{itemize}
\item The $N$ term comes from the backward time induction.
\item The $M$ term corresponds to the cost of a regression, which, in the
case of a local basis, can be brought down to $\mathcal{O}\left(M\right)$
(cf. \cite{Bouchard11}).
\end{itemize}

\subsubsection{Memory complexity}

The memory size required for solving optimal switching problems (as
well as the simpler American option problems and the more complex
BSDE problems) by Monte Carlo methods is often said to be in $\mathcal{O}\!\left(N.M\right)$,
because, as the Euler scheme is a forward scheme and the dynamic programming
principle is a backward scheme, the storage of the Monte Carlo trajectories
seems inescapable. This fact is the major limitation of such methods,
as acknowledged in \cite{Carmona08} for instance.

Since such a complexity would be unbearable in high dimension, we
describe below a general memory reduction method to obtain a much
more amenable $\mathcal{O}\!\left(N+M\right)$ complexity (or, more
precisely, of $\mathcal{O}\!\left(m.N+q.M\right)$ with $m\ll M$).
This improvement really opens the door to the use of Monte Carlo methods
for American options, optimal switching and BSDEs on high-dimensional
practical applications. Note that this tool can be combined with all
the existing Monte Carlo backward methods which (seem to) require
the storage of all the trajectories.

A drawback of this tool is that it is limited to Markovian processes.
However, one can usually circumvent this restriction by increasing
the dimension of the state variable.

\subsection{General memory reduction method\label{sub:MemoryReduction}}

\subsubsection{Description}

The memory reduction method for Monte Carlo pricing of American options
was pioneered by \cite{Chan04} for the geometric Brownian motion,
and was subsequently extended to multi-dimensional geometric Brownian
motions (\cite{Chan06}) as well as exponential Lévy processes (\cite{Chan11}).
These papers take advantage of the additivity property of the processes
considered. However, as briefly hinted in \cite{Volpe09}, the memory
reduction trick can be extended to more general processes. In particular,
it can be combined with any discretization scheme, for instance the
Euler scheme or Milstein scheme, as long as the value of the stochastic
process at one time step can be expressed via its value at the subsequent
time step.

From a practical point of view, the production of ``random'' sequences
usually involves wisely chosen deterministic sequences, with statistical
properties as close as possible to true randomness (cf. \cite{Kroese11}
for instance for an overview). These sequences can usually be set
using a \textit{seed}, i.e. a (possibly multidimensional) fixed value
aimed at initializing the algorithm which produces the sequence:
\begin{equation}
\left\{ \mathrm{set\, seed\,}s\right\} \begin{array}[b]{c}
\mathrm{rand()}\\
\rightarrow
\end{array}\varepsilon_{1}\begin{array}[b]{c}
\mathrm{rand()}\\
\rightarrow
\end{array}\varepsilon_{2}\begin{array}[b]{c}
\mathrm{rand()}\\
\rightarrow
\end{array}\cdots\begin{array}[b]{c}
\mathrm{rand()}\\
\rightarrow
\end{array}\varepsilon_{n}\label{eq:randomSequence}
\end{equation}
where the $\mathrm{rand()}$ operation consists in going to the next
element of the sequence. Now two useful aspects can be stressed. The
first is that one can usually recover the current seed at any stage
of the sequence. The second is that, if the seed is set later to,
say, once again the seed $s$ from equation \eqref{eq:randomSequence},
then the following elements of the sequence will be once again $\varepsilon_{1}$,
$\varepsilon_{2}$, $\ldots$ In other words, one can recover any
previously produced subsequence of the sequence $\left(\varepsilon_{n}\right)_{n\geq1}$,
provided one stored beforehand the seed at the beginning of the subsequence.
This feature is at the core of the memory reduction method, which
we are going to discuss below in a general setting.

Consider a Markovian stochastic process $\left(X_{t}\right)_{t\geq0}$,
for instance the solution of the stochastic differential equation
\eqref{eq:SDE-X}, recalled below:
\begin{eqnarray*}
X_{0} & = & x_{0}\in\mathbb{R}^{d}\\
dX_{s} & = & b\left(s,X_{s}\right)ds+\sigma\left(s,X_{s}\right)dW_{s}
\end{eqnarray*}
The application of the Euler scheme to this equation can be denoted
as follows:
\begin{eqnarray}
x_{t_{i+1}}^{j} & = & f\left(x_{t_{i}}^{j},\varepsilon_{i}^{j}\right)\label{eq:Euler}\\
f\left(x,\varepsilon\right) & := & x+b\left(t_{i},x\right)h+\sigma\left(t_{i},x\right)\varepsilon\sqrt{h}\label{eq:fiEuler}
\end{eqnarray}
where $\forall i\in\left[0,N-1\right]$ and $\forall j\in\left[1,M\right]$,
$\varepsilon_{i}^{j}\in\mathbb{R}^{d}$ is drawn from a $d$-dimensional
Gaussian random variable. Suppose that for any $\varepsilon\in\mathbb{R}^{d}$,
the function $x\mapsto f\left(x,\varepsilon\right)$ is invertible
(call $f_{\mathrm{inv}}$ its inverse). Then, starting from the final
value $x_{t_{N}}^{j}$ of the sequence \eqref{eq:Euler}, one can
recover the whole trajectory of $X$:
\begin{equation}
x_{t_{i}}^{j}=f_{\mathrm{inv}}\left(x_{t_{i+1}}^{j},\varepsilon_{i}^{j}\right)\label{eq:EulerInverse}
\end{equation}
as long as one can recover the previous draws $\varepsilon_{N-1}^{j}$,
$\ldots$, $\varepsilon_{0}^{j}$. The following pseudo-code describes
an easy way to do it. 

\begin{algorithm}[H]
\begin{minipage}[b][1\totalheight][t]{0.45\columnwidth}%
\begin{lstlisting}[basicstyle={\sffamily},language=Matlab,numbers=left]
% Initialization
for j from 1 to M
	X[j] <- xj
end for

% LOOP 1: Euler scheme
for i from 0 to N-1
   S[i] <- getseed()
   for j from 1 to M
      E <- rand(d)
      X[j] <- f(X[j],E)
   end for
end for
S[N] <- getseed()
\end{lstlisting}
\end{minipage}\hfill{}%
\begin{minipage}[b][1\totalheight][t]{0.45\columnwidth}%
\begin{lstlisting}[basicstyle={\sffamily},language=Matlab,numbers=left]
% LOOP 2: Inverse Euler scheme
for i from N-1 down to 0
   setseed(S[i])
   for j from 1 to M
      E <- rand(d)
      X[j] <- finv(X[j],E)
   end for
end for
setseed(S[N])
\end{lstlisting}
\end{minipage}

\caption{\label{alg:MemoryReduction}Euler Scheme\hspace{40mm} Inverse Euler
Scheme}
\end{algorithm}

The first column of Algorithm \ref{alg:MemoryReduction} corresponds
to the Euler scheme, with the addition of the storage of the seeds.
At the end of the first colum, the vector $\mathsf{X}$ contains the
last values $X_{T}^{j}$, $j=1,\ldots,M$. From this point, one can
recover the previous values $X_{t_{i}}^{j}$, $i=N-1,\ldots,0$, $j=1,\ldots,M$
as done in the second column.

Inside this last loop, one can perform the estimation of the conditional
expectations required by the resolution algorithm of our stochastic
control problem (equation \eqref{eq:DPP-CondExp}). Compared to the
standard storage of the full trajectories $X_{t_{i}}^{j}$, $i=0,\ldots,N$,
$j=1,\ldots,M$, the pros and cons are the following:
\begin{itemize}
\item The number of calls to the $\mathsf{rand\left(\right)}$ function
is doubled.
\item The memory needed is brought down from $\mathcal{O}\left(M\times N\right)$
to $\mathcal{O}\left(M+N\right)$ (storage of the vector space and
the seeds).
\end{itemize}
In other words, at the price of doubling the computation time, one
can bring down the required memory storage by the factor $\min\left(M,N\right)$,
which is a very significant saving. Moreover, the theoretical additional
computation time can be insignificant in practice, as the availability
of much more physical memory makes the resort to the slower virtual
memory much less likely.
\begin{rem}
Even though the storage of the seeds does take $\mathcal{O}\left(N\right)$
in memory size, the constant may be much greater than $1$. For instance,
on $\textrm{Matlab}^{\lyxmathsym{\textregistered}}$, a seed from
the Combined Multiple Recursive algorithm (refer for instance to \cite{Kroese11}
for a description of several random generators) is made of $12$ $\mathsf{uint32}$
($32$-bit unsigned integer), a seed from the Multiplicative Lagged
Fibonacci algorithm is made of $130$ $\mathsf{uint64}$, and a seed
from the popular Mersenne Twister algorithm is made of 625 $\mathsf{uint32}$.
\end{rem}
In order to relieve the storage of the seeds, we now provide a finer
memory reduction algorithm (Algorithm \ref{alg:MemoryReduction2}).
Although Algorithm \ref{alg:MemoryReduction2} requires three main
loops, it enables to perform the last loop without fiddling the seed
of the random generator, and without any vector of seeds locked in
memory, which will thus be fully dedicated to the regressions and
other resolution operations. Moreover, the first two main loops can
be performed beforehand once and for all, storing only the last values
of the vector $\mathsf{X}$ as well as the first seed $\mathsf{S\left[0\right]}$.
Finally, if the random generator is able to leapfrop a given number
of steps, the first loop can be drastically reduced.

\begin{algorithm}[H]
\begin{minipage}[b][1\totalheight][t]{0.45\columnwidth}%
\begin{lstlisting}[basicstyle={\sffamily},language=Matlab,numbers=left]
% LOOP 1: Seeds storage
for i from 0 to N-1
   S[i] <- getseed()
   for j from 1 to M
      E <- rand(d)
   end for
end for

% Initialization
for j from 1 to M
	X[j] <- xj
end for
%
%
%
%
%
\end{lstlisting}
\end{minipage}\hfill{}%
\begin{minipage}[b][1\totalheight][t]{0.45\columnwidth}%
\begin{lstlisting}[basicstyle={\sffamily},language=Matlab,numbers=left]
% LOOP 2: Euler scheme
for i from 0 to N-1
   setseed(S[N-i-1])
   for j from 1 to M
      E <- rand(d)
      X[j] <- f(X[j],E)
   end for
end for
setseed(S[0]) ; free(S)

% LOOP 3: Inverse Euler scheme
for i from N-1 down to 0
   for j from 1 to M
      E <- rand(d)
      X[j] <- finv(X[j],E)
   end for
end for
\end{lstlisting}
\end{minipage}

\caption{\label{alg:MemoryReduction2}General Memory Reduction Method}
\end{algorithm}

\subsubsection{Numerical stability}

Theoretically, the trajectories produced by the Euler scheme \eqref{eq:Euler}
and the inverse Euler scheme \eqref{eq:EulerInverse} are exactly
the same. In practice however, a discrepancy may appear, the cause
of which is discussed below.

On a computer, not all real numbers can be reproduced. Indeed, they
must be stored on a finite number of bits, using a predefined format
(usually the IEEE Standard for Floating-Point Arithmetic (IEEE 754)).
In particular, there exists an incompressible distance $\varepsilon>0$
between two different numbers stored. This causes rounding errors
when performing operations on real numbers.

For instance, consider $x\in\mathbb{R}$ and an invertible function
$f:\mathbb{R}\mapsto\mathbb{R}$. Compute $y=f\left(x\right)$ and
then compute $\hat{x}=f_{\mathrm{inv}}\left(y\right)$. One would
expect that $\hat{x}=x$, but in practice, because of rounding effects,
one may get $\hat{x}=x+\epsilon z$ for a small $\epsilon>0$, where
$z$ is a discrete variable, which can be deemed random, taking values
around zero. This phenomenon is illustrated on Figure \ref{fig:Rounding},
which displays a histogram of $\hat{x}-x$ for $n=10^{7}$ different
values of $x\in\left[0,1\right]$ and for the simple linear function
$f\left(x\right)=2x+3$.

\begin{figure}[H]
\noindent \begin{centering}
\includegraphics[width=0.5\paperwidth]{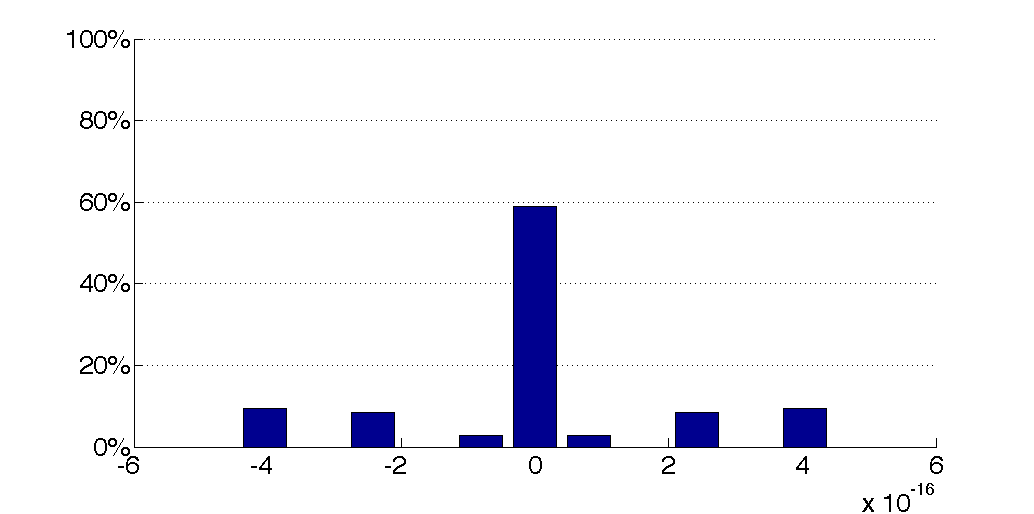}
\par\end{centering}

\caption{\label{fig:Rounding}Histogram of rounding errors}
\end{figure}

We now describe how this affects our memory reduction method. Recall
equation \ref{eq:Euler}:
\[
x_{t_{i+1}}^{j}=f\left(x_{t_{i}}^{j},\varepsilon_{i}^{j}\right)
\]

Now, instead of equation \eqref{eq:EulerInverse}, the inverse Euler
scheme will provide something like:
\begin{eqnarray}
y_{t_{N}}^{j} & = & x_{t_{N}}^{j}\nonumber \\
y_{t_{i}}^{j} & = & f_{\mathrm{inv}}\left(y_{t_{i+1}}^{j},\varepsilon_{i}^{j}\right)+\epsilon z_{i}^{j}\label{eq:EulerInverseError}
\end{eqnarray}
for a small $\epsilon>0$, where $z_{i}^{j}$, $i=0,\ldots,N$, $j=1,\ldots,M$,
can be deemed realizations of a discrete random variable $Z$, independent
of $W$. The distribution of $Z$ is unknown, but data suggests it
may be innocuously assumed centered, symmetric, and with finite moments.

We are now interested in studying the compound rounding error $y_{t_{i}}-x_{t_{i}}$
as a function of $\epsilon$. Of course, its behaviour depends on
the choice of $f$ (equation \eqref{eq:fiEuler}). Below, we explicit
this error on two simple examples: an arithmetic Brownian motion and
an Ornstein-Uhlenbeck process. These two examples illustrate how the
compound rounding error can vary dramatically w.r.t. $f$.

\paragraph{First example: arithmetic Brownian motion}

Consider first the case of an arithmetic Brownian motion with drift
parameter $\mu$ and volatility parameter $\sigma$. Here $f$ and
its inverse are given by:
\begin{eqnarray*}
f\left(x,\varepsilon\right) & = & x+\mu h+\sigma\sqrt{h}\varepsilon\\
f_{\mathrm{inv}}\left(x,\varepsilon\right) & = & x-\mu h-\sigma\sqrt{h}\varepsilon
\end{eqnarray*}
Hence, using equation \eqref{eq:EulerInverseError}, for every $j=1,\ldots,M$:
\[
y_{t_{i}}^{j}-x_{t_{i}}^{j}=\epsilon\sum_{k=i}^{N-1}z_{k}^{j}
\]
In other words, the compound rounding error behaves as a random walk,
multiplied by the small parameter $\epsilon$. Hence, as long as $\epsilon\ll h$
(which is always the case as real numbers smaller than $\epsilon$
cannot be handled properly on a computer), this numerical error is
harmless.

Remark that a similar numerical error arises from the algorithms proposed
in \cite{Chan04} , \cite{Chan06} and \cite{Chan11}, but, fortunately,
as discussed above, this error is utterly negligible.

\paragraph{Second example: Ornstein-Uhlenbeck process}

Now, consider the case of an Ornstein-Uhlenbeck process with mean
reversion $\alpha>0$, long-term mean $\mu$ and volatility $\sigma$.
Here:
\begin{eqnarray*}
f\left(x,\varepsilon\right) & = & x+\alpha\left(\mu-x\right)h+\sigma\sqrt{h}\varepsilon\\
f_{\mathrm{inv}}\left(x,\varepsilon\right) & = & \frac{1}{1-\alpha h}\left(x-\alpha\mu h-\sigma\sqrt{h}\varepsilon\right)
\end{eqnarray*}
Using equation \eqref{eq:EulerInverseError}, for every $j=1,\ldots,M$
the compound error is given by:
\[
y_{t_{i}}^{j}-x_{t_{i}}^{j}=\epsilon\sum_{k=i}^{N-1}\frac{1}{\left(1-\alpha h\right)^{k-i}}z_{k}^{j}
\]
As $\left(1-\alpha h\right)^{-N}\sim\exp\left(\alpha T\right)$ when
$h\rightarrow0$, one can see that, as soon as $T>-\frac{\ln\left(\epsilon\right)}{\alpha}$,
this error may become overwhelming. This phenomenon is illustrated
on Figure \ref{fig:withoutSaves} on a sample of $100$ trajectories.

In order to mitigate this effect, we propose to modify the Algorithm
\ref{alg:MemoryReduction2} as follows: in its second loop (usual
Euler scheme), instead of saving only the last values $x_{T}^{j}$,
one may define a small subset $\tilde{\Pi}\subset\Pi$ and save the
intermediate values $x_{t_{i}}^{j}$, $t_{i}\in\tilde{\Pi}$. Then,
in the last loop (inverse Euler scheme), every time that $t_{i}\in\tilde{\Pi}$,
the current value of the set $x_{t_{i}}^{j}$ may be recovered from
this previous storage.

Figure \ref{fig:withSaves} illustrates the new behaviour of the compound
rounding error with this mended algorithm, on an example with $T=10$
years and $4$ intermediate saves (in addition to the final values).

The drawback of this modification, of course, is that it multiplies
the required storage space by the factor $\#\tilde{\Pi}$. However,
this remains much smaller than the $\mathcal{O}\left(M\times N\right)$
required by the naive full storage algorithm.

\begin{figure}[H]
\hspace{-5.1mm}\subfloat[\label{fig:withoutSaves}Without intermediate saves]{\includegraphics[width=0.41\paperwidth]{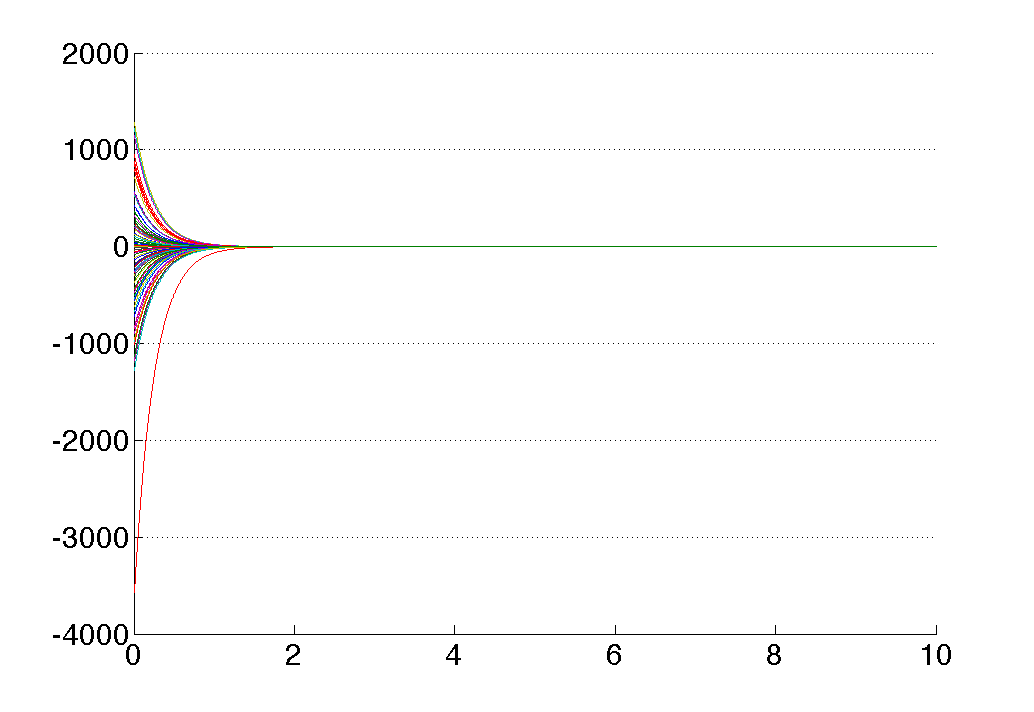}

}\hspace{-7.9mm}\subfloat[\label{fig:withSaves}With intermediate saves]{\includegraphics[width=0.41\paperwidth]{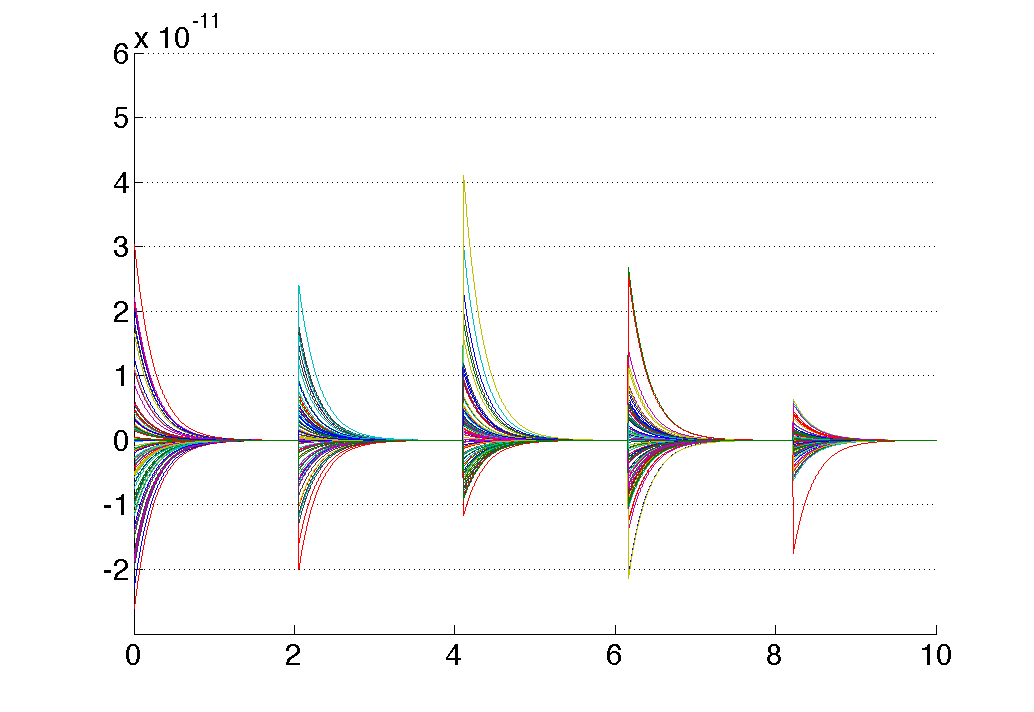}

}\caption{\label{fig:EulerInvOU}Compound rounding error for the Ornstein-Uhlenbeck
process}
\end{figure}

\section{Application to investment in electricity generation\label{sec:Application}}

This section is devoted to an application of the resolution method
studied in Section \ref{sec:Optimal-switching}. We choose to apply
it to an investment problem in electricity generation on a single
geographical zone. We intend to show that it is possible to provide
a probabilistic outlook of future electricity generation mixes instead
of a deterministic outlook provided by planification methods. Nevertheless,
the problem presents so many difficulties that addressing all of them
in the same model is unresonable. Some aspects have thus to be left
aside. Our goal here is to show that the algorithm described in Section
\ref{sub:Convergence} can handle high-dimensional investment problems.
We focus on the influence of investment decisions on the spot price,
consistently with the fundamentals of the electricity spot price formation
mechanism.

Although the strategic aspect of investment is an important driver
of utilities' decisions, this aspect is beyond the scope of our modeling
approach. There exist models limited to a two-stage decision making
(see for instance \cite{Murphy05}), but in the case of continuously
repeated multiplayer game models, defining what is a closed-loop strategy
is already a difficulty (see Sec. 2 in \cite{Back09}).

We did not consider time-to-build in this implementation either. Relying
on the fact that it is possible to transform an investment model with
time-to-build into a model without time-to-build by replacing capacities
with committed capacities (see \cite{BarIlan02,Aguerrevere03} for
implementations in dimension one, and \cite{Federico12} in dimension
two), we left this aspect for future work.

Finally, we did not consider the dynamic constraints of power generation.
Their effect on spot prices is well-known: they tend to increase spot
prices during peak hours and to decrease them during off-peak hours
(see \cite{Langrene11}). However, we assume here that this effect
is negligible compared to the effect induced by a lack or an excess
of capacity.

Thus, we focused on the following key factors of electricity spot
prices: demand, capacities (including random outages) and fuel prices.
Our model is based on \cite{Aid09,Aid12}, where the electricity spot
price is defined as a linear combination of fuel prices multiplied
by a scarcity factor. This model exhibits the main feature wanted
here, which is that the spot price, being determined both by the fuel
prices and the residual capacity, is directly affected by the evolution
of the installed capacity. When the residual capacity tends to decrease,
spot prices will tend to increase, making investment valuable. Thus,
in this model, investments are undertaken not on the specific purpose
of satisfying the demand but as soon as they are profitable. In our
example, new capacities are invested according to the criterion of
value maximization. Energy non-served and loss of load probability
may still be adjusted through the price cap on the spot market.

In this section, we first detail the chosen modelling and objective
function (which will be shown to be encompassed in the general optimal
multiple switching problem \eqref{eq:v}), and then solve it numerically
using the general algorithm developed in the previous sections.

\subsection{Modelling\label{sub:Modelling}}

The key variable in order to describe our electricity generation investment
problem is the price of electricity. More precisely, the key quantities
are the spreads between the prices of electricity and other energies.
To model these spreads accurately, it may be worth considering a structural
model for electricity (cf. the survey \cite{Carmona12}). Here we
choose such a model, mainly inspired by those introduced in \cite{Aid09}
and \cite{Aid12}, albeit amended and customized for a long-term time
horizon. All the variables involved are detailed below.

\subsubsection{Electricity demand}

The electricity demand, or electricity load, at time $t$ on the given
geographical zone considered is modelled by an exogenous stochastic
process $\left(D_{t}\right)_{t\geq0}$:
\begin{equation}
D_{t}=f_{0}\left(t\right)+Z_{t}^{0}\label{eq:Demand}
\end{equation}
where $Z^{0}$ is an Ornstein-Uhlenbeck (henceforth O.U.) process:
\[
dZ_{t}^{0}=-\alpha_{0}Z_{t}^{0}dt+\beta_{0}dW_{t}^{D}
\]
where $\alpha_{0}$ and $\beta_{0}$ are constants, and $f_{0}$ is
a deterministic function that takes into account demand seasonalities:
\begin{equation}
f_{0}\left(t\right)=d_{1}+d_{2}\cos\left(2\pi\frac{t-d_{3}}{l_{1}}\right)+f_{week}\left(t\right)\label{eq:SeasonalityD}
\end{equation}
where $d_{j}$, $1\leq j\leq3$ are constants, and, assuming that
$t$ is expressed in years, $l_{1}=1$ (yearly seasonality), and $f_{week}$
is a periodic non-parametric deterministic function describing the
intra-week load pattern.

\subsubsection{Production capacities\label{par:Production-capacities}}

Let $d'$ be the number of different production technologies. Denote
as $I_{t}=\left(I_{t}^{1},\ldots,I_{t}^{d'}\right)$ the installed
production capacities at time $t$. They represent the maximum amount
of electricity that is physically possible to produce. These fleets
can be modified: at a given time $\tau_{n}$, one can decide to build
(or dismantle) an amount $\zeta_{n}$ of capacities:
\begin{equation}
I_{\tau_{n}}=I_{\tau_{n}^{-}}+\zeta_{n}\,,\,\, n\geq0\label{eq:InstalledCapacity}
\end{equation}
Denote as $\alpha=\left(\tau_{n},\zeta_{n}\right)_{n\geq1}$ the corresponding
impulse control strategy, where $\left(\tau_{n}\right)_{n\geq0}$
is an increasing sequence of stopping times with $\tau_{n}\nearrow\infty$
when $n\rightarrow\infty$, and $\left(\zeta_{n}\right)_{n\geq0}$
is a sequence of vectors corresponding to the increases (or decreases)
in capacities. Apart from these variations, $I_{t}$ will be deemed
constant, i.e.:
\begin{equation}
I_{t}=I_{0-}+\sum_{n,\,\tau_{n}\leq t}\zeta_{n}\,.\label{eq:InstalledCapacitySum}
\end{equation}
Now, denote as $C_{t}=\left(C_{t}^{1},\ldots,C_{t}^{d'}\right)$ the
available production capacities. Because of spinning reserves, maintenance
and random outages, these quantities are lower than the installed
capacities $I_{t}$, which represent their physical maximum. In other
terms, $C_{t}$ is a fraction of $I_{t}$:
\begin{equation}
C_{t}^{i}=I_{t}^{i}\times A_{t}^{i}\label{eq:AvailableCapacity}
\end{equation}
for every $1\leq i\leq d'$, where $A_{t}^{i}$ corresponds to the
rate of availability of the $\mathrm{i^{th}}$ production technology.
Therefore one must choose a model for the process $A_{t}$ that ensures
that it stays within the interval $\left[0,1\right]$.

One possibility would be to model it as a Jacobi process (see for
instance \cite{Veraart10},where it is used to model stochastic correlations,
and the references therein for more information on this process).
This process is however tricky to estimate and simulate (see \cite{Gourieroux04}
for the description of some possible methods), and its main simulation
method (the truncated Euler scheme) disables our memory reduction
method described in Subsection \ref{sub:MemoryReduction}. Hence we
look for a simpler model.

In \cite{Wagner12}, a detailed structural model for electricity is
developed, which includes renewable energies like wind and solar.
In particular, wind power infeed efficiency (which belongs to $\left[0,1\right]$)
is modelled as a logit transform of an Ornstein-Uhlenbeck process
with seasonality. Adapting this idea, we model $\left(A_{t}^{i}\right)_{t\geq0}^{1\leq i\leq d'}$
as follows:
\begin{equation}
A_{t}^{i}:=\mathcal{T}\left(f_{i}\left(t\right)+Z_{t}^{i}\right)\label{eq:AvailabilityRate}
\end{equation}
where $Z$, $f$ and $\mathcal{T}$ are chosen as follows:
\begin{itemize}
\item $Z^{i}$ is an O.U. process :
\[
dZ_{t}^{i}=-\alpha_{i}Z_{t}^{i}dt+\beta_{i}dW_{t}^{Z^{i}}
\]
where $\alpha_{i}>0$, $\beta_{i}>0$ and $\left(W_{t}^{Z^{i}}\right)_{t\geq0}$
is a Brownian motion.
\item The deterministic function $f_{i}$ accounts for the seasonality in
the availability of production capacities:
\begin{equation}
f_{i}\left(t\right)=c_{1}^{i}+c_{2}^{i}\cos\left(2\pi\frac{t-c_{3}^{i}}{l_{1}}\right)\label{eq:SeasonalityC}
\end{equation}
where $c_{k}^{i}$, $1\leq k\leq3$, $1\leq i\leq d'$ are constants.
This seasonality stems from the maintenance plannings, which usually
mimic the long term seasonality of demand, which in turn originates
in the seasonality of temperature.
\item The function $\mathcal{T}:\mathbb{R}\rightarrow\left[0,1\right]$
is here to ensure that $\forall t\geq0,\, A_{t}\in\left[0,1\right]^{d'}$.
One can choose the versatile logit function as in \cite{Wagner12},
or any other mapping of $\mathbb{R}$ into $\left[0,1\right]$. For
instance, any cumulative distribution function would be suitable.
As the process $Z$ is Gaussian and asymptotically stationary, we
choose for $\mathcal{T}$ the (standard) normal cumulative distribution
function, as it makes, in particular, the calibration process trivial.
\end{itemize}

\subsubsection{Fuels and\textmd{ $\textrm{CO}_{\textrm{2}}$} prices}

For each technology $i$, denote as $S_{t}^{i}$ the price of the
fuel $i$ to produce electricity at time $t$. In the particular case
of renewable energies, which, \textit{per se,} do not involve traded
fuels, the corresponding $S_{t}^{i}$ can be chosen to be zero. Moreover,
define $S_{t}^{0}$ as the price of $\textrm{CO}_{\textrm{2}}$. Denote
as $S_{t}$ the full vector $\left(S_{t}^{0},S_{t}^{1},\ldots,S_{t}^{d'}\right)$.

Now, we introduce the multiplicative constants needed to convert theses
quantities into \euro{}/MWh. For each technology $i=1,\ldots,d'$,
let $h_{i}$ denote its heat rate, and $h_{i}^{0}$ denote its $\textrm{CO}_{\textrm{2}}$
emission rate. Hence, the quantity
\begin{equation}
\tilde{S}_{t}^{i}:=h_{i}^{0}S_{t}^{0}+h_{i}S_{t}^{i}\label{eq:FuelPrice}
\end{equation}
expressed in \euro{}/MWh, corresponds to the price in \euro{} to
pay in order to produce $1$MWh of electricity using the $i$th technology.
We note $h^{0}=\left(h_{1}^{0},\ldots,h_{d'}^{0}\right)\in\mathbb{R}^{d'}$
and $h=\left(h_{1},\ldots,h_{d'}\right)\in\mathbb{R}^{d'}$.
\begin{rem}
One can choose to add a fixed cost into the definition of $\tilde{S}_{t}^{i}$.
This is all the more so relevant for technologies whose fixed costs
outweigh the cost of fuel (e.g. nuclear).
\end{rem}
Adapting the work of \cite{Benmenzer07}, we model $S_{t}$ as a multidimensional,
cointegrated geometric Brownian motion:
\[
dS_{t}=\Xi S_{t}dt+\Sigma S_{t}dW_{t}^{S}
\]
where $\Xi$ and $\Sigma$ are $\left(d'+1\right)\times\left(d'+1\right)$
matrices with $1\leq\mathrm{rank}\left(\Xi\right)<d'$, and $\left(W_{t}^{S}\right)_{t\geq0}$
is a $\left(d'+1\right)$-dimensional Brownian motion. This model
ensures the positivity of prices, as well as the existence of long-term
relationships between energy prices (the relevance of which is illustrated,
for instance, in \cite{Obermayer09}).

\subsubsection{Electricity price\label{sub:ElectricityPrice}}

We model the price of electricity using a long-term structural model.
First, we define the marginal cost of electricity using the previously
introduced variables. For any time $t\geq0$, define the permutation
$\left(1\right),\ldots\left(M\right)$ of the numbers $1,\ldots,M$,
such that $S_{t}^{\left(1\right)}\leq\ldots\leq S_{t}^{\left(M\right)}$.
Then, define $\overline{C}_{t}^{\left(i\right)}$ as the total capacity
available at time $t$ from the $i$ first technologies, i.e. $\overline{C}_{t}^{\left(i\right)}:=\sum_{j\leq i}C_{t}^{\left(j\right)}$.
Using these notations and equation \eqref{eq:FuelPrice}, the marginal
cost of electricity at time $t$ is given by:
\begin{eqnarray*}
MC_{t}: & = & \widetilde{S}_{t}^{\left(1\right)}\mathbf{1}\left\{ D_{t}<\overline{C}_{t}^{\left(1\right)}\right\} +\sum_{i=2}^{M-1}\widetilde{S}_{t}^{\left(i\right)}\mathbf{1}\left\{ \overline{C}_{t}^{\left(i-1\right)}\leq D_{t}<\overline{C}_{t}^{\left(i\right)}\right\} +\widetilde{S}_{t}^{\left(M\right)}\mathbf{1}\left\{ \overline{C}_{t}^{\left(M-1\right)}\leq D_{t}\right\} \\
 & = & \widetilde{S}_{t}^{\left(1\right)}+\sum_{i=1}^{M-1}\left(\widetilde{S}_{t}^{\left(i+1\right)}-\widetilde{S}_{t}^{\left(i\right)}\right)\mathbf{1}\left\{ D_{t}-\overline{C}_{t}^{\left(i\right)}\geq0\right\} 
\end{eqnarray*}
Refer to \cite{Aid09} for more details on marginal costs. Remark
that the price of $\textrm{CO}_{\textrm{2}}$ emissions is explicitly
included in the marginal cost (through equation \eqref{eq:FuelPrice}).

Now, we are going to use this marginal cost as a building block of
our price model, along with some power law \textit{scarcity premiums}
(along the lines of \cite{Aid12}) as well as a fixed upper bound
\footnote{Indeed, in the French, German and Austrian markets for instance, power
prices cannot be set outside the $\left[-3000,3000\right]$\euro{}/MWh
range, see \url{http://www.epexspot.com/en/product-info/auction}..%
}.

First, consider two points $\left(x_{1},y_{1}\right)$ and $\left(x_{2},y_{2}\right)$
in $\mathbb{R}^{2}$. One can always find three positive constants
$a:=a\left(x_{1},x_{2},y_{1},y_{2}\right)$, $b:=b\left(x_{1},x_{2},y_{1},y_{2}\right)$
and $c:=c\left(x_{1},x_{2},y_{1},y_{2}\right)$ such that the function:
\begin{equation}
p\left(x\right):=p\left(x;x_{1},x_{2},y_{1},y_{2}\right)=\frac{a}{b-x}+c\label{eq:knitting}
\end{equation}
satisfies $p\left(x_{1}\right)=y_{1}$ and $p\left(x_{2}\right)=y_{2}$
\footnote{For instance, fix $a>0$, then define $b=\frac{1}{2}\left(x_{1}+x_{2}+\sqrt{\left(x_{2}-x_{1}\right)^{2}+4a\frac{x_{2}-x_{1}}{y_{2}-y_{1}}}\right)$$ $
and finally $c=y_{1}-\frac{a}{b-x_{1}}$.%
}.

Using this notation, introduce the price $P_{t}$ of electricity,
defined as follows:
\begin{eqnarray}
P_{t} & := & \widetilde{S}_{t}^{\left(1\right)}\mathbf{1}\left\{ D_{t}<0\right\} +\left\{ \widetilde{S}_{t}^{\left(1\right)}+p\left(D_{t};0,\overline{C}_{t}^{\left(1\right)},\widetilde{S}_{t}^{\left(1\right)},\widetilde{S}_{t}^{\left(2\right)}\right)\right\} \mathbf{1}\left\{ 0\leq D_{t}<\overline{C}_{t}^{\left(1\right)}\right\} \nonumber \\
 &  & \sum_{i=2}^{d'-1}\left\{ \widetilde{S}_{t}^{\left(i\right)}+p\left(D_{t};\overline{C}_{t}^{\left(i-1\right)},\overline{C}_{t}^{\left(i\right)},\widetilde{S}_{t}^{\left(i\right)},\widetilde{S}_{t}^{\left(i+1\right)}\right)\right\} \mathbf{1}\left\{ \overline{C}_{t}^{\left(i-1\right)}\leq D_{t}<\overline{C}_{t}^{\left(i\right)}\right\} \nonumber \\
 &  & +\left\{ \widetilde{S}_{t}^{\left(d'\right)}+p\left(D_{t};\overline{C}_{t}^{\left(d'-1\right)},\overline{C}_{t}^{\left(d'\right)},\widetilde{S}_{t}^{\left(d'\right)},M_{\max}\right)\right\} \mathbf{1}\left\{ \overline{C}_{t}^{\left(d'-1\right)}\leq D_{t}\right\} \label{eq:TheModel}
\end{eqnarray}

where $M_{\max}>0$ is a fixed upper bound on the price of electricity.
In particular, the last term, the one involving $M_{\max}$, enables
price spikes to occur (when the residual capacity is small).

Moreover, thanks to the knitting function \eqref{eq:knitting}, the
electricity price $P$ is a Lipschitz continuous function of the structural
variables $D$, $C$ and $S$ %
\footnote{Rigorously, this property requires that $C$ does not reach zero.
One can, for instance, add a fixed minimum availability rate $1\gg a_{\min}>0$
to the definition \eqref{eq:AvailabilityRate}, replacing $\mathcal{T}$
by $a_{\min}+\left(1-a_{\min}\right)\mathcal{T}$%
}, which is what motivated this specific choice of model.

\subsubsection{Objective function\label{sub:ObjectiveFunction}}

We now explicit the objective function of the investor in electricity
generation. Suppose that, at time $t$, an agent (a producer, or an
investor) modifies the level of installed capacity of type $j\in\left[1,d'\right]$,
from $I_{t-}^{j}$ to $I_{s}^{j}=I_{t-}^{j}+\zeta^{j}$, $s\geq t$
. It generates the cost:
\[
k\left(\zeta^{j}\right):=\begin{cases}
\kappa_{j}^{f+}+\zeta^{j}\kappa_{j}^{p+} & ,\,\zeta^{j}>0\\
0 & ,\,\zeta^{j}=0\\
\kappa_{j}^{f-}-\zeta^{j}\kappa_{j}^{p-} & ,\,\zeta^{j}<0
\end{cases}
\]
where $\kappa_{j}^{f+}$ and $\kappa_{j}^{p+}$ are the fixed and
proportional costs of building new plants of type $j$, and $\kappa_{j}^{f-}$
and $\kappa_{j}^{p-}$ are the fixed and proportional costs of dismantling
old plants of type $j$. 

Consider the case of new plants ($\zeta^{j}>0$). Assuming that the
global availability rate\eqref{eq:AvailabilityRate} of technology
$j$ applies to the new plants, they can then produce up to $\zeta^{j}A_{s}^{j}$,
$s\geq t$, or, more precisely, according to the stack order principle:
\[
\min\left\{ \zeta^{j}A_{s}^{j},\,\left(D_{s}-\overline{C}_{s}^{\left(j-1\right)}\right)^{+}\right\} 
\]

assuming that, in the stack order, the new plants are called before
the older plants $I_{t-}$ of the same technology (as they can be
expected to have an at least slightly better efficiency rate compared
to the older plants of the same technology, a phenomenon that can
be seen as partly captured by the function \eqref{eq:knitting}).

At time $s\geq t$, this production is sold at price $P_{s}$, but
costs $\tilde{S}_{s}$ to produce (if $P_{s}<\tilde{S}_{s}$, then
of course the producer chooses not to produce). In addition, regardless
of the output level, there may exist a fixed maintenance cost $\kappa_{j}$.
Summing up all these gains, discounted to time $t$ using a constant
interest rate $\rho>0$, the new plant yield a revenue of:

\[
\int_{t}^{\infty}e^{-\rho s}\left(\min\left\{ \zeta^{j}A_{s}^{j},\, D_{s}-\overline{C}_{s}^{\left(j-1\right)}\right\} \times\left(P_{s}-\widetilde{S}_{s}^{j}\right)^{+}-\kappa_{i}\right)ds
\]

(noticing that with our power price model, $\left\{ D_{s}-\overline{C}_{s}^{\left(j-1\right)}\leq0\right\} \Leftrightarrow\left\{ P_{s}-\widetilde{S}_{s}^{j}\leq0\right\} $).
This was the cost-benefit analysis for one quantity $\zeta^{j}$ of
new plants. Now, consider as a whole the full fleet of the geographical
zone considered. Maximizing the expected gains along the potential
new plants yields the following value function:
\begin{equation}
v\left(t,x,i\right)=\sup_{\alpha\in\mathcal{A}_{t,i}}\mathbb{E}\left[\sum_{j=1}^{d'}\int_{t}^{\infty}e^{-\rho s}\left(\min\left\{ C_{s}^{j},D_{s}-\overline{C}_{s}^{\left(j-1\right)}\right\} \times\left(P_{s}-\widetilde{S}_{s}^{j}\right)^{+}-\kappa_{i}\right)ds-\sum_{\tau_{n}\geq t}e^{-\rho\tau_{n}}k\left(\zeta^{j}\right)\right]\label{eq:OBJECTIVE}
\end{equation}
where the strategies $\alpha$ affect the installed capacities (equations
\eqref{eq:InstalledCapacitySum}), hence also the available capacities
(equation \eqref{eq:AvailableCapacity}) as well as the power price
(equation \eqref{eq:TheModel}), and where the cash flows are purposely
discounted up to time $0$, the time of interest.
\begin{rem}
Replacing $P$ in \eqref{eq:OBJECTIVE} by its definition \eqref{eq:TheModel},
it is patent that this objective function fits into the mould studied
thoroughly in Section \ref{sub:Convergence}. In Subsection \ref{sub:Numerics}
below, our algorithm will be applied to this specific objective function.
\end{rem}
\medskip{}

\begin{rem}
Remark that under this modelling, the demand is satisfied as long
as it does not exceed the total available capacity. Indeed, the effective
output of the plant $\zeta^{j}$ is equal to $\mathbf{1}\left\{ P_{s}-\widetilde{S}_{s}^{j}>0\right\} \times\min\left\{ \zeta^{j}A_{s}^{j},\,\left(D_{s}-\overline{C}_{s}^{\left(j-1\right)}\right)^{+}\right\} $.
It is indeed governed by the electricity spot price level, but, as
under our modelling $\mathbf{1}\left\{ P_{s}-\widetilde{S}_{s}^{j}>0\right\} =\mathbf{1}\left\{ D_{s}-\overline{C}_{s}^{\left(j-1\right)}>0\right\} $,
summing up the effective outputs of all the power plants yield $\sum_{j=1}^{d'}\min\left\{ C_{s}^{j},D_{s}-\overline{C}_{s}^{\left(j-1\right)}\right\} \times\mathbf{1}\left\{ D_{s}-\overline{C}_{s}^{\left(j-1\right)}>0\right\} =\min\left\{ D_{s},\overline{C}_{s}^{\left(d'\right)}\right\} $.
\end{rem}

\subsection{Numerical results\label{sub:Numerics}}

Finally, we solve the control problem described in Subsection \ref{sub:Modelling}
on a numerical example, using the algorithm detailed in Subsection
\ref{sub:Convergence} combined with the general memory reduction
method described in Subsection \ref{sub:MemoryReduction}.

Our purpose here is not to perform a full study of investments in
electricity markets, but a more modest attempt at illustrating the
practical feasibility of our approach, with some possible outputs
that the algorithm can provide.

We consider a numerical example including two cointegrated fuels (in
addition to the price of $\mathrm{CO_{2}}$): one ``base fuel''
and one ``peak fuel'', starting respectively from $40$\euro{}/MWh
and $80$\euro{}/MWh. Hence, using the notations from Subsection
\ref{sub:Modelling}, $d'=2$ (two technologies) and $d=6$ ( electricity
demand, $\mathrm{CO_{2}}$ price, two fuel prices and two availability
rates). The main choices of parameters for this application (initial
fuel prices and volatilities, initial fleet and proportional costs
of new power plants) are summed up in Table \ref{tab:param}. Moreover,
the demand process starts from $D_{0}=70$GW and does not integrate
any linear trend.

\begin{table}[H]
\noindent \begin{centering}
\begin{tabular}{|c|c|c|c|c|}
\hline 
i & $S_{0}^{i}$ & $\sigma^{i}$ & $I_{0}^{i}$ & $\kappa_{i}^{p+}$\tabularnewline
\hline 
\hline 
1 & $40$\euro{}/MWh & $5\%$ & $67$GW & $0.24$ $10^{9}$\euro{}/GW\tabularnewline
\hline 
2 & $80$\euro{}/MWh & $15\%$ & $33$GW & $2.00$ $10^{9}$\euro{}/GW\tabularnewline
\hline 
\end{tabular}
\par\end{centering}

\caption{\label{tab:param}Model parameters}
\end{table}

In order to take into account the minimum size of one power plant
we restrict the values of the installed capacity process\eqref{eq:InstalledCapacitySum}
to a (bi-dimensional) fixed grid $\Lambda^{d'}$, with a mesh of $1$GW.
We make the simplifying assumptions that investments are irreversible,
and that no dismantling can occur (recall from Subsection \ref{sub:Complexity}
the computational gain provided by this assumption).
\begin{rem}
If such a grid is indeed manageable in dimension $d'=2$, it may less
be the case if additional technologies were considered. However, as
discussed in \cite{Tan11} equation (3.2), instead of performing one
regression for each $i\in\Lambda^{d'}$, one can solve equation \eqref{eq:DDP-approx-5}
at time $t_{i}$ by only one $(d+d')$-dimensional regression, by
choosing an a priori law for the randomized control $\zeta_{t_{i}}$.
The error analysis from Section \ref{sec:Optimal-switching} can be
easily generalized to such regressions in higher dimension.
\end{rem}
Finally, we consider the following numerical parameters. We choose
a time horizon $T=40$ years and a time step $h=\frac{1}{730}$ (i.e.
two time steps per day, allowing for some intraday pattern in the
demand process) but allow for only one investment decision per year.
For the regression, we consider a basis of $b=2^{d}=64$ adaptative
local functions, chosen piecewise linear on each hypercube (which
is a bit more refined than the piecewise constant basis studied in
Section \ref{sub:Convergence}) on a sample of $M=5000$ trajectories.

The numerical results obtained under this set of parameters are displayed
on Figures \ref{fig:Strategies} and \ref{fig:ElecPrice}.

First, Figure \ref{fig:Strategies} deals with the optimal strategies.
Figure \ref{fig:stratContour} displays the time evolution of the
average as well as the variability of the optimal fleet (only the
new plants are shown). One can distinguish a first short phase characterised
by the construction of several GW of peak load assets, followed by
a much slower second phase involving the construction of both base
load and peak load assets. Moreover, the variability of the optimal
fleet increases over time. The detailed histogram of the optimal strategy
at time $T=40$ years is displayed on Figure \ref{fig:finalFleet},
where it is combined with the price of fuel. One can see that the
more the peak fuel is expensive (and hence both fuels are expensive
on average, as they are cointegrated), the more constructions of base
load plants occur.

\begin{figure}[H]
\hspace{-5.1mm}\subfloat[\label{fig:stratContour}Time evolution of new capacities]{\includegraphics[width=0.41\paperwidth]{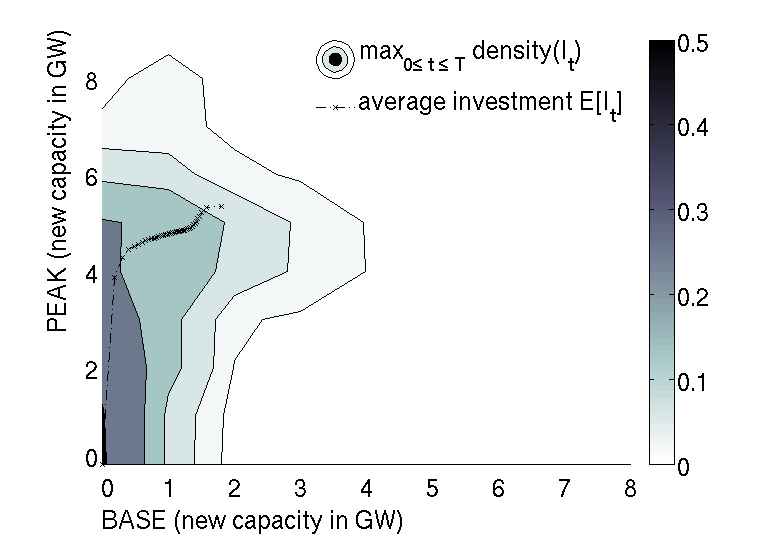}

}\hspace{-3mm}\subfloat[\label{fig:finalFleet}Final fleet distribution]{\includegraphics[width=0.41\paperwidth]{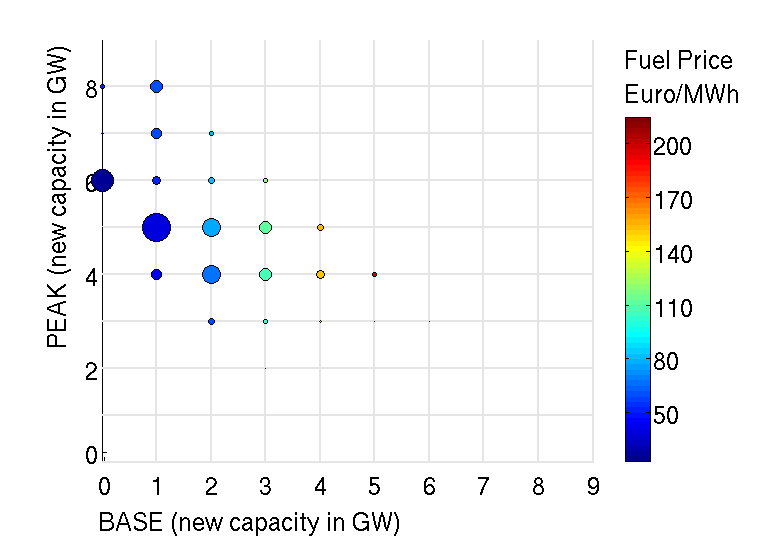}

}\caption{\label{fig:Strategies}Optimal strategies}
\end{figure}

The fact that the average fleet seem to converge is related to the
fact that this numerical example does not consider any growth trend
in the electricity demand (see equation \eqref{eq:SeasonalityD}).
Otherwise, more investments would occur, indeed, to keep the pace
with consumption.

Then Figure \ref{fig:ElecPrice} provides information on the price
of electricity. Figure \ref{fig:priceDensity} displays the time evolution
of the electricity spot price density. For better readability, each
density covers one whole year. One can see how the density moves away
from the initial bimodal density (with prices clustering around the
initial prices of the two fuels) towards a more diffuse density. Moreover,
the downward effect of investments on prices can be noticed. This
downward effect is even more visible on Figure \ref{fig:stratComparison}.
It compares the effect on electricity prices of three different strategies:
the optimal strategy, the optimal deterministic strategy (computed
as the average of the optimal strategy), and the do-nothing strategy.
For each strategy, the joint time-evolution of the yearly median price
and the yearly interquartile range are drawn. As expected, prices
tend to be higher and more scattered without any new plant. Nevertheless,
on this specific example, the price distribution under the optimal
deterministic strategy is close to that under the optimal strategy
(only slightly more scattered). 

\begin{figure}[H]
\hspace{-5.1mm}\subfloat[\label{fig:priceDensity}Time evolution of electricity spot price
density]{\includegraphics[width=0.41\paperwidth]{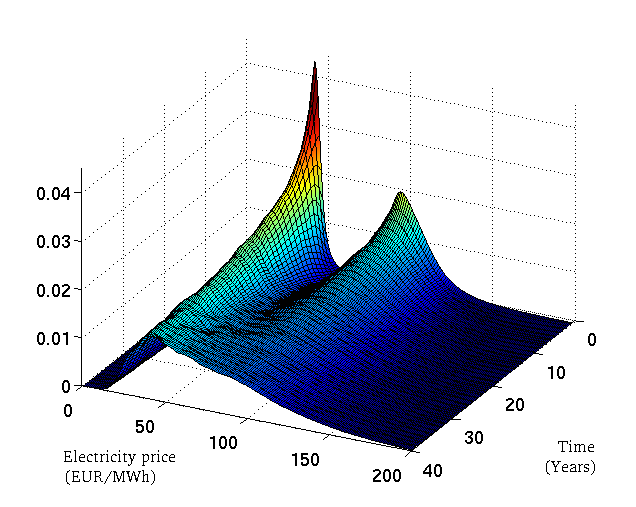}

}\hspace{-3mm}\subfloat[\label{fig:stratComparison}Comparison between investment strategies]{\includegraphics[width=0.41\paperwidth]{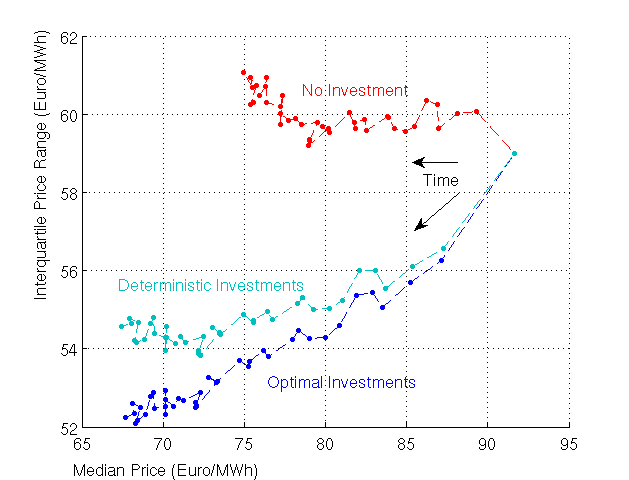}

}\caption{\label{fig:ElecPrice}Electricity spot price}
\end{figure}

These few pictures illustrate the kind on information that can be
be extracted from the resolution of this control problem. Of course,
as a by-product of the resolution, much more can be extracted and
analyzed (distribution of income, $\mathrm{CO}_{2}$ emissions, optimal
exercise frontiers, etc) if needed.

\section{Conclusion\label{sec:Conclusion}}

In this paper, we presented a probabilistic method to solve optimal
multiple switching problems. We showed on a realistic investment model
for electricity generation that it can efficiently provide insight
into the distribution of future generation mixes and electricity spot
prices. We intend to develop this work in several directions in the
future. First, we wish to take into account more generation technologies,
most notably wind farms, nuclear production, as well as solar distributed
production. These additions would raise the dimension of the problem
from eight to fifteen. Yet another range of innovations in numerical
methods will be necessary to overcome this increase in dimension.
Second, we wish to take time-to-build into account. And last but not
least, we wish to adapt the problem to a continuous-time multiplayer
game and contribute to the quest for an efficient algorithm to solve
it.

\appendix
\bibliographystyle{abbrv}
\phantomsection\addcontentsline{toc}{section}{\refname}\bibliography{BiblioPHD}

\end{document}